\documentclass[11pt,notitlepage]{amsart} 
\usepackage{amsmath}
\usepackage{amsthm}
\usepackage{amsfonts}
\usepackage{scalerel}
\usepackage{amssymb,amscd,epsf,verbatim,mathtools,framed}
\usepackage{mathrsfs}
\usepackage{graphicx} 
\usepackage{latexsym}
\usepackage{lscape}
\usepackage[colorlinks=true]{hyperref}
\hypersetup{colorlinks, citecolor=blue, filecolor=black, linkcolor=purple, urlcolor=violet}
\usepackage{epstopdf}
\usepackage{tikz}
\usepackage[all]{xy}
\usetikzlibrary{calc}
\usetikzlibrary{matrix,arrows,decorations.pathmorphing}
\usepackage{tikz-cd}
\usepackage{color}
\usepackage{multirow}  
\usepackage{scalefnt}
\usepackage{fancyhdr}
\usepackage[margin=1.2in]{geometry}
\usepackage[T1]{fontenc}

\usepackage{relsize} 
\usepackage[bbgreekl]{mathbbol} 
\DeclareSymbolFontAlphabet{\mathbb}{AMSb} 
\DeclareSymbolFontAlphabet{\mathbbl}{bbold}

\newcommand{\Z}{\mathbb{Z}}
\newcommand{\D}{\mathbb{D}}
\newcommand{\F}{\mathbb{F}}
\newcommand{\N}{\mathbb{N}}
\newcommand{\R}{\mathbb{R}}
\newcommand{\Q}{\mathbb{Q}}
\newcommand{\A}{\mathbb{A}}
\renewcommand{\L}{\mathbb{L}}
\renewcommand{\P}{\mathbb{P}}
\newcommand{\C}{\mathbb{C}}

\newcommand{\G}{\mathbb{G}}

\newcommand{\mC}{\mathcal{C}}

\newcommand{\mE}{\mathcal{E}}

\newcommand{\mI}{\mathcal{I}}

\newcommand{\mO}{\mathcal{O}}

\newcommand{\mS}{\mathcal{S}}

\newcommand{\mX}{\mathcal{X}}

\newcommand{\sH}{\mathscr{H}}

\newcommand{\fX}{\mathfrak{X}} 
\newcommand{\fY}{\mathfrak{Y}}

\newcommand{\fm}{\mathfrak{m}}

\newcommand{\tu}{\textup}
\newcommand{\cl}{\overline}
\newcommand{\ul}{\underline}
\newcommand{\so}{\Rightarrow} 

\renewcommand{\le}{\leqslant}
\renewcommand{\ge}{\geqslant}
\newcommand{\ra}{\rightarrow}

\newcommand{\sq}{\widetilde}
\newcommand{\minus}{\backslash}

\DeclareMathOperator{\GL}{GL}

\DeclareMathOperator{\lra}{\: \longrightarrow \:} 
\DeclareMathOperator{\isom}{\;\xrightarrow{\: {}_{\sim} \:} \;} 
\DeclareMathOperator{\Gal}{Gal}

\newcommand{\gr}[1]{\langle {#1} \rangle} 
\DeclareMathOperator{\spec}{Spec}
\DeclareMathOperator{\spf}{Spf}
\DeclareMathOperator{\spa}{Spa}

\DeclareMathOperator{\ett}{\textup{\'et}}
\DeclareMathOperator{\ket}{\textup{k\'et}}
\DeclareMathOperator{\proet}{\textup{pro\'et}}
\DeclareMathOperator{\proket}{\textup{prok\'et}}

\theoremstyle{theorem}

\newtheorem{theorem}{Theorem}[section]
\newtheorem{proposition}[theorem]{Proposition} 
\newtheorem{corollary}[theorem]{Corollary}

\newtheorem{lemma}[theorem]{Lemma}

\newtheorem{question}[theorem]{Question}

\newtheorem*{theorem*}{Theorem}
\newtheorem*{proposition*}{Proposition}

\theoremstyle{definition}
\newtheorem{example}[theorem]{Example}
\newtheorem{definition}[theorem]{Definition}
  \newtheorem{construction}[theorem]{Construction} 

\newtheorem{notation}[theorem]{Notation}
\newtheorem{remark}[theorem]{Remark}

\title{\large $\textup{Monodromy and rigidity of crystalline local systems}$}
\author{Hansheng Diao}
\email{hdiao@mail.tsinghua.edu.cn}
\address{Yau Mathematical Sciences Center, Tsinghua University, China.}
\author{Zijian Yao}
\email{yao@math.ucsb.edu}
\address{Department of Mathematics, University of California, Santa Barbara, CA 93106, USA.}

\numberwithin{equation}{section}
\begin{document}

\begin{abstract} 
We study several rigidity properties of $p$-adic local systems on a smooth  rigid analytic space  $X$ over a $p$-adic field. We prove that the monodromy of the log isocrystal attached to a  $p$-adic local system is ``rigid'' along irreducible components of the special fiber. Then we give several applications. First, suppose that $X$ has good reduction. We show that if a family of semistable representations is crystalline at one classical point on $X$, then it is crystalline everywhere. Second, combining with the $p$-adic monodromy theorem recently studied by the authors and their collaborators, we prove the following surprising rigidity result conjectured by Shankar: for any $p$-adic local system on a smooth projective variety with good reduction, if it is potentially crystalline at one classical point, then it is potentially crystalline everywhere. Finally, we show that if a $p$-adic local system on the complement of a reduced normal crossing divisor on a smooth  rigid analytic space is crystalline at all classical points, then it extends uniquely to a $p$-adic local system on the entire space. In other words, such a local system cannot have geometric monodromy if it has no arithmetic monodromy everywhere on the complement of a reduced normal crossing divisor.  
\end{abstract}

\maketitle
 \thispagestyle{empty}
\tableofcontents

\section{\large Introduction} 

\subsection*{\large Main results on rigidity of crystallinity} \indent 
\vspace*{0.2cm}

\noindent The goal of this paper is to investigate certain \textit{rigidity properties} for $p$-adic local systems. 
Throughout the article, let $K/\Q_p$ be a discretely-valued nonarchimedean field extension with perfect residue field $k$ and let $\mO_K$ be its ring of integers. We pick a uniformizer $\varpi \in K$ and normalize the absolute value on $K$ such that $|\varpi| = 1/p$. In \cite{Liu_Zhu}, Liu--Zhu prove the following rigidity result for de Rham local systems. 

\begin{theorem*}[Liu--Zhu] 
Let $X$ be a geometrically connected rigid analytic space over $K$ and let $\L$ be an \'etale $\Q_p$-local system on $X$. If there exists a classical point\footnote{Throughout the article, by a \emph{classical point} on a smooth rigid analytic space over $K$ (viewed as an adic space), we mean a point $\spa(K', \mO_{K'})\ra X$ for some finite extension $K'/K$. In particular, it has residue field $K'$. When $X$ is the rigid analytification of a smooth algebraic variety over $K$, a classical point is the same as a closed point on the underlying algebraic variety.} $x_0$ on $X$ such that the restriction $\L|_{x_0}$ is de Rham as a $p$-adic Galois representation,\footnote{More precisely, we consider the stalk $\L_{\bar{x}_0}$ for some geometric point $\bar{x}_0$ above $x_0$, which can be viewed as a $p$-adic Galois representation of the residue field of $x_0$. Abusing the notation, we write $\L|_{x_0}$ for the $p$-adic Galois representation $\L_{\bar{x}_0}$ throughout this article. When $\L|_{x_0}$ is de Rham (resp., crystalline, semistable), we also say that $\L$ is de Rham (resp., crystalline, semistable) at $x_0$.} then $\L$ is a de Rham local system in the sense of \cite{Scholze_p_adic_Hodge}. In particular, it is de Rham at all classical points on $X$. 
\end{theorem*}

As Liu--Zhu have already remarked in \cite{Liu_Zhu}, the naive analogues of rigidity of de Rham local systems do not hold for crystalline or semistable local systems.\footnote{The notions of \textit{de Rham}, \textit{semistable}, and \textit{crystalline} $p$-adic local systems are natural generalizations of Fontaine's \textit{de Rham}, \textit{semistable}, and \textit{crystalline} $p$-adic Galois representations. We refer the reader to \cite{Scholze_p_adic_Hodge} for the notion of de Rham local systems, and to \cite{Faltings_90, Faltings_almost, Andreatta_Iovita_crys, Andreatta_Iovita_st, Tsuji_notes, GY, DLMS2} for the notion of crystalline and semistable local systems under the presence of smooth or semistable integral models. For the purpose of this article, let us note that, to a semistable \'etale $\Z_p$-local system on the adic generic fiber $X$ of a semistable $p$-adic formal model $\fX$ defined over $\mO_K$, one can attach a log $F$-isocrystal on its special fiber $\fX_s$, where $\fX_s$ is equipped with the pullback log structure from the divisorial log structure $\mO_{\fX, \ett} \cap (\mO_{\fX, \ett}[\frac{1}{p}])^{\times}$ on $\fX$ (see Definition \ref{def:st_local_sys}). We refer the reader to \cite[\S 3.5 \& \S 3.6]{DLMS2} and the references therein for more details. Also note that, if $\fX$ turns out to be smooth, the log structure on its special fiber $\fX_s$ is simply the one associated to the pre-log structure $\N \ra \mO_{\fX_s, \ett}$ sending $1 \mapsto 0$ (also see \S \ref{sec:log_crystals}). } 
The central theme of this article concerns the following natural question: \textit{in what sense is crystallinity rigid?}  \\

To put our results into context, it is perhaps elucidative 
to first consider what happens in the $l$-adic setting for a prime $l$ different from $p$.

\begin{proposition*}[Oswal--Shankar--Zhu]
    Let $\mathcal X$ be a smooth connected proper scheme over $\mO_K$ and let $X$ be its generic fiber over $K$. Let $\mathbb M$ be an \'etale $\Z_l$-local system on $X$ of rank $d$ with $l \ne p$. Further suppose that $p$ is coprime to the cardinality of $\mathrm{GL}_d(\F_l)$. If $\mathbb M$ is unramified at a closed $K$-point $x_0 \in X(K)$ (i.e., $\mathbb M|_{x_0}$ is unramified as a $p$-adic Galois representation), then $\mathbb M$ is unramified at all classical points on $X$. 
\end{proposition*} 

This result is essentially \cite[Proposition 3.6]{OSZ}. For convenience of the reader, let us recall the salient point of the proof, which uses \'etale fundamental groups.

\begin{proof}    
By assumption, $\mathbb M$ gives rise to a monodromy representation $\rho: \pi_1 (X)^{(p)} \ra \GL_d (\Z_l)$, where ${(-)}^{(p)}$ denotes the maximal prime-to-$p$ quotient. We have a short exact sequence  
\[
1 \ra I_{K}^{(p)} \ra \pi_1(X)^{(p)} \ra \pi_1  (\mathcal X)^{(p)} \ra 1
\]
where $I_K$ is the inertia group of $K$. 
Since  $\mathbb M$ is unramified at $x_0$, $\rho$ factors through $\pi_1 (\mathcal X)^{(p)}$. Therefore,  $\mathbb M$ spreads to a local system on the integral model $\mathcal X$ and the claim follows. 
\end{proof}

In other words, when $X$ has good reduction, the property of being unramified is quite ``rigid'' along families of $l$-adic representations. From this perspective, the question on rigidity of crystallinity for $p$-adic local systems seeks $p$-adic analogues of this phenomenon. Our first result towards this direction is the following.

\begin{theorem} \label{theorem:main_intro}
Let $\fX$ be a smooth connected $p$-adic formal scheme over $\mO_K$. Let $X$ be the rigid analytic generic fiber of $\fX$, viewed as an adic space over $\mathrm{Spa}(K, \mO_K)$. Let $\L$ be a \textit{semistable} \'etale $\Z_p$-local system on $X$. If there exists a classical point $x_0\in X(K)$\footnote{By $x_0\in X(K)$, we mean that the classical point $x_0$ is $K$-rational, namely, it has residue field precisely $K$.} such that $\L$ is crystalline at $x_0$,  then $\L$ is a crystalline local system. In particular, it is crystalline at all classical points on $X$. 
\end{theorem}

In other words, rigidity of crystallinity is satisfied for semistable local systems on an rigid analytic space with \textit{good reduction}. This rigidity property in turn has the following geometric incarnation (see \cite[Theorem 3.4]{OSZ} for a comparison).   

\begin{corollary} \label{cor:good_reduction_everywhere}
Let $X$ be a smooth connected rigid analytic space over $K$ with good reduction and let $\pi: A \ra X$ be a family of abelian varieties over $X$. Suppose that the fibers of $\pi$ have semistable reduction at all classical points on $X$. If there is a classical point $x_0 \in X(K)$ such that the fiber $A_{x_0}$ of $\pi$ has good reduction, then $A$ has good reduction over every classical point on $X$.  
\end{corollary}

\begin{remark} \label{remark:example_of_good_reduction} 
It may be helpful to remark that, examples of $X$ that satisfies the condition in Theorem \ref{theorem:main_intro} include smooth proper varieties with good reduction in the usual sense, the $p$-adic unit disc $\D = \{|z| \le 1\}$,  the ``thin annulus'' $\{|z| = 1\} \subset \D$,  
etc. It excludes the punctured disc $\D^\times = \{ 0 < |z| \le 1\}$, or ``thick annuli'' of the form $\{r_1\le |z| \le r_2\}$ where $r_1, r_2 \in p^{\Q}$ and $0 < r_1 < r_2$, or algebraic varieties such as $\A^1$ and $\G_m$. 
\end{remark}

More generally, for rigid analytic spaces with \textit{semistable reduction} instead of good reduction, we have a similar rigidity result, except that now one needs to test crystallinity at multiple classical points, one for each irreducible component of the special fiber. For illustration purposes, let us provide the simplest example of what we can prove, and refer the reader to Theorem \ref{cor:rigidity_crystalline_over_semistable} for the precise statement in the general case. For the setup, consider the closed unit disc 
\[\D=\D_K=\spa(K\langle z\rangle, \mO_K\langle z\rangle)=\{|z|\le 1\}\]
over $K$ and consider the ``thick annulus''   
\[ A_1 = \{ 1/p \le |z| \le 1 \}\subset \D.\] 
Let $\mathbb B_0 = \{|z| = 1\}$ and $\mathbb B_1 = \{|z| = 1/p\}$ be its outer and inner boundaries, which are both what we refer to as ``thin annuli''. Let $U_1 = \{1/p < |z| < 1\}$ denote the open annulus inside $A_1$.  Note that both $\mathbb B_0$ and $\mathbb B_1$ have good reduction in our sense (with special fiber being a copy of $\G_m$ over $k$). On the other hand, $A_1$ has semistable reduction with a standard semistable formal model  
\[ 
\spf \mO_K \gr{x, y}/(xy - \varpi).
\] 
Figure \ref{fig:model_for_A1} below provides an illustration of $A_1$ and the special fiber of its semistable model, which consists of two copies of $\A^1_k$ intersecting at a point. Under this setup, we prove

\begin{theorem} \label{thm:main_intro_for_log_schemes}
Let $\L$ be a \textit{semistable} \'etale $\Z_p$-local system on $A_1$. Then the following are equivalent. 
\begin{enumerate}
    \item $\L$ is crystalline at one classical point $x_0 \in \mathbb B_0(K)$ \textit{and} at one classical point $x_1 \in \mathbb B_1(K)$.  
    \item $\L$ is crystalline at all classical points on $U_1$. 
    \item $\L$ is crystalline at all classical points on $A_1$. 
    \item $\L$ is a crystalline local system\footnote{This is not a standard terminology in most of the existing literature because $A_1$ does not have good reduction.} in the sense of Definition \ref{def:st_local_sys}.
\end{enumerate}
\end{theorem}

For both of the theorems above, we start with a \textit{semistable} local system and prove rigidity of crystallinity in a suitable sense. Perhaps more surprisingly, we have the following rigidity result for general \'etale $\Z_p$-local systems. This was first conjectured by Ananth Shankar.\footnote{This was communicated to us in a private conversation. In fact, we expect this conjecture to hold for all geometrically connected smooth rigid analytic spaces with good reduction.} 

\begin{theorem}[Shankar's conjecture]\label{thm:conjecture_for_projective_varieties_intro}
Let $\mX$ be a geometrically connected smooth projective scheme over $\mO_K$ and let $X$ be its generic fiber over $K$. Let $\L$ be an \'etale $\Z_p$-local system on $X$. If there exists a classical point $x_0$ on $X$ such that $\L|_{x_0}$ is potentially crystalline, then $\L$ is potentially crystalline at all classical points on $X$. 
\end{theorem}

Let us remark that, the proof of this result combines the methods and techniques developed in this article with the \textit{$p$-adic monodromy theorem} for curves, where the latter is established in the upcoming work of the authors and their collaborators \cite{DDMY}. We shall return to the detail of this discussion later in the introduction. \\
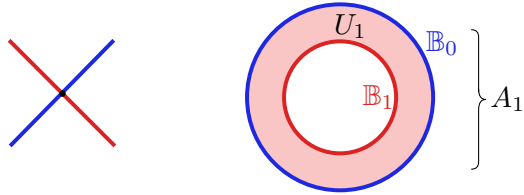
\begin{figure}[h]
\begin{tikzpicture}[x=0.4pt,y=0.4pt,yscale=-1,xscale=1]
\draw  [color={rgb, 255:red, 30; green, 40; blue, 225 }  ,draw opacity=1 ][fill={rgb, 255:red, 247; green, 151; blue, 151 }  ,fill opacity=0.56 ][line width=1.5]  (313,154) .. controls (313,105.4) and (352.4,66) .. (401,66) .. controls (449.6,66) and (489,105.4) .. (489,154) .. controls (489,202.6) and (449.6,242) .. (401,242) .. controls (352.4,242) and (313,202.6) .. (313,154) -- cycle ;
\draw  [color={rgb, 255:red, 220; green, 35; blue, 35 }  ,draw opacity=1 ][fill={rgb, 255:red, 255; green, 255; blue, 255 }  ,fill opacity=1 ][line width=1.5]  (348,154) .. controls (348,124.73) and (371.73,101) .. (401,101) .. controls (430.27,101) and (454,124.73) .. (454,154) .. controls (454,183.27) and (430.27,207) .. (401,207) .. controls (371.73,207) and (348,183.27) .. (348,154) -- cycle ;
\draw   (525,222) .. controls (529.67,222) and (532,219.67) .. (532,215) -- (532,166) .. controls (532,159.33) and (534.33,156) .. (539,156) .. controls (534.33,156) and (532,152.67) .. (532,146)(532,149) -- (532,97) .. controls (532,92.33) and (529.67,90) .. (525,90) ;
\draw [color={rgb, 255:red, 220; green, 35; blue, 35 }  ,draw opacity=1 ][line width=1.5]    (88,100) -- (188,200) ;
\draw [color={rgb, 255:red, 30; green, 40; blue, 225 }  ,draw opacity=1 ][line width=1.5]    (89,200) -- (187,100) ;
\draw  [fill={rgb, 255:red, 0; green, 0; blue, 0 }  ,fill opacity=1 ] (141,150.5) .. controls (141,149.12) and (139.88,148) .. (138.5,148) .. controls (137.12,148) and (136,149.12) .. (136,150.5) .. controls (136,151.88) and (137.12,153) .. (138.5,153) .. controls (139.88,153) and (141,151.88) .. (141,150.5) -- cycle ;

\draw (420,140) node [anchor=north west][inner sep=0.75pt]  [color={rgb, 255:red, 220; green, 35; blue, 35 }  ,opacity=1 ] [align=left] {$\displaystyle \mathbb{B}_{1}$};
\draw (480,89) node [anchor=north west][inner sep=0.75pt]  [color={rgb, 255:red, 30; green, 40; blue, 225 }  ,opacity=1 ] [align=left] {$\displaystyle \mathbb{B}_{0}$};
\draw (393,73) node [anchor=north west][inner sep=0.75pt]   [align=left] {$\displaystyle U_{1}$};
\draw (541,142) node [anchor=north west][inner sep=0.75pt]   [align=left] {$\displaystyle A_{1}$};
\end{tikzpicture} 
    \caption{$A_1$ with the special fiber of its standard semistable model}
    \label{fig:model_for_A1}
\end{figure} 

From a slightly different perspective, the results above all indicate that, within a $p$-adic local system, the ``arithmetic monodromy'' arising from (the filtered $(\varphi, N)$-modules attached to) the $p$-adic Galois representations satisfies suitable rigidity properties.\footnote{More precisely, according to Theorem \ref{theorem:main_intro}, if $X$ has good reduction and $\L$ has no arithmetic monodromy at one classical point on $X$, then there is no arithmetic monodromy anywhere. When the rigid analytic space in question has semistable reduction, one needs to test arithmetic monodromy at multiple points, as in Theorem \ref{thm:main_intro_for_log_schemes}.} 
On the other hand,  we may ask whether one can infer any information about the ``geometric monodromy'' of a $p$-adic local system, which is of global nature, from information of the purely local notion of arithmetic monodromy. To this end, we prove the following result.

\begin{theorem} \label{mainthm:extend_across_ncd}
Let $Y$ be a smooth rigid analytic space over $K$. Let $D \subset Y$ be a reduced normal crossing divisor \footnote{See, for example, \cite[Example 2.3.17]{DLLZ1} for the notion of reduced normal crossing divisor.} and let $U = Y -  D$ be the complement of $D$. Let $\L$ be an \'etale $\Z_p$-local system on $U$. 
\begin{enumerate}
    \item 
Suppose that $\L$ is crystalline at all classical points on $U$, then $\L$ extends (necessarily uniquely) to an \'etale $\Z_p$-local system on $Y$. 
\item Moreover, if $Y$ is quasi-compact, then, up to replacing $K$ by a finite extension, $\L$ is crystalline at all classical points on $Y$.
\end{enumerate}
\end{theorem}

In other words, a $p$-adic local system cannot have geometric monodromy along a normal crossing divisor if it does not have arithmetic monodromy everywhere on the complement. 

\begin{remark} \label{remark:not_really_counterexamples} 
The readers may wonder why our assertions are compatible with examples coming from universal elliptic curves over modular curves. 
To this end, let us first note that, over a modular curve $X_\Gamma$ with good reduction at $p$,\footnote{For example, consider $X_0(N)$ where $p \nmid N$.} 
we only have generalized elliptic curves over the cusps. In particular, the pushforward $R^1 \pi_*  \ul{\Z_p}$ of the constant local system $\ul{\Z_p}$ on the universal generalized elliptic curve $\mE_{\Gamma}$ along the structure map $\pi: \mE_\Gamma \ra X_\Gamma$ only gives rise to a log (Kummer) \'etale local system, so it does not satisfy the assumption of Theorem \ref{thm:conjecture_for_projective_varieties_intro} in our setup. If we remove the cusps, then $Y_\Gamma $ 
is no longer the adic generic fiber of a smooth $p$-adic formal scheme, so it does not have good reduction in our sense as an adic space --- rather it has semistable reduction. 
On the other hand, Theorem \ref{mainthm:extend_across_ncd} requires a $\Z_p$-local system on  the open modular curve $Y_{\Gamma}$  to be crystalline at all classical points to extend across the cusps, which is not satisfied by the restriction of $R^1 \pi_* \ul{\Z_p}$ to $Y_{\Gamma}$.  In fact, from Theorem \ref{thm:conjecture_for_projective_varieties_intro} we can deduce the classical result that the Kummer \'etale local system $R^1 \pi_* \ul{\Z_p}$ never descends to an \'etale local system on $X_\Gamma$. 
\end{remark}

\begin{remark}
Theorem \ref{mainthm:extend_across_ncd} is closely related to (and is a $p$-adic local system analogue of) the Borel extension type results recently studied by Oswal--Shankar--Zhu (see \cite[Theorem 1.1, Corollary 1.7]{OSZ}). We expect that the results developed in this article will simplify certain steps in \textit{loc.cit.}, and should allow one to generalize $p$-adic hyperbolicity (which is one of the main results of \textit{loc.cit.}) to a more general class of $p$-adic symmetric domains and period domains.  
\end{remark}

In the rest of this introduction, we will elaborate on the content and the proof of these results, and discuss the relevant new techniques introduced in this article.

\vspace*{0.1cm}
\subsection*{\large Main results on logarithmic (iso)crystals}
 \noindent 
\vspace*{0.2cm}
 
\noindent  
Let us start with Theorem \ref{theorem:main_intro}, which will follow from a somewhat surprising statement about logarithmic (iso)crystals over smooth varieties in characteristic $p$. 

For the setup, let $k$ be a perfect field of characteristic $p$ and let $Z$ be a geometrically connected smooth variety over $k$. Let $Z^{\log}$ be the fs (fine and saturated) log scheme associated to the pre-log scheme $(Z, \alpha)$, where $\alpha$ is the pre-log structure $\N \ra \mO_{Z_{\ett}}$ sending $1 \mapsto 0$. If $z: \spec k' \ra Z$ is a closed point on $Z$, where $k'/k$ is an extension of perfect fields, we denote by $z^{\log}$ the fiber product 
\[z^{\log} := z \times_Z Z^{\log}\] 
in the category of fs log schemes. Abstractly, $z^{\log}$ is isomorphic to the log point associated to the pre-log algebra $(k', \N \xrightarrow{1 \mapsto 0} k')$.

Let $\mE$ be a log (iso)crystal over $Z^{\log}$. For every closed point $z: \spec k' \ra Z$ on $Z$, $\mE$ restricts to a log (iso)crystal $\mE_z$ over the log point $z^{\log}$. The classical theory of Hyodo--Kato \cite{Hyodo_Kato} tells us that, the data of the log crystal (resp. log isocrystal) $\mE_z$ over $z^{\log}$ gives rise to a $(\varphi, N)$-module over $W(k')$ (resp. over $W(k')[1/p]$), and in particular gives rise to a monodromy operator $N$. We call $N$ the \emph{monodromy} of $\mE$ at $z$. The following result is one of the main technical results of this paper.  

\begin{theorem} \label{theorem:log_crystal}
\begin{enumerate}
\item If a log (iso)crystal $\mE$ over $Z^{\log}$ has trivial monodromy at one closed point on $Z$, then $\mE$ descends to an (iso)crystal over $Z$. In particular, the log (iso)crystal $\mE$ has trivial monodromy everywhere on $Z$. 
\item More generally, for a log (iso)crystal $\mE$ over $Z^{\log}$, the nilpotent rank of the monodromy operator of $\mE$ is constant over closed points on $Z$. 
\end{enumerate}
\end{theorem}

In the theorem above, the underlying scheme of $Z^{\log}$ is smooth. 
One step further, we are able to generalize the results to log schemes $Z^{\log}$ of ``semistable type''. For simplicity, we demonstrate our result in the following toy example. We refer the reader to \S \ref{sec:log_crystals_on_log_schemes} (Theorem \ref{theorem:log_crystal_2_global}) for the rigidity results for general log schemes of semistable type. 

Let $D = \spec k[x, y]/xy$ and let $D^{\log}$ be the log scheme associated to the pre-log ring\footnote{See the illustration on the left in Figure \ref{fig:model_for_A1}.} 
\[ (k[x, y]/xy,\,\, x^\N \oplus y^\N).\] 
Notice that $D$ has a unique singluar point $z_*$ given by $x = y = 0$. Let $D^{\mathrm{sm}} = D-\{z_*\}$ denote the smooth locus. By construction \[D^{\mathrm{sm}} = U_0 \sqcup U_1 \] is a disjoint union of two irreducible components, each being a copy of $\G_m$. Let us note that, for $i = 0, 1$, the log structure on $U_i$ restricted from $D^{\log}$ is simply the log structure associated with the pre-log structure $\N\ra \mO_{U_i, \ett}$ sending $1\mapsto 0$. In particular, given a log (iso)crystal $\mE$ over $D^{\log}$ and a closed point $z$ on either $U_0$ or $U_1$, we can still make sense of the notion of monodromy of $\mE$ at $z$. Under this setup, let us present an example of the rigidity results on log (iso)crystals that we can prove. 

\begin{theorem} \label{theorem:log_crystals_intro_st}
Let $D^{\log}$ be the log scheme associated with the pre-log ring $(k[x, y]/xy,\,\, x^\N \oplus y^\N)$ as above and let $\mE$ be a log (iso)crystal over $D^{\log}$. If $\mE$ has trivial monodromy at one closed point on $U_0$ and one closed point on $U_1$, then $\mE$ descends to an (iso)crystal over $D$. In particular, $\mE$ has trivial monodromy everywhere on $D$.\footnote{At the singular point $z_*$, the log (iso)crystal $\mE$ actually has ``trivial monodromy in all directions'' in the sense of Definition \ref{def:trivial_monodromy_in_all_directions}.} 
\end{theorem}

Theorem \ref{theorem:main_intro} (resp. Theorem \ref{thm:main_intro_for_log_schemes}) follows immediately from Theorem \ref{theorem:log_crystal} (resp. Theorem \ref{theorem:log_crystals_intro_st}). We shall discuss the proofs of these theorems in \S \ref{sec:log_crystals} and \S \ref{sec:log_crystals_on_log_schemes}.  

\begin{remark}
During the preparation of this article, we learned that for a log $F$-isocrystal $\mE$ over $D^{\log}$, the first claim in Theorem \ref{theorem:log_crystal} can also be deduced from the recent work \cite{GY} of Guo--Yang as follows. In \textit{loc.cit.}, they show that the category of log $F$-isocrystals over $D^{\log}$ is equivalent to the category of $F$-isocrystals over $D$ equipped with a ``relative monodromy operator'' satisfying certain conditions. Moreover, they show that, if the logarithmic $F$-isocrsytal $(\mE, \varphi)$ corresponds to $(\mE', \varphi', N_{\mE'})$ under this equivalence, where $(\mE', \varphi')$ is an $F$-isocrystal over $D$, then the sub-object $(\mE')^{N_{\mE'} = 0} = \ker (N_{\mE'}) $ is again an $F$-isocrystal over $D$. Thus, if $N_{\mE'} = 0$ at one classical point on $D$, then it is trivial at all classical points. Their method essentially proves a relative version of the Hyodo--Kato isomorphism, where the Frobenius operator plays an essential role in the argument and the descent only seems to work rationally (in other words, for isocrystals). In comparison, our approach studies the descent from the log crystalline site to the crystalline site directly, which works both integrally and rationally, and involves no Frobenius structure in the argument. The irrelevance of Frobenius in establishing such rigidity results is somewhat surprising to us. More importantly, our argument is robust enough to generalize to log smooth settings in \S \ref{sec:log_crystals_on_log_schemes} (for example, to prove Theorem \ref{theorem:log_crystals_intro_st}), which is crucial for the proof of Theorem \ref{thm:conjecture_for_projective_varieties_intro} and of Theorem \ref{mainthm:extend_across_ncd}. 
\end{remark}


\vspace*{0.1cm}
\subsection*{\large Shankar's rigidity conjecture and the $p$-adic monodromy theorem for curves} \noindent 
\vspace*{0.1cm}

\noindent  
 Let us now briefly explain some ingredients of the proof of Theorem \ref{thm:conjecture_for_projective_varieties_intro}, which uses a mixture of techniques, including the full strength of our study of rigidity properties of the monodromy operator attached to logarithmic isocrystals, the $p$-adic monodromy theorem for curves, as well as techniques from both algebraic and rigid analytic geometry.

Our first step is to reduce to the case of curves. This relies on the following assertion, which says that, given a finite collection of closed points on a smooth projective variety with good reduction, we can connect these points by a curve with good reduction. 

\begin{proposition} \label{mainprop:reducing_to_curves}
Let $X$ be a geometrically connected smooth projective variety over $K$ with good reduction and let $T = \{x_1, ..., x_m\}$ be a finite set of closed points on $X$ with distinct specializations to characteristic $p$. Then, up to replacing $K$ by a finite (even unramified) extension, there exists a smooth projective curve $C$ over $K$ with good reduction and a closed embedding $\alpha: C \ra X$ such that $T$ is contained in the image of $C$. 
\end{proposition}
 
This claim follows from standard arguments in algebraic geometry on deforming very ample line bundles from characteristic $p$ to characteristic $0$ (we refer readers to \S \ref{ss:reducing_to_curves} for more details). Once we are reduced to the case of curves, we are able to apply the $p$-adic monodromy theorem as alluded to previously.

\begin{theorem}[The $p$-adic monodromy theorem for curves \cite{DDMY}] \label{thm:p_adic_monodromy}
Let $X$ be a smooth projective curve over $K$ and let $\L$ be a  de Rham $\Z_p$-local system on $X$. Then $\L$ is potentially semistable in the following sense: after replacing $K$ by a finite extension if necessary, there exists a finite cover $f:Y\ra X$ between smooth projective curves over $K$,
where $Y$ has semistable reduction over $\mO_K$, such that $f^* \L$ is a semistable local system.    
\end{theorem} 



Now let $X$ be a geometrically connected smooth projective curve over $K$ with good reduction and let $\L$ be a $\Z_p$-local system on $X$ as in Theorem \ref{thm:conjecture_for_projective_varieties_intro}. 
The rough idea is to first pass to a finite cover $f:Y \ra X$ over which $\L$ becomes semistable, so that we are under a setup to apply rigidity results such as 
Theorem \ref{theorem:main_intro} and Theorem \ref{thm:main_intro_for_log_schemes}. One major difficulty we encounter in our argument is the following: after passing to the finite cover $Y$, one generally loses the property of having good reduction. Even worse, after choosing a suitable integral model $\widehat{f}:\fY\ra\fX$ of $f$ and consider the induced map $f_s: \fY_s \ra \fX_s$ on the special fibers, it might happen that some irreducible component in $\fY_s$ is entirely contracted to one closed point on $\fX$ (see Figure \ref{fig:f_splitting_into_3_cases} in \S \ref{sec:conjecture} for an illustration of this phenomenon). This makes it difficult to directly apply the techniques we develop in \S \ref{sec:log_crystals} and \S \ref{sec:log_crystals_on_log_schemes} on the rigidity of the monodromy operators of log ($F$-)isocrystals. 

To remedy this issue, we apply some techniques from rigid analytic geometry and perform a slight ``$p$-adic perturbation'' to $Y$  to remove the bad ``contracting locus''  near a given point, at the cost of possibly replacing $Y$ by a further finite cover, and possibly introducing additional contracting loci at some other irrelevant points. We call such a procedure \textit{dodging the contracting locus of $f$}, which is achieved via a careful study of local reduction behavior of maps between curves (not necessarily with good reduction). This eventually reduces to some explicit computations using coordinates on $\P^1$. For such arguments to work, it is also crucial that we enter the realm of rigid analytic geometry in the sense of Huber's adic spaces. 

Let us give an example of how we ``dodge'' the contracting locus for curves over $\P^1$, which is in fact a key step towards proving the general case. For the setup, let $Y$ be a smooth projective curve over $K$, viewed as an adic space, and let $g: Y \ra \P^1_K$ be a finite cover. Let $\xi_0 \in \P^1_K$ be a non-type I point\footnote{For a quick review on points of type I, I\!I, I\!I\!I, I\!V, V on a smooth curve, see \S \ref{ss:points_on_curves}.} and let $s_0 \in \P^1_K$ be a classical point. We now state the dodging theorem in its crudest form (see Theorem \ref{theorem:modified_spreadout_for_dodging} for the precise statement and see Figure \ref{fig:dodging} for a cartoon that illustrates this procedure). 

\begin{theorem} \label{mainthm:modified_spreadout_for_dodging_intro}
Up to replacing $K$ by a finite extension, there exists an open neighborhood $U\subset \P^1_K$ of $\xi_0$ and another finite cover $g': Y'\ra \P^1_K$ between smooth projective curves such that 
\begin{enumerate}
\item $Y'$ has semistable reduction over $\mO_K$;
\item $g$ and $g'$ restrict to the ``same'' map over $U$;
\item there exists a classical point $y_0' \in (g')^{-1} (s_0)$ that lies on the non-contracting locus of $Y'$.
\end{enumerate}
\end{theorem}

In other words, we can ``dodge'' the contracting locus for curves over $\P^1$ after a ``modification'' of the curve away from some neighborhood of a fixed non-type I point in a suitable sense. 
We refer the reader to \S \ref{sec:dodging} for the precise statement and proof of this theorem along with its variants, as well as the definitions of the contracting and non-contracting loci. The results in \S \ref{sec:dodging} will eventually allow us to complete the proof of Shankar's rigidity conjecture (Theorem \ref{thm:conjecture_for_projective_varieties_intro}), which we explain in \S \ref{sec:conjecture}.

\begin{remark}
One might ask whether a slightly stronger version of Shankar's conjecture is true. Namely, under the same assumption as in Theorem \ref{thm:conjecture_for_projective_varieties_intro}, is it true that $\L$ always becomes a crystalline local system on $X$ by only replacing $K$ by a finite extension. This stronger version of the conjecture turns out to be false,  and one could produce counterexamples in the same spirit of \cite{Lawrence_Li} (see Remark \ref{remark:counter_example_to_optimistic_guess}). In other words, it is crucial that we allow \textit{potential} crystallinity in the statement. 
\end{remark}

Let us also take this opportunity to record the following question, which can be viewed as a geometric analogue of Shankar's conjecture.  

\begin{question} Let $X$ be a geometrically connected smooth projective variety over $K$ with good reduction. If $\pi: Y \ra X$ is a smooth proper family over $X$ such that $Y_{x_0}$ has potential good reduction at one closed point $x_0$ on $X$, then does $Y$ necessarily have potential good reduction everywhere over $X$?
\end{question} 

\vspace*{0.1cm}
\subsection*{\large Extension of local systems across normal crossing divisors} \noindent 
\vspace*{0.2cm}
 
\noindent  
Finally, let us turn to Theorem \ref{mainthm:extend_across_ncd}, which we will deduce from the following special case of a punctured disc. 

\begin{theorem} \label{theorem:main_intro_punctured_disc} 
Let $\D=\D_K=\{|z|\le 1\}$ be the closed unit disc over $K$ and let $\D^\times=\D-\{0\}$ be the punctured closed unit disc. Let $\L$ be an \'etale $\Z_p$-local system on $\mathbb D^\times$. Assume that $\mathbb{L}|_x$ is crystalline at all classical points $x$ in $\mathbb D^\times$. Then $\mathbb{L}$ extends uniquely to an \'etale $\Z_p$-local system on $\mathbb{D}$ and is necessarily a crystalline local system.
\end{theorem}

In the remainder of this introduction, let us outline the proof of Theorem \ref{theorem:main_intro_punctured_disc}. Let us write $\D_m$ for the closed sub-annulus 
\[ \D_m := \{1/p^m \le  |z| \le 1 \}
\]inside the closed unit disc, which is affinoid and corresponds to the Tate algebra 
\[ K\gr{\frac{\varpi^m}{z}, z} =K\gr{z, z'}/(z z' - \varpi^m).\] Note that $\D_m$ has a standard semistable formal model $\mathcal D_m$ over $\spf \mO_K$ obtained as successive admissible blowing-ups of $\spf \mO_K \gr{z, z'}/(zz' - \varpi^m)$ at the singular point. The special fiber $\mathcal D_{m,s}$ of $\mathcal D_m$ consists of a ``tree'' of smooth curves over $k$, such that each irreducible component of $\mathcal D_{m,s}$ is isomorphic to either $\P^1_k$ or $\A^1_k$. Equivalently, it can be obtained from gluing the standard semistable model 
\[\mathcal A_j = \spf \mO_K \gr{x_{j-1}, y_{j-1}}/(x_{j-1} y_{j-1} - \varpi)\] of the annulus $A_j = \{1/p^j \le  |z| \le 1/p^{j-1}\}$ to $\mathcal A_{j+1}$ along the $p$-adic formal torus, identifying $y_{j-1}$ with $x_j^{-1}$ for each $j = 1, ..., m-1$. As Figure \ref{fig:model_for_punctured_disc} below illustrates, the special fiber $\mathcal D_{m, s}$ of $\mathcal D_m$ has $m+1$ irreducible components, which consists of $m-1$ copies of $\P^1_{k}$ and two copies of $\A^1_{k}$ ``at the two ends''. 
\begin{figure}[h]
    \tikzset{every picture/.style={line width=0.75pt}} 
\scalebox{0.7}{ 
\tikzset{every picture/.style={line width=0.75pt}} 

\begin{tikzpicture}[x=0.75pt,y=0.75pt,yscale=-1,xscale=1]

\draw  [color={rgb, 255:red, 50; green, 70; blue, 225 }  ,draw opacity=1 ][fill={rgb, 255:red, 255; green, 207; blue, 207 }  ,fill opacity=0.72 ][line width=1.5]  (357,142) .. controls (357,81.8) and (405.8,33) .. (466,33) .. controls (526.2,33) and (575,81.8) .. (575,142) .. controls (575,202.2) and (526.2,251) .. (466,251) .. controls (405.8,251) and (357,202.2) .. (357,142) -- cycle ;
\draw  [color={rgb, 255:red, 0; green, 0; blue, 0 }  ,draw opacity=1 ][line width=1.5]  (385,142) .. controls (385,97.26) and (421.26,61) .. (466,61) .. controls (510.74,61) and (547,97.26) .. (547,142) .. controls (547,186.74) and (510.74,223) .. (466,223) .. controls (421.26,223) and (385,186.74) .. (385,142) -- cycle ;
\draw  [color={rgb, 255:red, 0; green, 0; blue, 0 }  ,draw opacity=1 ][line width=1.5]  (406,142) .. controls (406,108.86) and (432.86,82) .. (466,82) .. controls (499.14,82) and (526,108.86) .. (526,142) .. controls (526,175.14) and (499.14,202) .. (466,202) .. controls (432.86,202) and (406,175.14) .. (406,142) -- cycle ;
\draw  [color={rgb, 255:red, 0; green, 0; blue, 0 }  ,draw opacity=1 ][line width=1.5]  (424.25,142) .. controls (424.25,118.94) and (442.94,100.25) .. (466,100.25) .. controls (489.06,100.25) and (507.75,118.94) .. (507.75,142) .. controls (507.75,165.06) and (489.06,183.75) .. (466,183.75) .. controls (442.94,183.75) and (424.25,165.06) .. (424.25,142) -- cycle ;
\draw  [color={rgb, 255:red, 225; green, 35; blue, 35 }  ,draw opacity=1 ][fill={rgb, 255:red, 255; green, 255; blue, 255 }  ,fill opacity=1 ][line width=1.5]  (437.63,142) .. controls (437.63,126.33) and (450.33,113.63) .. (466,113.63) .. controls (481.67,113.63) and (494.38,126.33) .. (494.38,142) .. controls (494.38,157.67) and (481.67,170.38) .. (466,170.38) .. controls (450.33,170.38) and (437.63,157.67) .. (437.63,142) -- cycle ;
\draw [color={rgb, 255:red, 225; green, 35; blue, 35 }  ,draw opacity=1 ][line width=1.5]    (23,107.13) -- (92.18,184.75) ;
\draw [color={rgb, 255:red, 0; green, 0; blue, 0 }  ,draw opacity=1 ][line width=1.5]    (50.83,186.87) -- (126.37,105) ;
\draw [color={rgb, 255:red, 0; green, 0; blue, 0 }  ,draw opacity=1 ][line width=1.5]    (185.63,109.25) -- (254.81,186.87) ;
\draw [color={rgb, 255:red, 50; green, 70; blue, 225 }  ,draw opacity=1 ][line width=1.5]    (213.46,189) -- (289,107.13) ;
\draw    (480,152) -- (484.97,160.29) ;
\draw [shift={(486,162)}, rotate = 239.04] [color={rgb, 255:red, 0; green, 0; blue, 0 }  ][line width=0.75]    (4.37,-1.32) .. controls (2.78,-0.56) and (1.32,-0.12) .. (0,0) .. controls (1.32,0.12) and (2.78,0.56) .. (4.37,1.32)   ;
\draw    (583,178) -- (546.91,166.6) ;
\draw [shift={(545,166)}, rotate = 17.53] [color={rgb, 255:red, 0; green, 0; blue, 0 }  ][line width=0.75]    (4.37,-1.32) .. controls (2.78,-0.56) and (1.32,-0.12) .. (0,0) .. controls (1.32,0.12) and (2.78,0.56) .. (4.37,1.32)   ;
\draw    (571,210) -- (563.14,198.64) ;
\draw [shift={(562,197)}, rotate = 55.3] [color={rgb, 255:red, 0; green, 0; blue, 0 }  ][line width=0.75]    (4.37,-1.32) .. controls (2.78,-0.56) and (1.32,-0.12) .. (0,0) .. controls (1.32,0.12) and (2.78,0.56) .. (4.37,1.32)   ;
\draw [line width=1.5]    (164,149.66) -- (193.81,184.75) ;
\draw [color={rgb, 255:red, 0; green, 0; blue, 0 }  ,draw opacity=1 ][line width=1.5]    (152.46,186.87) -- (228,105) ;
\draw [color={rgb, 255:red, 0; green, 0; blue, 0 }  ,draw opacity=1 ][line width=1.5]    (87.63,106.06) -- (124,146.47) ;
\draw  [fill={rgb, 255:red, 0; green, 0; blue, 0 }  ,fill opacity=1 ][line width=1.5]  (72,162.5) .. controls (72,162.22) and (72.22,162) .. (72.5,162) .. controls (72.78,162) and (73,162.22) .. (73,162.5) .. controls (73,162.78) and (72.78,163) .. (72.5,163) .. controls (72.22,163) and (72,162.78) .. (72,162.5) -- cycle ;
\draw  [color={rgb, 255:red, 0; green, 0; blue, 0 }  ,draw opacity=1 ][fill={rgb, 255:red, 0; green, 0; blue, 0 }  ,fill opacity=1 ] (204,130.5) .. controls (204,130.22) and (204.22,130) .. (204.5,130) .. controls (204.78,130) and (205,130.22) .. (205,130.5) .. controls (205,130.78) and (204.78,131) .. (204.5,131) .. controls (204.22,131) and (204,130.78) .. (204,130.5) -- cycle ;
\draw  [fill={rgb, 255:red, 0; green, 0; blue, 0 }  ,fill opacity=1 ][line width=1.5]  (235,165.5) .. controls (235,165.22) and (235.22,165) .. (235.5,165) .. controls (235.78,165) and (236,165.22) .. (236,165.5) .. controls (236,165.78) and (235.78,166) .. (235.5,166) .. controls (235.22,166) and (235,165.78) .. (235,165.5) -- cycle ;
\draw  [fill={rgb, 255:red, 0; green, 0; blue, 0 }  ,fill opacity=1 ][line width=1.5]  (174,163) .. controls (174,162.45) and (174.45,162) .. (175,162) .. controls (175.55,162) and (176,162.45) .. (176,163) .. controls (176,163.55) and (175.55,164) .. (175,164) .. controls (174.45,164) and (174,163.55) .. (174,163) -- cycle ;
\draw  [color={rgb, 255:red, 0; green, 0; blue, 0 }  ,draw opacity=1 ][fill={rgb, 255:red, 0; green, 0; blue, 0 }  ,fill opacity=1 ] (105.82,126.41) .. controls (105.82,126.08) and (106.08,125.82) .. (106.41,125.82) .. controls (106.74,125.82) and (107,126.08) .. (107,126.41) .. controls (107,126.74) and (106.74,127) .. (106.41,127) .. controls (106.08,127) and (105.82,126.74) .. (105.82,126.41) -- cycle ;
\draw    (282,213.5) .. controls (286.93,178.04) and (293.79,165.87) .. (273.93,136.84) ;
\draw [shift={(273,135.5)}, rotate = 55.01] [color={rgb, 255:red, 0; green, 0; blue, 0 }  ][line width=0.75]    (10.93,-3.29) .. controls (6.95,-1.4) and (3.31,-0.3) .. (0,0) .. controls (3.31,0.3) and (6.95,1.4) .. (10.93,3.29)   ;
\draw    (40,224.5) .. controls (30.15,214.65) and (17.39,170.84) .. (43.77,143.73) ;
\draw [shift={(45,142.5)}, rotate = 136.04] [color={rgb, 255:red, 0; green, 0; blue, 0 }  ][line width=0.75]    (10.93,-3.29) .. controls (6.95,-1.4) and (3.31,-0.3) .. (0,0) .. controls (3.31,0.3) and (6.95,1.4) .. (10.93,3.29)   ;
\draw  [color={rgb, 255:red, 0; green, 0; blue, 0 }  ,draw opacity=1 ] (240,94.5) .. controls (240,89.83) and (237.67,87.5) .. (233,87.5) -- (167.5,87.5) .. controls (160.83,87.5) and (157.5,85.17) .. (157.5,80.5) .. controls (157.5,85.17) and (154.17,87.5) .. (147.5,87.5)(150.5,87.5) -- (82,87.5) .. controls (77.33,87.5) and (75,89.83) .. (75,94.5) ;

\draw (358,135) node [anchor=north west][inner sep=0.75pt]   [align=left] {$\displaystyle A_{1}$};
\draw (384,135) node [anchor=north west][inner sep=0.75pt]  [font=\small] [align=left] {$\displaystyle {\textstyle A_{2}}$};
\draw (404,134) node [anchor=north west][inner sep=0.75pt]  [font=\footnotesize] [align=left] {$\displaystyle \dotsc $};
\draw (421,136) node [anchor=north west][inner sep=0.75pt]  [font=\scriptsize] [align=left] {$\displaystyle {\textstyle A_{m}}$};
\draw (456,265) node [anchor=north west][inner sep=0.75pt]    {$\mathbb{D}_{m}$};
\draw (568,212) node [anchor=north west][inner sep=0.75pt]  [color={rgb, 255:red, 50; green, 70; blue, 225 }  ,opacity=1 ]  {$\mathbb{B}_{0}$};
\draw (582,171) node [anchor=north west][inner sep=0.75pt]  [font=\small,color={rgb, 255:red, 0; green, 0; blue, 0 }  ,opacity=1 ]  {$\mathbb{B}_{1}$};
\draw (467,135) node [anchor=north west][inner sep=0.75pt]  [font=\scriptsize,color={rgb, 255:red, 225; green, 35; blue, 35 }  ,opacity=1 ]  {$\mathbb{B}_{m}$};
\draw (121,136.47) node [anchor=north west][inner sep=0.75pt]  [font=\footnotesize] [align=left] {$\displaystyle \dotsc $};
\draw (141,136.47) node [anchor=north west][inner sep=0.75pt]  [font=\footnotesize] [align=left] {$\displaystyle \dotsc $};
\draw (263,214) node [anchor=north west][inner sep=0.75pt]  [color={rgb, 255:red, 50; green, 70; blue, 225 }  ,opacity=1 ]  {$\mathbb{A}_{k}^{1}$};
\draw (40,213.5) node [anchor=north west][inner sep=0.75pt]  [color={rgb, 255:red, 225; green, 35; blue, 35 }  ,opacity=1 ]  {$\mathbb{A}_{k}^{1}$};
\draw (194,52.5) node [anchor=north west][inner sep=0.75pt]    {$\mathbb{P}_{k}^{1}$};
\draw (93,55) node [anchor=north west][inner sep=0.75pt]  [font=\footnotesize] [align=left] {$\displaystyle ( m-1) \ copies\ of$};
\draw (149,265) node [anchor=north west][inner sep=0.75pt]    {$\mathcal D_{m,s}$};

\end{tikzpicture}
} 
\vspace*{-0.3cm}
    \caption{The special fiber and rigid analytic generic fiber of $\mathcal D_{m}$}
    \label{fig:model_for_punctured_disc}
\end{figure}
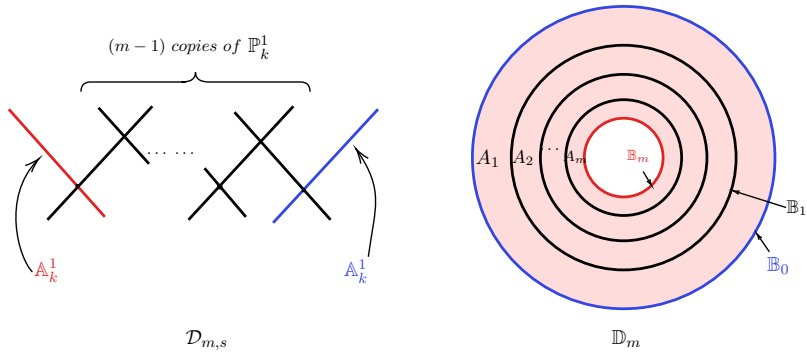

Let us write $\mathcal D_{m, s}^{\log}$ for the special fiber $\mathcal D_{m, s}$ equipped with the pullback of the divisorial log structure on $\mathcal D_m$ given by the divisor $\varpi = 0$ (namely, the divisorial log structure coming from the special fiber). 
Applying the pointwise criterion from \cite{GY}, we know that the local system $\L|_{\D_m}$ on $\D_m$ is a semistable $\Z_p$-local system. Thus it gives rise to a log $F$-isocrystal $\mE_m^{\log}$ over the log scheme $\mathcal D_{m, s}^{\log}$. Now, since $\L$ is crystalline everywhere on $\D_m$, by Theorem \ref{theorem:log_crystals_intro_st}, we know that the log $F$-isocrystal $\mE_m^{\log}$ descends to an $F$-isocrystal $\mE_m$ over the scheme $\mathcal D_{m, s}$. 

On the other hand, using the logarithmic $p$-adic Riemann-Hilbert functor developed by \cite{DLLZ2}, we obtain a vector bundle equipped with a logarithmic connection
\[ 
 (\mathbb{D}_{\mathrm{dR}, \log} (\L),  \nabla_{\L})
\] over the disc $\mathbb D$, with possibly log poles at the origin of the disc. By contemplating the relation between the restriction of the pair $(\mathbb{D}_{\mathrm{dR}, \log} (\L), \nabla_{\L})$ to the annulus $\D_{m}$ and the $F$-isocrystal $\mE_{m}$ over $\mathcal D_{m, s}$ (this is in spirit close to the comparison between log crystalline and log de Rham cohomology), we prove that $\nabla_{\L}$ admits a full set of solutions over each annulus $\D_m$, and consequently admits a full set of solutions over the punctured disc $\mathbb D^\times$. This in turn implies that $\L$ in fact extends to an \'etale local system on the entire disc $\mathbb D$, via a slightly upgraded version of the logarithmic $p$-adic Riemann--Hilbert correspondence (Theorem \ref{prop:residue vs monodromy}). Finally, once we have extended the local system to $\mathbb D$, it is not difficult to see that it is indeed a crystalline local system. We explain this in the end of \S \ref{section:punctured disc}.

\subsection*{Conventions} We largely follow \cite{DLLZ1} and \cite{logprism} for the  convention on log geometry, including the notion of $p$-adic formal schemes, log adic spaces, and Kummer \'etale local systems on log adic spaces, etc. Throughout the article, we use \emph{log crystals} (resp. \emph{log isocrystals}) to denote \emph{finite locally free} log crystals (resp. log isocrystals) on a log scheme, in the sense of \cite[Definition B.18]{DLMS2}. See Definition \ref{def:isocrystals} for details. 


\subsection*{Acknowledgement} 
This article originally evolved from several conversations between the authors and Ananth Shankar. We are very grateful to him for many inspiring conversations. 
It is also pleasure to thank Sasha Beilinson, Dori Bejleri, Brian Conrad,  Matt Emerton, Johan de Jong, Mark Kisin, Kevin Lin, Madhav Nori, Lue Pan, Lie Qian, Zhiyu Tian, and Junyi Xie for discussions related to the content of this paper. During the preparation of the article, the first author was partially supported by the National Key R{\&}D Program of China No. 2023YFA1009703 and No. 2021YFA1000704, and the National Natural Science Foundation of China No. 12422101. Part of this work is completed during the second author's visits to the Yau Mathematical Sciences Center at Tsinghua University in Beijing, he wishes to thank the institute for its great hospitality. Finally, the authors wish to thank the organizers of the $p$-adic geometry workshop in Shenzhen 2024 for their invitation to attend the workshop and for providing an excellent environment to finish part of the writeup of this article.

\newpage 
\section{\large Log (iso)crystals over smooth schemes} \label{sec:log_crystals}
  
The goal of this section is to prove Theorem \ref{theorem:log_crystal} from the introduction (restated as Theorem \ref{theorem:log_crystal_1}) on the rigidity properties of monodromy of log isocrystals over smooth schemes in characteristic $p$, from which we shall deduce Theorem \ref{theorem:main_intro}.  

\subsection{Log crystals and log isocrystals}
Let us first make precise the notion of log crystals and log isocrystals used in this article. Let $k$ be a perfect field of characteristic $p$ and let $Z^{\log} = (Z, M_Z)$ be an fs log scheme over $k$. Let $(Z^{\log})_{\mathrm{crys}}$ denote the big log crystalline site of $Z^{\log}$ over $W(k)$, defined as the colimit over $n$ of the big log crystalline site $(Z^{\log}/W_n(k))_{\mathrm{crys}}$ (see \cite[Tag 07I5]{SP} and \cite{Beilinson}), which comes equipped with a structure sheaf $\mO_{Z^{\log}, \tu{crys}}$.\footnote{Since $k$ is perfect, the site $(Z^{\log})_{\mathrm{crys}}$ is equivalent to the big absolute log crystalline site studied in \cite[1.12]{Beilinson} and \cite[Definition B.6]{DLMS2} and $\mO_{Z^{\log}, \tu{crys}}$ is denoted by  $\mO_{Z^{\log}/\Z_p}$ in both of these references. See the remarks in \cite[1.12]{Beilinson}. } 
Let $F_{\mathrm{crys}}^*$ denote the map on sheaves on  $(Z^{\log})_{\mathrm{crys}}$ induced by the absolute Frobenius map 
\[ F_{\mathrm{crys}}=(\mathrm{Frob}: Z \ra Z,\,\,\, p: M_Z \ra M_Z)\] on the log scheme $Z^{\log}$. We refer the reader to \cite{Beilinson} and \cite[Appendix B]{DLMS2} for more details on these notions. 
\begin{definition} \label{def:isocrystals}
\begin{enumerate}
    \item  A \textit{finite locally free log crystal} over $Z^{\log}$, which we shall simply refer to as a \textit{log crystal} in this article,  is a sheaf of $\mO_{Z^{\log}, \tu{crys}}$-modules $\mE$ such that for each log PD-thickening $(U, T)$ in  $(Z^{\log})_{\mathrm{crys}}$, the induced Zariski sheaf $\mE_{T}$ is a finite locally free $\mO_T$-module, such that for each morphism 
    \[ g: (U', T') \ra (U, T) \]  in the site $(Z^{\log})_{\mathrm{crys}}$, the induced map 
    \[ 
    g^* \mE_T \isom \mE_{T'}
    \] is an isomorphism. 
    \item  A \textit{log isocrystal} over $Z^{\log}$ is an object in the isogeny category of log crystals.\footnote{
    This definition is equivalent to the notion of \textit{finite locally free log isocrystals} in the sense of \cite[Definition B.18]{DLMS2}.}  
    \item  A \textit{log $F$-crystal}  (resp. \textit{log $F$-isocrystal}) over $Z^{\log}$ consists of a pair $(\mE, \varphi)$ where $\mE$ is a log crystal (resp. log isocrystal) over $Z^{\log}$ and $\varphi$ is the Frobenius map 
    \[ \varphi: F_{\mathrm{crys}}^* \mE \ra \mE \] which is compatible with $F_{\mathrm{crys}}$ and becomes an isomorphism of log isocrystals upon inverting $p$. 
\end{enumerate}
\end{definition}

Let $\mathrm{Vect} ((Z^{\log})_{\mathrm{crys}})$ 
(resp.  $\mathrm{Isoc} ((Z^{\log})_{\mathrm{crys}})$) denote the category of log crystals  (resp. log isocrystals)  over $Z^{\log}$, and let $\mathrm{Vect}^{\varphi}((Z^{\log})_{\mathrm{crys}})$  (resp. $\mathrm{Isoc}^{\varphi}((Z^{\log})_{\mathrm{crys}})$) denote the category of log $F$-crystals  (resp. log $F$-isocrystals) over $Z^{\log}$. 

\begin{remark} \label{remark:affine_site}
Let $(A_0, M_0)$ be a pre-log $k$-algebra and let 
\[ Z^{\log} = \spec (A_0, M_0)^a \] be the associated log scheme. Then we sometimes write $(A_0, M_0)_{\mathrm{crys}} = (Z^{\log})_{\mathrm{crys}}$ by a slight abuse of notations. In this article, we will also consider the $p$-completed affine log crystalline site $(A_0, M_0)_{\mathrm{crys}}^{\mathrm{aff}, \wedge}$ over $W(k)$ (for example see \cite[Tag 07HL and 07KH]{SP}). In particular, objects in the opposite category of $(A_0, M_0)_{\mathrm{crys}}^{\mathrm{aff}, \wedge}$ are $p$-completed as $W(k)$-algebras. By Zariski descent, the natural functor 
\[ (A_0, M_0)_{\mathrm{crys}}^{\mathrm{aff}, \wedge} \ra (A_0, M_0)_{\mathrm{crys}} \]on the underlying categories induces an equivalence on the corresponding topoi. 
\end{remark}

\subsection{Setup and notations} \label{ss:notation_section_2}
Now let us recall the setting for Theorem \ref{theorem:log_crystal}.  Let $Z$ be a smooth geometrically connected variety over $k$. Let $s = \spec k$ and let $s^{\log}$ denote the log point associated to the pre-log ring $\N \ra k$ sending $1 \mapsto 0$. We sometimes just write $s^{\log}=(\spec k, 0^{\N})$, and refer to such a point as a \emph{standard log point}. Let $Z^{\log}$ be the fs log scheme obtained from the base change 
    \[
    \begin{tikzcd} 
    Z^{\log} \arrow[r, "\pi"] \arrow[d] & Z \arrow[d] \\ 
    s^{\log} \arrow[r] & s 
    \end{tikzcd},
    \]
%
where we write
\begin{equation} \label{eq:map_on_X_s_forgetting_log}
     \pi: Z^{\log} \ra Z
\end{equation}
for the canonical map induced from $s^{\log} \ra s$. 
Also recall that, for a point $z: \spec k' \ra Z$ on $Z$, we denote by $z^{\log}$ the base change of $z \rightarrow Z$ along $\pi$. Abstractly, $z^{\log}$ is isomorphic to a standard log point $(\spec k', 0^{\N})$.

 \begin{notation}
 Let us also fix the following notations. 
\begin{itemize}   
    \item In a divided power algebra (PD algebra), we write $a^{[n]}$ to denote the $n^{\mathrm{th}}$ divided power of an element $a$, and write $A[x]^{\mathrm{PD}}$ (resp. $A \gr{x}^{\mathrm{PD}}$) for the ring of PD polynomials over a ring $A$ (resp. its $p$-adic completion). 
    \item To ease notation, we often write $W= W(k)$ and write $K_0 = W[1/p]$.  
    \item  Let $\mathrm{Mod}_{/W(k)}^{\varphi}$ denote  the category of $\varphi$-modules over $W(k)$, which consists of a finite free $W(k)$-module $M$ equipped with a $\varphi_{W(k)}$-linear Frobenius map $\varphi: M \ra M$, such that $\varphi$ becomes an isomorphism upon inverting $p$.   One similarly defines the category $\mathrm{Mod}_{/K_0}^{\varphi}$ of $\varphi$-modules over $K_0$.    
    \item Let $\mathrm{Mod}_{/W(k)}^{N}$ denote the category of finite free $W(k)$-modules $M$ equipped with a $W(k)$-linear endomorphism $N: M \ra M$. One similarly defines the category $\mathrm{Mod}_{/K_0}^{N}$.  
    \item  Let $\mathrm{Mod}_{/W(k)}^{\varphi, N}$ denote the category of $(\varphi, N)$-modules over $W(k)$, which consists of a $\varphi$-module $(M, \varphi) \in \mathrm{Mod}^{\varphi}_{/W(k)}$ and a $W(k)$-linear endomorphism $N: M \ra M$, satisfying the condition $N \varphi = p \varphi N.$ One similarly defines the category $\mathrm{Mod}_{/K_0}^{\varphi, N}$ of $(\varphi, N)$-modules over $K_0$.\footnote{Note that the condition $N \varphi = p \varphi N$ forces the endomorphism $N$ to be nilpotent.}
    \item We naturally identify $\mathrm{Mod}_{/W(k)}^{\varphi}$  with the subcategory $ \mathrm{Mod}_{/W(k)}^{\varphi, N = 0}$ of $(\varphi, N)$-modules with $N = 0$. Similarly, we view the category $\mathrm{Mod}_{/W(k)}$ of finite free $W(k)$-modules as the subcategory of $\mathrm{Mod}_{/W(k)}^{N}$ with $N = 0$.  
\end{itemize} 
\end{notation}

\subsection{Monodromy of log (iso)crystals and rigidity}

The work of Hyodo--Kato \cite{Hyodo_Kato}  associates to every log $F$-(iso)crystal over $s^{\log}$  a $(\varphi, N)$-module over $W(k)$ (resp. over $K_0$).   In fact, they associate to every log (iso)crystal an object in $\mathrm{Mod}^{N}_{/W(k)}$ (resp. in $\mathrm{Mod}^{N}_{/K_0}$). We will recall this construction in \S \ref{ss:log_point}.  
The endomorphism $N$ of this associated object is called the \textit{monodromy} of the log $F$-(iso)crystal or log (iso)crystal.

For a point $z:\spec k'\ra Z$ on $Z$ and a log (iso)crystal $\mE$ on $Z^{\log}$, we obtain a log (iso)crystal $\mE_z$ by restricting $\mE$ along $z^{\log}\ra Z^{\log}$. Then the monodromy of $\mE_z$ is called the \emph{monodromy} of $\mE$ at $z$. In particular, we say that $\mE$ \emph{has trivial monodromy at $z$} if the monodromy of $\mE_z$ is zero.

\begin{definition} 
Retain the notations from above. We say that a log (iso)crystal $\mE$ over $Z^{\log}$ \emph{descends} to $Z$ if there exists an (iso)crystal $\mE'$ over $Z$ such that 
\[ \mE \cong \pi^* \mE' \] 
where $\pi$ is the canonical map in (\ref{eq:map_on_X_s_forgetting_log}). 
We make a similar definition for log $F$-(iso)crystals. 
\end{definition}

We can now state the main result of this section.  

\begin{theorem}[Theorem \ref{theorem:log_crystal}]  \label{theorem:log_crystal_1}
Let $Z$ be a smooth geometrically connected variety over $k$ as above and let $\mE$ be a log crystal over $Z^{\log}$. 
\begin{enumerate}
\item If $\mE$ has trivial monodromy at one closed point $z\in Z$, then $\mE$ descends to a crystal over $Z$. In particular, $\mE$ has trivial monodromy at all closed points on $Z$. 
\item  More generally, $\mE$ has ``constant monodromy'' on $Z$; that is, the nilpotent rank of the monodromy operator is constant on the closed points of $Z$. 
\end{enumerate} 
The same assertions apply to log isocrystals, as well as log $F$-crystals and log $F$-isocrystals over $Z^{\log}.$   
\end{theorem}

\subsection{Semistable and crystalline local systems} 

Before we prove Theorem \ref{theorem:log_crystal_1}, let us first demonstrate how to deduce Theorem \ref{theorem:main_intro} from Theorem \ref{theorem:log_crystal_1}. For convenience of the reader, we recall the notion of crystalline and semistable $\Z_p$-local systems. 

\begin{definition}[{\cite[Definition 3.4]{GY}}] \label{def:st_local_sys} Let $\fX$ be a $p$-adic formal scheme over $\mO_K$ with semistable reduction and let $X$ be its adic generic fiber over $K$. Let $\fX_s$ denote the special fiber of $\fX$ and let $\fX_s^{\log}$ denote the log scheme $(\fX_s, \alpha_s)$ where the log structure $\alpha_s$ is obtained from the pullback of the divisorial log structure  
\[\alpha: \mO_{\fX, \ett} \cap (\mO_{\fX, \ett}[1/p])^{\times} \ra \mO_{\fX, \ett} \]  on $\fX$ (also see the setup in \S \ref{ss:st_log_scheme}). Let $\L$ be an \'etale $\Z_p$-local system on $X$. 
\begin{enumerate}
    \item We say that $\L$ is \textit{semistable} if there exists a log $F$-isocrystal $(\mE, \varphi)$ over $\fX_s^{\log}$, together with a Frobenius equivariant isomorphism of $\mathbb{B}_{\mathrm{crys}}$-vector bundles
    \[ 
     \mathbb{B}_{\mathrm{crys}} (\mE) \isom \mathbb{B}_{\mathrm{crys}} \otimes_{\Z_p} \L.
    \] Here $\mathbb{B}_{\mathrm{crys}}(\mE)$ is the sheaf of $\mathbb{B}_{\mathrm{crys}}$-modules on the pro-\'etale site $X_{\proet}$ associated to $(\mE, \varphi)$ as defined in \cite[Construction 4.1]{GY}. 
    \item We say that $\L$ is \textit{crystalline} if there exists an $F$-isocrystal $(\mE, \varphi)$ over $\fX_s$, together with a Frobenius equivariant isomorphism of $\mathbb{B}_{\mathrm{crys}}$-vector bundles
    \[ 
     \mathbb{B}_{\mathrm{crys}} (\mE) \isom \mathbb{B}_{\mathrm{crys}} \otimes_{\Z_p} \L.
    \]
\end{enumerate}
We remark that our $F$-isocrystals and log $F$-isocrystals are necessarily locally free according to Definition \ref{def:isocrystals}. 
\end{definition}

Note that, \textit{a priori}, the notions of crystallinity and semistability of local systems depend on the integral model $\fX$ over $\mO_K$. The fact that they are actually independent of the integral models can be deduced from either the purity result of \cite{DLMS2} or the main result of \cite{GY} on the pointwise criterion for crystalline and semistable local systems. Now let us restate and prove Theorem \ref{theorem:main_intro}. 

\begin{theorem}[Theorem \ref{theorem:main_intro}] \label{theorem:main_intro_1}
Let $\fX$ be a geometrically connected smooth $p$-adic formal scheme over $\mO_K$ with adic generic fiber $X$. Let $\L$ be a semistable \'etale $\Z_p$-local system on $X$ and suppose that there exists a classical point $x_0 \in X(K)$, such that $\L$ is crystalline at $x_0$, then $\L$ is crystalline. In particular, it is crystalline at all classical points on $X$.
\end{theorem}

\begin{proof}
    By Definition \ref{def:st_local_sys},  this immediately follows from Theorem \ref{theorem:log_crystal_1}. 
\end{proof}

The proof of Theorem \ref{theorem:log_crystal_1} will occupy the rest of this section. The key to the proof is to understand the monodromy of a log (iso)crystal over $Z^{\log}$ via certain descent data. The proof is in fact surprisingly elementary and boils down to explicit computations of coefficients of certain power series.

\subsection{Log crystals over a log point}
\label{ss:log_point}

It is instructive to first review the classical construction of Hyodo--Kato \cite{Hyodo_Kato} using a language that is tailored to our setup.  

\begin{lemma}[Hyodo--Kato]\label{lemma:descent_over_log_point}
Let $s^{\log}=(\spec k, 0^{\N})$ as above. There is a canonical functor 
\[
    \mathrm{ev}: \mathrm{Vect}  ((s^{\log})_{\mathrm{crys}}) \lra  \mathrm{Mod}_{/W(k)}^{N}
\]
which upgrades to a canonical functor 
\[
    \mathrm{ev}: \mathrm{Vect}^{\varphi} ((s^{\log})_{\mathrm{crys}}) \lra  \mathrm{Mod}_{/W(k)}^{\varphi, N}
\]
along the natural forgetful functors on both sides, 
such that a log crystal $\mE_0 \in \mathrm{Vect}((s^{\log})_{\mathrm{crys}})$ (resp. a log $F$-crystal $\mE_0 \in \mathrm{Vect}^{\varphi} ((s^{\log})_{\mathrm{crys}})$) descends to a crystal (resp. an $F$-crystal) $\mE'_0$ over $s$ if and only its monodromy $N$ is trivial, in other words, if and only if $\mathrm{ev}(\mE_0)$ lands in the subcategory $\mathrm{Mod}_{/W(k)}$ (resp. $\mathrm{Mod}_{/W(k)}^{\varphi}$). Similar statements hold for log isocrystals and log $F$-isocrystals over $s^{\log}$. 
\end{lemma}

This is a well-known result and  a simple exercise in the theory of log crystals (and at least implicitly stated in \cite[2.17, 2.18]{Hyodo_Kato}, see also \cite[Example 3.5]{Yao_st}). Nevertheless it is illustrative to recall the proof.  

\begin{proof}
Let us prove the lemma for log crystals and leave the case of log $F$-crystals to Remark \ref{remark:role_of_Frobenius}. The case for log isocrystals and log $F$-isocrystals will follow from similar arguments.   As in  \cite[Example 3.5]{Yao_st}, we consider the object $W_{\mathrm{HK}} = (W(k), 0^{\N})$ in (the opposite category of) the $p$-completed affine log crystalline site of $s^{\log}$ over $W(k)$. Taking its self-coproduct, we obtain 
\[
W_{\mathrm{HK}}^{(1)} = (W[u, v,  (u/v - 1)^{[i]} ]_{i \ge 1}/(u^{[l]})_{l \ge 1})^{\wedge}_p \cong W \gr{t}^{\mathrm{PD}}\]
where $W\gr{t}^{\mathrm{PD}}$ denotes the $p$-completed PD polynomial ring over $W$, and $t = u/v - 1$. The triple self-coproduct of $W_{\mathrm{HK}}$ is 
\[
W_{\mathrm{HK}}^{(2)} = (W[u, v, w,  (u/v - 1)^{[i]},  (v/w - 1)^{[j]}]_{i, j \ge 1}/(u^{[l]})_{l \ge 1})^{\wedge}_p \cong W \gr{t_1, t_2}^{\mathrm{PD}} 
\]
where $t_1 = u/v - 1$ and $t_2 = v/w - 1$. We can now evaluate a log crystal $\mE_0$ on the simplicial object
\begin{equation} \label{eq:simplicial_HK_complex}
W_{\mathrm{HK}} \mathrel{\substack{\textstyle\longrightarrow\\[-0.6ex] \textstyle\longrightarrow}}  W_{\mathrm{HK}}^{(1)}  \mathrel{\substack{\textstyle\longrightarrow\\[-0.6ex] \textstyle\longrightarrow \\[-0.6ex] \textstyle\longrightarrow}} W_{\mathrm{HK}}^{(2)} 
\end{equation}
where the double arrows $\iota_1, \iota_2: W_{\mathrm{HK}} \ra W_{\mathrm{HK}}^{(1)}$ are both given by the inclusion $W \ra W\gr{t}^{\mathrm{PD}}$,\footnote{In other words, they induce the same map on the underlying rings, but the two maps $\iota_1, \iota_2$ differ on the (pre-)log structures.} while the triple arrows $\iota_{12}, \iota_{23}, \iota_{13}: W_{\mathrm{HK}}^{(1)} \ra W_{\mathrm{HK}}^{(2)}$ are given by 
\[
\iota_{12} (t) = t_1, \quad \iota_{23} (t) = t_2, \quad \iota_{13} (t) = t_1t_2 + t_1 + t_2, 
\]
respectively. In other words, the data of an object $\mE_0 \in  \mathrm{Vect}((s^{\log})_{\mathrm{crys}}) $ naturally gives rise to the data of a finite free $W$-module $M_0$ equipped with a $W\gr{t}^{\mathrm{PD}}$-linear isomorphism 
\begin{equation} \label{eq:the_delta_map_over_a_point}
\delta: M_0 \otimes_{W, \iota_1} W \gr{t}^{\mathrm{PD}} \isom  W \gr{t}^{\mathrm{PD}} \otimes_{W, \iota_2} M_0
\end{equation}
that satisfies the usual cocycle condition $\iota_{23}^* (\delta) \circ \iota_{12}^* (\delta) = \iota_{13}^* (\delta).$ Pick a basis $e_1, ..., e_d$ of $M_0$ over $W$, then the descent isomorphism $\delta$ is determined by a matrix $T_{\delta} \in M_{d} (W \gr{t}^{\mathrm{PD}})$, which we may further write as 
\begin{equation} \label{eqn:expression_of_descent}
T_{\delta} = N_0 + N_1 t + N_2 t^2 + ... 
\end{equation}
with $N_i \in \frac{1}{i!} M_d (W)$ for each $i$.  The cocycle condition can now be rewritten as 
\begin{equation*} 
\iota_{23} (T_{\delta}) \cdot \iota_{12} (T_{\delta}) = \iota_{13} (T_{\delta})
\end{equation*}
in terms of matrices. If we write $T_{\delta} = T_{\delta} (t)$ as a power series on $t$ with coefficients in $M_d (K_0)$, then the cocycle condition becomes $T_{\delta} (t_2) \cdot T_{\delta} (t_1) = T_{\delta} (t_1 t_2 + t_1 + t_2),$ in other words
\begin{multline} \label{eqn:cocycle}
 \qquad  \Big( N_0 + N_1 t_2 + N_2 t_2^2 + ...  \Big) \cdot \Big( N_0 + N_1 t_1 + N_2 t_1^2 + ...  \Big)  \\  = N_0 + N_1 \big(t_1t_2+t_1+t_2 \big) + N_2 \big(t_1t_2+t_1+t_2 \big)^2 + ... \quad \qquad 
\end{multline} 
This immediately implies that $N_0 = 1$ is the identity matrix. A simple computation shows that all the $N_i$'s are determined by $N_1$. If fact, if we write $N = N_1 \in M_d (W)$, then we have $N_i = \frac{1}{i ! } N (N-1) \cdots (N-i+1)$ for all $i \ge 1$. From this, we can further compute that $T_{\delta}$ is of the form 
\begin{equation} \label{eqn:expression_of_T_delta_using_N}
T_{\delta} (t) = \exp (N \cdot \log(t+1)) := 1+ N \log(t+1) + \frac{N^2 (\log (t+1))^2}{2!} +  \cdots.
\end{equation}
This gives the desired functor in the lemma (which one checks is independent of the chosen basis).  
For the last statement, we simply observe that, since $W(k)$ is the initial object in (the opposite category of) the $p$-completed affine crystalline site of the point $s$, the category $\mathrm{Vect} (s_{\mathrm{crys}}) $ of crystals over $s$ is equivalent to the category $\mathrm{Mod}_{/W(k)}$ of finite free $W(k)$-modules. If $\mathrm{ev} (\mE_0)$ has trivial monodromy, then $T_{\delta} = 1$ is the identity matrix by (\ref{eqn:expression_of_T_delta_using_N}), thus the desired assertion follows. 
\end{proof}

\begin{remark} \label{remark:role_of_Frobenius} 
If we start with a log $F$-crystal $\mE_0 \in  \mathrm{Vect}^{\varphi} ((s^{\log})_{\mathrm{crys}})$, then evaluating on the simplicial complex (\ref{eq:simplicial_HK_complex}) gives rise to  a finite free $W$-module $M_0$ equipped with a $\varphi_{W(k)}$-linear Frobenius map $\varphi: M_0 \ra M_0$ which is an isogeny, together with a $W\gr{t}^{\mathrm{PD}}$-linear isomorphism  $\delta$ as in (\ref{eq:the_delta_map_over_a_point}). Furthermore, the descent map $\delta$ is compatible with the Frobenius map (which sends $t \mapsto (t+1)^p - 1$ on $W \gr{t}^{\mathrm{PD}}$) and satisfies the usual cocycle condition $\iota_{23}^* (\delta) \circ \iota_{12}^* (\delta) = \iota_{13}^* (\delta)$ as before. The relation $N \varphi = p \varphi N$ follows from the compatibility of the Frobenius operator $\varphi$ with the descent isomorphism $\delta$. 
\end{remark}

\subsection{Log (iso)crystals over smooth schemes over $k$} \label{ss:log_F_crystal_over_smooth_schemes}

Now let us prove Theorem \ref{theorem:log_crystal_1}. 

\begin{proof}[Proof of Theorem \ref{theorem:log_crystal_1}] We start with Part (1). Let us first prove the assertion for log crystals. The case of log isocrystals will follow from a similar argument, which we treat at the end of the proof. The case of log $F$-crystals and log $F$-isocrystals then follow automatically, similar to Remark \ref{remark:role_of_Frobenius}.  \\

\noindent \textbf{The case of $\A^n_k$ (Part I).} \\ 
Let us first treat the case when $Z = \A^n_{k}$ is the affine space over $k$. In fact, it is illustrative to first treat the case when $n = 1$, since the only difference for general affine space is notational. 

In the case when $Z = \A^1_{k}$, we have $Z^{\log}=\A_k^{1,\log}=\spec (k[x], 0^{\N})^a$. Let us consider the object $A^{(0)} = W \gr{x}$ in (the opposite category of) the $p$-completed affine crystalline site $k[x]^{\mathrm{aff}, \wedge}_{\mathrm{crys}}$ of $k[x]$, which gives rise to an object $(A^{(0)}, 0^{\N})$ in (the opposite of) the $p$-completed affine log crystalline site $(k[x], 0^{\N})_{\mathrm{crys}}^{\mathrm{aff}, \wedge}$ (see Remark \ref{remark:affine_site}). Let $A^{(1)}$ (resp. $A^{(2)}$) be the self-coproduct (resp. triple self-coproduct) of $A^{(0)}$ in $k[x]^{\mathrm{aff}, \wedge, \mathrm{op}}_{\mathrm{crys}}$, and let  $A^{(1), \log}$ (resp. $A^{(2), \log}$) be the underlying ring of the self-coproduct (resp. triple self-coproduct) of $(A^{(0)}, 0^{\N})$ in $(k[x], 0^{\N})_{\mathrm{crys}}^{\mathrm{aff}, \wedge, \mathrm{op}}$. Then we have 
\begin{align} \label{eq:A^(1)} 
    A^{(1)} & = W\gr{x_1, x_2} \gr{x_2 - x_1}^{\mathrm{PD}} =  W\gr{x_1} \gr{x_2'}^{\mathrm{PD}} \\ \nonumber 
    A^{(2)} & = W\gr{y_1, y_2, y_3} \gr{y_2 - y_1, y_3 - y_2}^{\mathrm{PD}} =  W\gr{y_1} \gr{y_2', y_3'}^{\mathrm{PD}}
\end{align}
where we have adopted the notations 
\[ x_2' = x_2 - x_1, \quad y_2' = y_2 - y_1, \quad   y_3' = y_3 - y_2.\] Moreover, we have 
\begin{align} \label{eq:A^(2)} 
    A^{(1), \log} & = W\gr{x_1} \gr{x_2', t}^{\mathrm{PD}} \\ \nonumber 
    A^{(2), \log} & = W\gr{y_1} \gr{y_2', y_3', t_1, t_2}^{\mathrm{PD}}, 
\end{align}
which is similar to the case of a log point treated in \S \ref{ss:log_point}. We have the simplicial object
\begin{equation} \label{eqn:simplicial_over_A^0}
A^{(0)} \mathrel{\substack{\textstyle\longrightarrow\\[-0.6ex] \textstyle\longrightarrow}}  
A^{(1), \log}  \mathrel{\substack{\textstyle\longrightarrow\\[-0.6ex] \textstyle\longrightarrow \\[-0.6ex] \textstyle\longrightarrow}} A^{(2), \log}  
\end{equation}
where $\iota_1, \iota_2: A^{(0)} \ra A^{(1), \log}$ are induced from the maps from $A^{(0)} $ to $A^{(1)}$ (the non-log version) given by 
\[\iota_1(x) = x_1, \quad \iota_2 (x) = 
x_1 + x_2',\] 
while the arrows $\iota_{12}, \iota_{13}, \iota_{23}: A^{(1), \log} \ra A^{(2), \log}$ are given by 
\begin{itemize}
    \item $\iota_{12}: \quad  x_1 \mapsto y_1, \quad  x_2' \mapsto y_2', \quad \mathrm{and } \:\:  t \mapsto t_1$
    \item $\iota_{13}: \quad x_1 \mapsto y_1, \quad x_2' \mapsto y_2' + y_3', \quad \mathrm{and } \:\:  t \mapsto t_1t_2 + t_1+t_2$ 
    \item $\iota_{23}: \quad x_1 \mapsto y_1+y_2', \quad x_2' \mapsto y_3', \quad \mathrm{and } \:\:  t \mapsto t_2$. 
\end{itemize}
As in the proof of Lemma \ref{lemma:descent_over_log_point} for the log point, the data of a log crystal $\mE$ over $\spec (k[x], 0^{\N})^a$ gives rise to a 
finite projective module $M$ over $W\gr{x}$, together with an $A^{(1), \log}$-linear isomorphism 
\[
\delta: M \otimes_{W\gr{x}, \iota_1} A^{(1), \log} \isom A^{(1), \log} \otimes_{W\gr{x}, \iota_2} M 
\] 
which satisfies the cocycle condition $\iota_{23}^* (\delta) \circ \iota_{12}^* (\delta) = \iota_{13}^* (\delta).$

Let us first assume that $M$ is a finite free module over $W\gr{x}$ and treat the more general case afterwards. We follow the same strategy as in the proof of Lemma \ref{lemma:descent_over_log_point}. Recall that 
$A^{(1), \log} = A^{(1)} \gr{t}^{\mathrm{PD}}$. Upon picking a basis of $M$ over $W \gr{x}$, we can write the matrix $T_{\delta}$ for the descent isomorphism $\delta$ as  
\begin{equation} \label{eq:expression_of_delta_affine}
    T_{\delta} (t) = N_0  + N_1 t + N_2 t^2 + \cdots 
\end{equation} 
where each $N_i \in \frac{1}{i!} M_d(A^{(1)})$ is a matrix with entries in $\frac{1}{i!} A^{(1)}$. We shall regard each $N_i=N_i(x_1, x_2)$ as a matrix whose entries are functions on $x_1, x_2$. In particular, for $a, b\in W(\bar{k})$ such that $a\equiv b$ mod $p$, we obtain a matrix $N_i(a, b)$ with entries in $\frac{1}{i!}W(\bar{k})$. Likewise, we may write $\iota_{12} (N_i)$ as $N_i (y_1, y_2)$ and regard it as a matrix whose entries are functions in variables $y_1, y_2$. We similarly write $N_i (y_2, y_3) = \iota_{23} (N_i)$ and $N_i (y_1, y_3) = \iota_{13}  (N_i)$. The cocycle condition becomes
\begin{align}  \label{eq:cocycle_in_coordinates}
    \Big( N_0(y_2, y_3) + N_1 (y_2, y_3) t_2 + N_2 (y_2, y_3) t_2^2 + \cdots  \Big) 
  \cdot \Big( N_0(y_1, y_2) + N_1(y_1, y_2)  t_1 + N_2 (y_1, y_2)  t_1^2 + \cdots  \Big) \:\:\:  \nonumber
 \\    
   = N_0(y_1, y_3) + N_1(y_1, y_3) \cdot \big(t_1t_2+t_1+t_2 \big) + N_2(y_1, y_3) \cdot \big(t_1t_2+t_1+t_2 \big)^2 + \cdots \quad 
\end{align}
This is an identity of power series in $t_1, t_2$ with coefficients in $M_d (A^{(2)})[1/p]$. From this we immediately obtain the relation 
\begin{align} \label{eqn:cocycle_over_S}
  N_0 (y_1, y_3)  =  N_0 (y_2, y_3) \cdot N_0 (y_1, y_2)   
\end{align}
(in $M_d (A^{(2)})[1/p]$) by comparing the constant terms on both sides. Since 
\[ A^{(1)} = W \gr{x_1} \gr{x_2'}^{\mathrm{PD}} \]  where $x_2' = x_2 - x_1$, we may further write $N_0=N_0 (x_1, x_2)$ as 
\[
N_0 = A_0 + A_1 x_2' +  A_2 (x_2')^{2} + \cdots 
\]
where each $A_i \in \frac{1}{i!} M_d (W \gr{x_1})$ is viewed as a function with variable $x_1$. Condition (\ref{eqn:cocycle_over_S}) becomes 
\begin{align*}
&\qquad \qquad \Big(A_0(y_1) + A_1(y_1) (y_2'+y_3') + A_2(y_1) (y_2' + y_3')^2  +  \cdots\Big)  =   \\  
&\Big( A_0(y_1+y_2') + A_1(y_1+y_2') y_3' + A_2 (y_1+y_2') (y_3')^2     + \cdots\Big) \cdot \Big( A_0 (y_1) + A_1(y_1) y_2' + A_2 (y_1) (y_2')^2 + \cdots\Big). 
\end{align*}
 This is an identity of matrix-valued functions in variables $y_1, y'_2, y'_3$. Setting $y_2' = y_3' = 0$ we see that $A_0 = 1$ is the identity matrix. Moreover, by setting $y_1 = y_3$ (or equivalently, setting $y_2' = - y_3'$), we see from (\ref{eqn:cocycle_over_S}) again that $N_0 \in \GL_d (A^{(1)})$ is invertible. In fact, we have 
\[
N_0 (x_1, x_2) = N_0 (x_2, x_1)^{-1}. 
\]
Next, by considering the coefficients of $t_1, t_2$ 
in the cocycle condition (\ref{eq:cocycle_in_coordinates}), we obtain the following relations
\begin{align}  
N_1 (y_1, y_3)   
  \label{eqn:cocycle_over_S_2}  
  & =    N_0 (y_2, y_3)  \cdot N_1 (y_1, y_2)   \\  \label{eqn:cocycle_over_S_3}
  & =  N_1 (y_2, y_3) \cdot  N_0 (y_1, y_2).  
\end{align}
Similar to our analysis on $N_0$, we may write $N_1 (x_1, x_2) $ as 
\begin{equation} \label{eq:expression_of_N_1}
N_1 = B_0(x_1) + B_1(x_1) \cdot  x_2' +  B_2(x_1) \cdot (x_2')^{2} + \cdots 
\end{equation} 
where each $B_i$ lives in $\frac{1}{i!} M_d (W \gr{x_1})$. 

Note that, specializing the log crystal $\mE$ to a closed $\cl k$-point $\cl \alpha$ on $\A^1$ amounts to specializing the simplicial object in (\ref{eqn:simplicial_over_A^0}) along $A^{(0)} \ra W(\cl k)$ via $x \mapsto [\cl \alpha]$, along $A^{(1)} \ra W(\cl k)$ via $x_1, x_2 \mapsto [\cl \alpha]$, and along $A^{(2)} \ra W(\cl k)$ via $y_1, y_2, y_3 \mapsto [\cl \alpha]$, where $[\cl \alpha]$ denotes the Teichmuller lift of $\cl \alpha$. Therefore, by the assumption of Theorem \ref{theorem:log_crystal_1}, we have $N_1 ([\cl \alpha], [\cl \alpha]) = 0$ for some $\cl \alpha \in \cl k$.

We claim that, if $N_1 (a, a) = 0$ for some $a \in W(\cl k)$, then $N_1 (b, c) = 0$ for any $b, c \in W(\cl k)$ such that $b \equiv c \mod p$. Note that this in turn implies that $N_1 = 0$. 
To prove the claim, let us observe that, setting $y_1 = y_3 = a$ and $y_2 = b_0$ we obtain 
\[
N_1 (a, b_0) = N_0 (a, b_0) \cdot N_1 (a, a) = 0
\] 
for any $b_0 \in W(\cl k)$ such that $b_0 \equiv a \mod p$, using $(\ref{eqn:cocycle_over_S_2})$ and  $N_0 (a, b_0) = N_0 (b_0, a)^{-1}$. Similarly, setting $y_1 = a$ and $y_2 = y_3 = b_0$ in (\ref{eqn:cocycle_over_S_3}) we know that 
\[
N_1 (b_0, b_0) =    N_1 (a, b_0) \cdot N_0 (b_0, a)  = 0 
\] 
for any $b_0 \in W(\cl k)$ such $b_0 \equiv a \mod p$. In other words, the analytic function $B_0 (x_1)$ in (\ref{eq:expression_of_N_1}) satisfies $B_0 (b_0) = 0$ for any $b_0 \in W(\cl k)$ such that $b_0 \equiv a \mod p$, thus $B_0 (x_1) = 0$ (for example, by the $p$-adic Weierstrass preparation theorem), and this further implies that $N_1 (b, b) = 0$ for any $b \in W (\cl k).$ Now applying the previous argument one more time we know that $N_1 (b, c) = 0$ for all $b, c \in W(\cl k)$ such that $b \equiv c \mod p$. This justifies the claim. 

Next, by inductively looking at coefficients of terms of degree $i$ (in other words, terms of the form $t_1^{i_1} t_2^{i_2}$ with $i_1 + i_2 = i$) in (\ref{eq:cocycle_in_coordinates}) for $i = 1, 2, \ldots$, we conclude that $N_i = 0$ for all $i \ge 1$. Consequently, the descent isomorphism $\delta$ becomes a descent isomorphism over $A^{(1)}$.  In particular, the monodromy of the log crystal is trivial at all closed points, and the log crystal descends to a crystal over $\A^1_k$. This finishes the proof under the assumption that $M$ is finite free.\\

\noindent \textbf{The case of $\A^n_k$ (Part I\!I).} \\ 

\noindent For general $M$, consider a finite cover of $Z=\A^1_k$ by affine open subschemes, on each of which $M$ is finite free. Since the statement is insensitive to replacing $k$ by a finite extension, by smoothness, we may assume that each of these affine open subschemes admits an \'etale map to $\A^1_k$. Since we can find closed points on intersections of affine open subschemes, we can argue one by one on these affine open subschemes.

Our setup is now as follows: consider an \'etale algebra $k[x] \ra S_0$ (which corresponds to a $p$-completely faithfully flat \'etale map  $A^{0} = W\gr{x} \ra S$). Let $S^{(\bullet), \log}$ (resp. $S^{(\bullet)}$) denote the $p$-completed  simplical object formed by taking the Cech nerve of $S$ in the $p$-completed affine log crystalline site of $(k[x], 0^{\N})$ (resp. in the $p$-completed affine crystalline site of $k[x]$). We have a finite free module $M_S$ over $S$ from the data of the log crystal $\mE$, which is equipped with a  descent isomorphism
\begin{equation} \label{eq:descent_iso_over_S}
\delta_S: M_S \otimes_{S, \iota_1} S^{(1), \log} \isom S^{(1), \log} \otimes_{S, \iota_2} M_S 
\end{equation} that satisfies the cocycle condition. 
Let  $I = \ker(S \otimes_{W\gr{x}} S \ra S)$ and $J = \ker(S \otimes_{W} S \ra S)$ be the kernel of the respective multiplication maps, and let \[ S' := (S \otimes_{W\gr{x}} S)  \gr{I}^{\mathrm{PD}} \]be the $p$-completed PD-envelop of $S \otimes_{W\gr{x}} S$ along the kernel $I$. Note that we have isomorphisms 
\[ S \otimes_{W\gr{x}} S \cong S\otimes_{W} S/(x_2 - x_1) = S \otimes_W S/(x_2')\]
and $I = J/(x_2')$, using the same set of notations as in Part I. In particular, there is a natural map 
\[ S^{(1)} = (S \otimes_{W} S) \gr{J}^{\mathrm{PD}} \ra S'\]   induced from the natural surjection $S \otimes_W S \ra S \otimes_{W \gr{x}} S$. Note that, the natural map $S^{(1)} \ra S$ induces a nil-thickening after modulo $p^n$ (for each $n$), thus the \'etale map $\iota_1: S \ra S^{(1)}$ induces a unique lifting 
$\iota_1: S \widehat \otimes_{W \gr{x}} S \ra S^{(1)}$ compatible with the natural projections to $S$. This in turn induces a map $S' \ra S^{(1)}$, which further induces an isomorphism 
\[
S' \gr{x_2'}^{\mathrm{PD}} \isom S^{(1)}. \footnote{
We remark that choosing $\iota_2$ instead of $\iota_1$ amounts to writing $A^{(1)}$ as $W\gr{x_2} \gr{x_2'}^{\mathrm{PD}}$ instead of $W\gr{x_1} \gr{x_2'}^{\mathrm{PD}}$. }
\]  
From this identification, we obtain  
\[S^{(1), \log} = S^{(1)} \gr{t}^{\mathrm{PD}} \cong S' \gr{x_2', t}^{\mathrm{PD}} \qquad \] and 
\[
S^{(2), \log} = S^{(2)} \gr{t}^{\mathrm{PD}}  \cong S'' \gr{y_2', y_3', t}^{\mathrm{PD}}, \: 
\]
where 
\[ S'' = (S \otimes_{W \gr{x}} S \otimes_{W \gr{x}} S) \gr{I^{(2)}}^{\mathrm{PD}}\] is the $p$-completed PD envelop of the kernel $I^{(2)}$ of the natural multiplication  map. 
Since $S^{(1), \log} = S^{(1)} \gr{t}^{\mathrm{PD}}$, we may once again express the matrix $T_{\delta_S}$ for the descent isomorphism $\delta_S$ (upon picking a basis of $M_S$ over $S$) as  
\[
T_{\delta_S} (t) = N_0  + N_1 t + N_2 t^2 + \cdots 
\] 
where each $N_i \in \frac{1}{i!} M_d(S^{(1)})$. The cocycle condition (\ref{eq:cocycle_in_coordinates}) again implies that $\iota_{13} (N_0) =  \iota_{23} (N_0) \cdot \iota_{12} (N_0).$ 

Let us introduce one slight generalization of notations. Let $\alpha, \beta: S \ra W(\cl k)$ be two $W(\cl k)$-points on $S$ such that they give rise to the same reduction map $\cl \alpha = \cl \beta: S \ra \cl k$, then we write \[(\alpha, \beta): S^{(1)} \ra W (\cl k)\] for the corresponding map it induces on $S^{(1)}$. We further write 
\[ N_i (\alpha, \beta) \in \frac{1}{i!} M_d (W(\cl k)) \] for the image of $N_i$ along the map $\frac{1}{i!}M_d (S^{(1)}) \ra \frac{1}{i!} M_d (W(\cl k))$ induced by $(\alpha, \beta)$. The cocycle condition implies that 
\[
N_0 (\alpha_1, \alpha_3) = N_0 (\alpha_2, \alpha_3) \cdot N_0 (\alpha_1, \alpha_2)
\]
for any three maps $\alpha_1, \alpha_2, \alpha_3$ such that $\cl \alpha_1 = \cl \alpha_2 =\cl \alpha_3$. In particular, this implies that $N_0 (\alpha, \alpha) = 1$ for any $\alpha$ and that $N_0 (\alpha, \beta)$ is invertible 
for any $\alpha, \beta: S \ra W(\cl k)$ such that $\cl \alpha = \cl \beta$. As in Part I, we have 
\begin{align} \label{eq:cocycle_over_general_S_for_N1}
N_1 (\alpha_1, \alpha_3) & = N_0 (\alpha_2, \alpha_3) \cdot N_1 (\alpha_1, \alpha_2)   \\ 
    & =   N_1 (\alpha_2, \alpha_3) \cdot N_0 (\alpha_1, \alpha_2) \nonumber
\end{align}
for any three such maps $\alpha_1, \alpha_2, \alpha_3$ as above. 

Now suppose that $N_1 (\alpha, \alpha) = 0$ for some $\alpha: S \ra W(\cl k)$ (this is indeed the case under our assumption). We claim that $N_1 (\beta, \beta) = 0$ for any $\beta: S \ra W(\cl k)$. This follows from a similar argument as in Part I. To simplify notations, we let $\cl N_1$ denote the image of $N_1$ along the natural projection $S^{(1)} \ra S$. Note that $N_1 (\beta, \beta) =  \cl N_1 (\beta)$ for all $\beta: S \ra W(\cl k)$.  Therefore, we have $\cl N_1 (\beta_0 ) = N_1 (\beta_0, \beta_0) = 0$ for all $\beta_0$ such that $\cl \beta_0 = \cl \alpha$ by (\ref{eq:cocycle_over_general_S_for_N1}). We may view $\cl N_1$ as a matrix whose entries are analytic functions on the rigid analytic space associated to $S[1/p]$ (which is \'etale over the rigid analytic disc by construction),  so we must have $\cl N_1 = 0$ since it is zero on a residue disc. This in turn implies that $N_1 (\alpha, \beta) = 0$ for all $\alpha, \beta$ such that $\cl \alpha = \cl \beta$ by applying (\ref{eq:cocycle_over_general_S_for_N1}) once again, and thus we know that $N_1 = 0$. 
By a similar argument as in the previous case, we know that $N_i = 0$ for all $i \ge 1$ and we are done with the case $Z=\A^1_k$. \\



\noindent \textbf{The case of $\A^n_k$ (Part I\!I\!I).} \\  
In the case when $Z = \A^n_k$, we consider $A^{(0)} = W \gr{x^{(1)}, ..., x^{(n)}}$ instead of $W \gr{x}$, and form the self-coproducts 
\begin{align*}
    A^{(1)} & = W \gr{x^{(1)}_1, ..., x^{(n)}_1} \gr{{x_2'}^{(1)}, ..., {x_2'}^{(n)}}^{\mathrm{PD}}  \\ 
     A^{(2)} & = W \gr{y^{(1)}_1, ..., y^{(n)}_1} \gr{{y_2'}^{(1)}, ..., {y_2'}^{(n)}, {y_3'}^{(1)}, ..., {y_3'}^{(n)}}^{\mathrm{PD}}
\end{align*}
in (the opposite cateogy of) the $p$-completed affine crystalline site. As before, we have 
\[ 
 A^{(1), \log} =  A^{(1)} \gr{t}^{\mathrm{PD}} \quad \text{ and } \:\: A^{(2), \log} =  A^{(2)} \gr{t_1, t_2}^{\mathrm{PD}}. 
\]
The rest of the proof carries over \textit{verbatim} as in Part I when $\mE$ gives rise to a finite free module $M$ over $A^{(0)}$, with only some notational complications. (For general $M$, one argues as in Part I\!I.) More precisely, we may still write down the matrix for the descent isomorphism as in  (\ref{eq:expression_of_delta_affine}), except now the matrices $N_i$'s have entries in $A^{(1)}$ whose elements are viewed as functions in $2n$ variables $\{x_1^{(i)}, x_2^{(i)} \}_{1 \le i \le n}$. 
Then, by varying one pair of coordinates at a time, the same argument above shows that, if $N_1 (x_1^{(i)}, x_2^{(i)})= 0$ for $x_1^{(i)} = x_2^{(i)} = a^{(i)}$ for a sequence of elements $a^{(i)} \in W(\cl k)$, then we must have $N_1 = 0$. From this, one further shows that $N_i=0$ for all $i$. Consequently, if the monodromy of $\mE$ is trivial at one closed point,  then it is trivial at all closed points, and the log crystal descends to a crystal  over $\A^n_k$.\\

\noindent \textbf{The general case.} 

\noindent For general $Z$, after replacing $k$ by a finite extension if necessary, we can cover $Z$ by affine open subschemes, each of which admits an \'etale map to $\A^n_k$. We immediately reduce to a question on each of these affine open subschemes. Then one argues in the same way as in Part I\!I above. (Notice that we can assume $M$ is finite free by further shrinking the affine open subschemes.) This finishes the proof of Theorem \ref{theorem:log_crystal_1} for log crystals.\\

\noindent \textbf{Log isocrystals.} 

\noindent For the case of log isocrystals, we have the same setup as above, except now we only have a finite $S$-module $M_S$, equipped with an isomorphism 
\begin{equation} \label{eq:descent_iso_for_isocrystal}
\delta_S:  M_S \widehat \otimes_{S, \iota_1} S^{(1), \log} \isom S^{(1), \log}  \widehat \otimes_{S, \iota_2} M_S
\end{equation}
satisfying the cocycle condition as in (\ref{eq:descent_iso_over_S}), such that $M_S[1/p]$ is a projective $S[1/p]$-module. Suppose that, at some point $\alpha: S \ra W(\cl k)$, the specialization of $\delta_S$ along $\alpha$ descends to an isomorphism coming from a (non-log) isocrystal over $W(\cl k)$, then we want to show that $\delta_S$ descends to an isomorphism $\delta_{S,0}:  M_S \widehat \otimes_{S, \iota_1} S^{(1) } \isom S^{(1)} \widehat \otimes_{S, \iota_2} M_S$ (satisfying the cocycle condition). Since $M_S$ is $p$-torsion free and we have $S^{(1), \log} = S^{(1)} \gr{t}^{\mathrm{PD}}$, it suffices to show that after inverting $p$, the map $\delta_S$ in (\ref{eq:descent_iso_for_isocrystal}) descends to an isomorphism.  
Now, since $M_S [1/p]$ is projective, we have 
\[
 \big(M_S \widehat \otimes_{S, \iota} S^{(1), \log} \big) [1/p] \cong  M_S[1/p] \otimes_{S, \iota} S^{(1), \log}.
\]
Thus after inverting $p$, (\ref{eq:descent_iso_for_isocrystal}) becomes an isomorphism 
\[
\delta_S: M_S[1/p]  \otimes_{S, \iota_1} S^{(1), \log} \isom S^{(1), \log}  \otimes_{S, \iota_2} M_S[1/p]
\]
satisfying the cocycle conditions. Pick a Zariski cover of $\spec S[1/p]$ consisting of open subschemes $\spec S[1/p, 1/f_i]$ for finitely many $f_i \in S[1/p]$, such that each module $M_S [1/p, 1/f_i]$ is a free module over $S_i := S[1/p, 1/f_i]$. By Zariski descent it suffices to show that each isomorphism 
\[
\delta_i: M_{S_i}   \otimes_{S_i, \iota_1} S_i^{(1), \log} \isom S_i^{(1), \log}  \otimes_{S_i, \iota_2} M_{S_i} 
\]
descends, where $S_i = S[1/p, 1/f_i]$ and $M_{S_i}$ is the base change of $M_S$ along $S \ra S_i$.  For this, we note that the same argument as in the case of log crystals works \textit{verbatim}. This concludes the proof for log isocrystals. \\

\noindent \textbf{Log $F$-crystals and log $F$-isocrystals.} 

\noindent 
The assertions in the theorem also hold for log $F$-crystals and log $F$-isocrystals, as Frobenius is irrelevant in the arguments above. This finishes the proof of Part (1) of the theorem.\\

\noindent \textbf{Constancy of nilpotent rank.} 

\noindent Now we prove Part (2) of the theorem. Let us retain the setup from the case of log isocrystals above. From (\ref{eq:cocycle_over_general_S_for_N1}) we may deduce that 
\[
N_1 (\beta, \beta) = N_0 (\beta, \alpha) \cdot N_1(\alpha, \alpha) \cdot  N_0 (\beta, \alpha)^{-1}
\]
for all $\alpha, \beta: S \ra W(\cl k)$ such that $\cl \alpha = \cl \beta$. Therefore, the matrix $N_1 (\beta, \beta)^m$ is conjugate to $ N_1(\alpha, \alpha)^m $ for all $m$. By our argument of Part (1), we know that for a fixed integer $m \ge 1$, either $N_1 (\alpha, \alpha)^m = 0$ for all points $\alpha: S \ra W(\cl k)$ or  $N_1 (\alpha, \alpha)^m \ne 0$ for all such points. This implies that the nilpotent rank of the monodromy operator of the log isocrystal $\mE$ is constant over $Z$. 
\end{proof}

\begin{remark}[Relation to log crystalline fundamental groups]  If we pretend that the map $Z^{\log} \ra Z$ is ``flat'' in a suitable sense, then we should expect an exact sequence 
\[ 
\pi_{1}^{\mathrm{crys}}(z^{\log}) \ra \pi_{1}^{\mathrm{crys}} (Z^{\log}) \ra \pi_{1}^{\mathrm{crys}} (Z) \ra 1
\]
of log crystalline fundamental groups, which is partially developed in \cite{logcrys_pi_1}.  In particular, a similar proof of the $l$-adic variant in the introduction should give an heuristic argument of Part (1) of Theorem \ref{theorem:log_crystal}.  Our result in fact raises the interesting question of whether such an exact sequence exists. 
\end{remark}

\newpage 
\section{\large Log (iso)crystals over semistable log schemes} \label{sec:log_crystals_on_log_schemes}

The goal of this section is to generalize the rigidity results in \S \ref{sec:log_crystals}, from the case of smooth log schemes, to log schemes of semistable type. In particular, we prove certain rigidity properties of the monodromy of log (iso)crystals over such semistable log schemes. The main results of this section are Theorem \ref{theorem:log_crystal_2} and Theorem \ref{theorem:log_crystal_2_global}, which contain Theorem \ref{theorem:log_crystals_intro_st} from the introduction as a special case. As an application, we prove a rigidity result on the crystallinity of $p$-adic local systems on a smooth rigid analytic space with semistable reduciton; in particular, we deduce Theorem \ref{thm:main_intro_for_log_schemes} in the introduction from Theorem \ref{theorem:log_crystal_2}.

\subsection{Semistable log schemes} \label{ss:st_log_scheme} For the setup, we continue to let $k$ be a perfect field of characteristic $p$. Let us fix an integer $n \ge 1$ and consider the projection map
\[ 
\pi: D^{n,\log} \ra D^{n},
\]
where $D^n$ is the scheme
\[ D^n = \spec   k[x_0, x_1,\ldots, x_n]/\prod_{i = 0}^n x_i
\] 
over $s = \spec k$, and $D^{n,\log}$ is the log scheme associated to the pre-log $k$-algebra 
\[ \N^{n+1} \ra  k[x_0, x_1, \ldots, x_n]/\prod_{i = 0}^n x_i.\] Here, the $i^{th}$-copy of $\N$ gets identified with $x_i^{\N}$ for each $i = 0, \ldots, n$. When $n$ is understood from the context, we shall simply write $D$ and $D^{\log}$ for $D^n$ and $D^{n,\log}.$  

Note that $D^{\log}$ is log smooth over the standard log point $s^{\log}=(\spec k, 0^{\N})$, where the correspondng monoid homomorphism is the diagonal map $\N \ra \N^{n+1}$ sending $1 \mapsto (1, ..., 1)$.

 \begin{notation} \label{notation:log_scheme_section}
Let us introduce the following additional notations for this section. 
 \begin{itemize} 
     \item  For every point $z:\spec k'\ra D$ on $D$, let $z^{\log}$ be the base change of $z \ra D$ along $\pi: D^{\log} \ra D$. This is similar to the convention in \S \ref{ss:notation_section_2}. 
     \item Let $D^{\mathrm{sm}}$ denote the smooth locus of $D$.
     \item  For each $i=0,\ldots, n$, let $D_i$ denote the irreducible component of $D$ given by $x_i = 0$.
     \item Let $z_* \in D$ denote the closed point given by $x_0 = x_1 = \cdots = x_n = 0$, which is the intersection of all irreducible components $D_i$ in $D$. Using the notation above,  $z_*^{\log}$ is isomorphic to the log scheme associated to the pre-log ring $\N^{n+1} \ra k$ sending every nonzero element to $0$. 
     \item For each $i=0,\ldots, n$, let $U_i$ denote the open subscheme of $D_i$ given by $x_i = 0$ and $x_j \ne 0$ for all $j \ne i$. In particular, each copy of $U_i$ is isomorphic to $\G_m^n$. We have $U_i=D_i\cap D^{\mathrm{sm}}$ and $D^{\mathrm{sm}}$ is the disjoint union of the $U_i$'s.
     \item Let $U_i^{\log}$ be $U_i$ equipped with the pullback log structure from $D^{\log}$. Note that the log structure can also be identified with the log structure obtained from the base change of $U_i \ra s$ along $s^{\log} \ra s$. In particular, if $z: \spec k'\ra U_i$ is a point on $U_i$, then $z^{\log}$ is isomorphic to a standard log point $(\spec k', 0^{\N})$.  
 \end{itemize}
\end{notation}

Let us introduce the notion \emph{log scheme of semistable type} used in this article. 
\begin{definition} \label{def:st_log_scheme}
We say that a log scheme $Z^{\log}$ over $k$ is a \textit{semistable log scheme} or has \textit{semistable type} if \'etale locally (on the underlying scheme $Z$ of the log scheme $Z^{\log}$), it is isomorphic to a log scheme strictly \'etale over $D^{n,\log} \times_{\spec k} \A^m_k$ for some $n, m \in \Z_{\ge 0}$, where $D^{n,\log} \times_{\spec k} \A^m_k$ is the fiber product of fs log schemes with $\spec k$ and $\A^m_k$ endowed with trivial log structures.
\end{definition}

\subsection{Rigidity of monodromy over semistable log schemes}

Now we state the main results of this section which contain Theorem \ref{theorem:log_crystals_intro_st} as a special case. For simplicity of exposition, let us introduce one more definition.

\begin{definition} \label{def:trivial_monodromy_in_all_directions}
Fix an integer $n \ge 1$ and let $D^{\log} = D^{n,\log}$ be as above. Let $\mE$ be a log crystal over $D^{\log}$. We say that $\mE$ \textit{has trivial monodromy at $z_*$ in all directions} if for every map $\iota: (\spec k', 0^{\N}) \ra z_*^{\log}$ from a standard log point $(\spec k', 0^{\N})$ where $k'$ is a perfect field extension of $k$, the pullback log crystal $\iota^* \mE$ of $\mE$ along the composition 
\[\iota: (\spec k', 0^{\N}) \ra z_*^{\log} \ra D^{\log} \] 
has trivial monodromy. We make similar definitions for log isocrystals and log $F$-(iso)crystals. 
\end{definition}

\begin{theorem} \label{theorem:log_crystal_2}
 Let $\mE$ be a log (iso)crystal over $D^{\log}$. Then the following are equivalent. 
 \begin{enumerate}
     \item There is a finite collection of close points $z_i \in U_i$, one for each $i=0,\ldots, n$, such that the restriction of $\mE$ at each $z_i^{\log}$ has trivial monodromy. 
     \item For each $i = 0,..., n$, the restriction of $\mE$ to $U_i^{\log}$ descends to an (iso)crystal on $U_i$. 
     \item $\mE$ has trivial monodromy at $z_*$ in all directions. 
     \item The restriction of $\mE$ at the log point $z_*^{\log}$ descends to an (iso)crystal on $z_*$. 
     \item For every map $\iota: (\spec k', 0^{\N})\ra D^{\log}$ from a standard log point $(\spec k', 0^{\N})$, the pullback log (iso)crystal $\iota^*\mE$ has trivial monodromy. 
     \item $\mE$ descends to an (iso)crystal over $D$.  
 \end{enumerate}
 The same assertion holds for log $F$-(iso)crystals. 
\end{theorem}

More generally, we have the following global  version of rigidity of log (iso)crystals over semistable log schemes. 

\begin{theorem}\label{theorem:log_crystal_2_global}
Let $Z^{\log}$ be a geometrically connected semistable log scheme over $k$ with underlying scheme $Z$. Let $Z^{\mathrm{sm}}\subset Z$ be the smooth locus and let $U_0, \ldots, U_r$ be the irreducible components of $Z^{\mathrm{sm}}$. Let $\mE$ be a log (iso)crystal over $Z^{\log}$. Then the following are equivalent.
\begin{enumerate}
    \item There is a finite collection of close points $z_i \in U_i$, one for each $i=0,\ldots, r$, such that the restriction of $\mE$ at each $z_i^{\log}$ has trivial monodromy. 
    \item For every closed point $z\in Z^{\mathrm{sm}}$, the restriction of $\mE$ at $z^{\log}$ has trivial monodromy.
    \item For every map $\iota: (\spec k', 0^{\N})\ra Z^{\log}$ from a standard log point $(\spec k', 0^{\N})$, the pullback log (iso)crystal $\iota^*\mE$ has trivial monodromy. 
    \item $\mE$ descends to an (iso)crystal over $Z$.  
\end{enumerate}
The same assertion holds for log $F$-(iso)crystals. 
\end{theorem}

By Definition \ref{def:st_local_sys}, Theorem \ref{theorem:log_crystal_2} immediately implies Theorem \ref{thm:main_intro_for_log_schemes} from the introduction; while Theorem \ref{theorem:log_crystal_2_global} implies a natural generalization of Theorem \ref{thm:main_intro_for_log_schemes} (see Theorem \ref{cor:rigidity_crystalline_over_semistable}). Moreover, 
as indicated in the introduction, these rigidity results play an essential role in the proof of Shankar's conjecture (Theorem \ref{thm:conjecture_for_projective_varieties_intro}) in \S \ref{sec:conjecture}, as well as the proof of Theorem \ref{theorem:main_intro_punctured_disc} in \S \ref{section:punctured disc}.

\subsection{Log (iso)crystals over semistable log schemes: Part I} \label{ss:log_F_isocrystal_semistable_log_schemes}

We start to prove Theorem \ref{theorem:log_crystal_2}. First note that the equivalence between $(1)$ and $(2)$ follows from Theorem \ref{theorem:log_crystal_1}. Let us also note that $(4)$ implies $(3)$;  $(5)$ implies $(1)$ and $(3)$; and clearly $(6)$ implies all other claims $(1) -(5)$. Thus it remains to show that $(1)\so (6)$ and $(3) \so (4) \so (1)$. For simplicity let us assume $n = 1$ in the proof, since the additional complication in the general case is again entirely notational as in \S \ref{sec:log_crystals}. To further ease notations, let us rename the variables $x_0, x_1$ by $x, y$ in the proof, so $D = \spec k[x, y]/xy$.  

In this Part I, we prove $(1) \so (6)$.

\begin{proof}[Proof of $(1) \Rightarrow (6)$ in Theorem \ref{theorem:log_crystal_2}] \indent 

\noindent Let us start with the case of a log crystal. (The case of a log isocrystal is similar, see the last paragraph of this subsection.) The proof is similar to the proof of Theorem \ref{theorem:log_crystal_1} (and in fact builds upon it). Let 
\[
A^{(0)} = W \gr{x, y} \gr{xy}^{\mathrm{PD}} = W \gr{x, y} \gr{w}^{\mathrm{PD}}/(xy - w)
\]
and equip it with the pre-log structure $\N^{\oplus 2} \ra A^{(0), \log}$ sending $(a, b) \mapsto x^a y^b$. We denote the corresponding (pre-)log ring by $(A^{(0)}, x^{\N}\oplus y^{\N})$. Then $A^{(0)}$ (resp. $(A^{(0)}, x^{\N}\oplus y^{\N})$) is a weakly initial object in (the opposite category of) the $p$-completed affine crystalline site $(k[x, y]/xy)_{\mathrm{crys}}^{\mathrm{aff},\wedge, \mathrm{op}}$ (resp. (the opposite category of) the $p$-completed affine log crystalline site $(k[x, y]/(xy), x^{\N}\oplus y^{\N})_{\mathrm{crys}}^{\mathrm{aff},\wedge, \mathrm{op}}$). Let $A^{(1)}$ (resp. $A^{(2)}$) be the self-coproduct (resp. triple self-coproduct) of $A^{(0)}$ in  $(k[x, y]/xy)_{\mathrm{crys}}^{\mathrm{aff},\wedge, \mathrm{op}}$, and let  $A^{(1), \log}$ (resp. $A^{(2), \log}$) be the underlying ring of the self-coproduct (resp. triple self-coproduct) of $(A^{(0)}, x^{\N}\oplus y^{\N})$ in $(k[x, y]/(xy), x^{\N}\oplus y^{\N})_{\mathrm{crys}}^{\mathrm{aff},\wedge, \mathrm{op}}$. Consider the simplicial object 
\begin{equation*}  
A^{(0)} \mathrel{\substack{\textstyle\longrightarrow\\[-0.6ex] \textstyle\longrightarrow}}  
A^{(1), \log}  \mathrel{\substack{\textstyle\longrightarrow\\[-0.6ex] \textstyle\longrightarrow \\[-0.6ex] \textstyle\longrightarrow}} A^{(2), \log}  
\end{equation*} where the arrows are labeled by $\iota_1, \iota_2$ and $\iota_{12}, \iota_{13}, \iota_{23}$, and similarly for the non-log version. As in the proof of Theorem \ref{theorem:log_crystal_1}, the data of a log crystal $\mE$ over $(k[x, y]/xy, x^{\N}\oplus y^{\N})$ is equivalent to a finite projective module $M$ over $A^{(0)}$ equipped with 
an $A^{(1), \log}$-linear isomorphism 
\begin{equation}
\label{eq:descent_over_A1_log_weakly_initial_for_log_st_scheme}
\delta: M \otimes_{A^{(0)}, \iota_1} A^{(1), \log} \isom A^{(1), \log} \otimes_{A^{(0)}, \iota_2} M 
\end{equation}
which 
satisfies the cocycle condition $\iota_{23}^* (\delta) \circ \iota_{12}^* (\delta) = \iota_{13}^* (\delta)$. A similar assertion holds for (non-log) crystals over $k[xy]/xy$. 

Again as in the proof of Theorem \ref{theorem:log_crystal_1}, we treat the case when $M$ is free (for the general case of log crystals, we cover $\spec k[xy]/xy$ by affine open subschmes and argue in the same way as in Part I\!I of \S \ref{ss:log_F_crystal_over_smooth_schemes}). Our goal is to show that, under Condition (1) (or equivalently, Condition (2)), the isomorphism $\delta$ descends to an isomorphism over $A^{(1)}$. Now let us observe that 
\begin{equation} \label{eq:A1_log_for_st_log_scheme}
    A^{(1), \log} = W \gr{x_1, y_1} \gr{\delta_x, \delta_y, x_1y_1}^{\mathrm{PD}}
\end{equation}
where we view $\delta_x$ as ``$ {x_2}/{x_1} -1$'' and $\delta_y$ as ``$ {y_2}/{y_1} -1$'',  and it is equipped with a natural pre-log structure 
\[ \alpha^{(1)}: N^{(1)} \ra A^{(1), \log}\]  with $N^{(1)} = \{(a_1, a_2, b_1, b_2) \in \Z^{4} \:|\:  a_1 + a_2 \ge 0, \,\,b_1 + b_2 \ge 0 \}$. On the other hand, the (non-log) self-coproduct is given by 
\[
A^{(1)} = W \gr{x_1, y_1} \gr{d_{x}, d_{y}, x_1y_1}^{\mathrm{PD}}
\]
where $d_x = x_2 - x_1$ and $d_y = y_2 - y_1$. The natural map $A^{(1)} \ra A^{(1), \log}$ is given by sending $d_x \mapsto x_1 \cdot \delta_x$ and $d_y \mapsto y_1 \cdot \delta_y$. Upon choosing a basis of $M$ over $A^{(0)}$, we may represent the isomorphism $\delta$ as a matrix
\begin{equation} \label{eq:expression_of_delta_normal_crossing}
T_{\delta} = \sum_{i, j \ge 0} F_{ij} (x_1, y_1) \cdot (\delta_x)^{i} (\delta_y)^{j}
\end{equation}
with each $F_{ij}=F_{ij}(x_1, y_1)$ a matrix with entries in $\frac{1}{i!j!} W\gr{x_1, y_1} \gr{x_1 y_1}^{\mathrm{PD}}$. The assertion that the log crystal $\mE$ descends to a crystal over $k[x, y]/xy$ is equivalent to the following condition:
\begin{itemize}
    \item For each $i, j \ge 0$, the matrix $F_{ij}$ is divisible by $(x_1)^i (y_1)^j$. In other words, each $F_{ij} (x_1, y_1)$ can be written as $(x_1)^i (y_1)^j \sq F_{ij}(x_1, y_1)$ for some matrix $\sq F_{ij}({x_1, y_1})$ with entries in $\frac{1}{i!j!}W\gr{x_1, y_1} \gr{x_1 y_1}^{\mathrm{PD}}$.   
\end{itemize}  
 We will prove this via the following two steps: 
\begin{enumerate}
     \item First, we will show that $F_{ij}$ is divisible by $x_1$ for all $i \ge 1$ and by $y_1$ for all $j \ge 1$. For this, we will study certain specialization maps on the log crystal. 
    \item Second, we will finish the argument by showing that $F_{ij}$ is divisible by $(x_1)^i$ and by $(y_1)^j$. For this, we
    induct on $i+j$ and make full use of the cocycle condition. \\
\end{enumerate}

\noindent \subsubsection*{Step 1} 
Let us restrict the log crystal to the locally closed subschemes $U_0$ and $U_1$ of $D$, which respectively correspond to the maps 
\begin{align} \label{eq:map_from_ncd_to_one_branch}
   & i_0: k[x, y]/xy \xrightarrow{x \mapsto 0} k[y] \hookrightarrow k[y^{\pm 1}] \\ \nonumber 
   & i_1: k[x, y]/xy \xrightarrow{y \mapsto 0} k[x] \hookrightarrow k[x^{\pm 1}]. 
\end{align}
Consider the objects 
\[ (W \gr{y^{\pm 1}}, 0^\N \oplus y^\Z) \quad  \textup{  and  } \quad  (W \gr{x^{\pm 1}}, 0^\N \oplus x^\Z) \]  in the $p$-completed affine log crystalline site of $(k[y^{\pm 1}], 0^\N)$ and of $(k[x^{\pm 1}], 0^\N)$, respectively. We can view them as objects in the $p$-completed affine log crystalline site of $(k[x, y]/xy, x^{\N}\oplus y^{\N})$ via the maps $i_0$ and $i_1$ above.

Let $A^{(1), \log}_0$ (resp. $A^{(1), \log}_1$) denote the underlying ring of the self-coproduct of the object $(W \gr{y^{\pm 1}}, 0^\N \oplus y^\Z)$ (resp. of $(W \gr{x^{\pm 1}}, 0^\N \oplus x^\Z)$) in the $p$-completed affine log crystalline site of $(k[y^{\pm 1}], 0^\N)$ (resp. of  $(k[x^{\pm 1}], 0^\N)$). Similarly construct $A^{(2), \log}_0$ and $A^{(2), \log}_1$ from the triple self-coproducts. Note that  we have 
\begin{align*}
    A^{(1), \log}_0 & = W \gr{y_1^{\pm 1}} \gr{d_y, t_y}^{\mathrm{PD}}   \\  A^{(1), \log}_1 & = W \gr{x_1^{\pm 1}} \gr{d_x, t_x}^{\mathrm{PD}},
\end{align*}
where $d_y = y_2 - y_1$ and $d_x = x_2 - x_1$ as before.\footnote{This is similar to the setup as in the proof of Theorem \ref{theorem:log_crystal_1}, although there we denote $d_x$ and $d_y$ by $x_1'$ and $y_1'$, respectively. Also note that the variables $t_y$ and $t_x$ come from the exactification of the pre-log structures.}
The maps $i_0$ and $i_1$ in (\ref{eq:map_from_ncd_to_one_branch}) induce natural maps $ \iota_0: A^{(1), \log} \ra A^{(1), \log}_0$ and $\iota_1: A^{(1), \log} \ra A^{(1), \log}_1$ on the self-coproduct of the weakly initial objects in the relevant log crystalline sites, which can be described as follows. 
First, $\iota_0$ is the map 
\begin{equation} \label{eq:i_0_on_self_product}
    \iota_0: W \gr{x_1, y_1} \gr{\delta_x, \delta_y, x_1y_1}^{\mathrm{PD}} \ra W \gr{y_1^{\pm 1}} \gr{d_y, t_y}^{\mathrm{PD}}
\end{equation}
which sends 
\[
x_1 \mapsto 0, \quad y_1 \mapsto y_1, \quad  \delta_x \mapsto t_y, \quad \delta_y \mapsto d_y/y_1. 
\]
Similarly, $\iota_1$ is the map 
\begin{equation} \label{eq:i_1_on_self_product}
    \iota_1: W \gr{x_1, y_1} \gr{\delta_x, \delta_y, x_1y_1}^{\mathrm{PD}} \ra W \gr{x_1^{\pm 1}} \gr{d_x, t_x}^{\mathrm{PD}}
\end{equation}
sending 
\[
x_1 \mapsto x_1, \quad y_1 \mapsto 0, \quad  \delta_x \mapsto d_x/x_1, \quad \delta_y \mapsto t_x. 
\]
Now, by assumption (Condition (1), equivalently Condition (2)), we know that the pullback $\iota_0^* \mE$ as a log crystal over $\spec (k[y^{\pm 1}], 0^\N)$ descends to a (non-log) crystal over $\spec k[y^{\pm 1}]$, and likewise for $\iota_1^* \mE$. By the proof of Theorem \ref{theorem:log_crystal_1}, this implies that the image of $T_{\delta}$ in (\ref{eq:expression_of_delta_normal_crossing}) along $\iota_0$ in (\ref{eq:i_0_on_self_product}) has coefficients equal to $0$ in front of all positive powers of $(t_y)^i$. In other words, if we write 
\[ \iota_0 (T_\delta) = \sum_{i \ge 0} N_i \cdot (t_y)^i
\] as in (\ref{eq:expression_of_delta_affine}), then $N_i = 0$ for all $i \ge 1$. This implies that for each $i \ge 1$, we have  
\[ 
\sum_{j \ge 0} F_{ij} (0, y_1) \cdot \frac{(d_y)^j}{(y_1)^j} = 0.
\] 
This in turn implies that $F_{ij}(0, y_1) = 0$ for all $i \ge 1$ and all $j$, from which we deduce that $x_1$ divides $F_{ij}$ for each $i \ge 1$. Similarly, by considering $\iota_1$ in place of $\iota_0$, we know that $y_1$ divides $F_{ij}$ for each $j \ge 1$. This finishes the first step of the argument.

\noindent \subsubsection*{Step 2}  
Now suppose that $F_{ij}$ is divisible by $(x_1)^i (y_1)^j$ for all $i, j$ such that $i + j \le k$ for some positive integer $k$, we will show that this continues to hold for all $i, j$ such that $i + j = k + 1$, which will finish the proof by induction. 
To this end, let us first examine the cocycle condition on $T_{\delta}$. Write $\delta_{x_1} = x_2/x_1 - 1$, $\delta_{x_2} = x_3/x_2 - 1$, and similarly for $\delta_{y_1}, \delta_{y_2}$. Similar to (\ref{eq:cocycle_in_coordinates}),  the cocyle condition  in our setup becomes 
\begin{align}
 \label{eq:log_scheme_cocycle_part_1}     \Big( 
\sum_{i, j} F_{ij} \big(x_1+x_1\delta_{x_1}, y_1+y_1\delta_{y_1} \big) \cdot \big(\delta_{x_2} \big)^i  \cdot \big( \delta_{y_2} \big)^j
\Big) \cdot 
 \Big( 
\sum_{i, j} F_{ij} \big(x_1, y_1 \big) \cdot \big(\delta_{x_1} \big)^i \cdot \big( \delta_{y_1} \big)^j
\Big) \\ 
 \label{eq:log_scheme_cocycle_part_2}     
       = \sum_{i, j} F_{ij} \big(x_1, y_1\big) \cdot \big( \delta_{x_1} + \delta_{x_2} + \delta_{x_1} \delta_{x_2} \big)^i \cdot \big( \delta_{y_1} + \delta_{y_2} + \delta_{y_1} \delta_{y_2} \big)^j. 
\end{align}  
By considering the constant term (with respect to variables $\delta_{x_1}, \delta_{x_2}, \delta_{y_1}, \delta_{y_2}$), we deduce that $F_{00} = 1$ is the identity matrix.  Let us first check that $(y_1)^{j}$ divides $F_{0j}$ for $j = k+1$ and that  $(x_1)^{i}$ divides $F_{i0}$ for $i = k+1$. In fact we will only prove the divisibility for $(y_1)^j$, the divisibility for $(x_1)^i$ can be argued similarly (or by symmetry). For this divisibility, we look at the coefficient of the term $\delta_{y_1} (\delta_{y_2})^{k}$. Let us observe from the right hand side (\ref{eq:log_scheme_cocycle_part_2}) of the equation, this term is only involved with $i = 0, j = k$ or when $i = 0, j = k+1$, and the terms contributing to the coefficient of $\delta_{y_1} (\delta_{y_2})^{k}$ are 
\begin{equation} \label{eq:the_coefficient_of_1k} 
k \cdot F_{0k} (x_1, y_1) \cdot \delta_{y_1} (\delta_{y_2})^{k} + (k+1) \cdot F_{0, k+1} (x_1, y_1)  \cdot \delta_{y_1} (\delta_{y_2})^{k}.
\end{equation}
Now we look at the left hand side (\ref{eq:log_scheme_cocycle_part_1}). Since the term $\delta_{y_1} (\delta_{y_2})^{k}$ does not have positive powers of $\delta_{x_1}$ or $\delta_{x_2}$, so the contribution to the coefficient of $\delta_{y_1} (\delta_{y_2})^{k}$ from the first factor of expression (\ref{eq:log_scheme_cocycle_part_1}) can only come from $F_{0k} (x_1 + x_1 \delta_{x_1}, y_1 + y_1 \delta_{y_1}) \cdot (\delta_{y_2})^k$. Even more precisely, the only contribution to the coefficient of $\delta_{y_1} (\delta_{y_2})^{k}$ from the first factor of (\ref{eq:log_scheme_cocycle_part_1}) must come from 
\[F_{0k} (x_1, y_1 + y_1 \delta_{y_1}) \cdot (\delta_{y_2})^k\] 
since the difference 
\[ F_{0k} (x_1 + x_1 \delta_{x_1}, y_1 + y_1 \delta_{y_1}) - F_{0k} (x_1, y_1 + y_1 \delta_{y_1})\]
involves the term $\delta_{x_1}$. Next, the contributions to the coefficient of $\delta_{y_1} (\delta_{y_2})^{k}$ from the second factor of (\ref{eq:log_scheme_cocycle_part_1}) are two-fold: it either comes from $F_{00} (x_1, y_1) = 1$ or comes from $F_{01} (x_1, y_1) \cdot \delta_{y_1}$. In other words, the expression (\ref{eq:the_coefficient_of_1k}) is equal to the $\delta_{y_1} (\delta_{y_2})^k$-term in 
\[
\Big( F_{0k} (x_1, y_1 + y_1 \delta_{y_1}) \cdot (\delta_{y_2})^k \Big) \cdot \Big(1 + F_{01}(x_1, y_1) \cdot \delta_{y_1} \Big).
\]
By assumption, we know that $(y_1)^k$ divides $F_{0k}$, so we may write 
\[ F_{0k} (x_1, y_1) = (y_1)^k \cdot \sq F_{0k} (x_1, y_1).
\] 
Therefore, $(k+1) \cdot F_{0, k+1} (x_1, y_1)$ is equal to the coefficient of the $\delta_{y_1} (\delta_{y_2})^k$-term in 
\begin{align} \nonumber 
& \Big( F_{0k} (x_1, y_1 + y_1 \delta_{y_1}) \cdot (\delta_{y_2})^k \Big) \cdot \Big(1 + F_{01}(x_1, y_1) \cdot \delta_{y_1} \Big)  \\ 
& \nonumber - k \cdot 
F_{0k} (x_1, y_1) \cdot \delta_{y_1} (\delta_{y_2})^{k} \\ 
\label{eq:expression_involving_star_log_ncr_case}
= \quad & (\star)   \: + \: F_{0k} (x_1, y_1 + y_1 \delta_{y_1}) \cdot F_{01} (x_1, y_1) \cdot \delta_{y_1} (\delta_{y_2})^k  \\ \nonumber &  -   k \cdot 
F_{0k} (x_1, y_1) \cdot \delta_{y_1}, (\delta_{y_2})^{k} 
\end{align}
where we use $(\star)$ to denote the expression 
\[
(\star) := (y_1)^k \cdot \big(1+ \delta_{y_1}\big)^k \cdot \sq F_{0k} (x_1, y_1 + y_1 \delta_{y_1}) \cdot (\delta_{y_2})^k.
\]
To analyze the coefficient in front of $\delta_{y_1} (\delta_{y_2})^k$, we need to extract precisely one copy of $\delta_{y_1}$ in $(y_1)^k \cdot \big(1+ \delta_{y_1}\big)^k \cdot \sq F_{0k} (x_1, y_1 + y_1 \delta_{y_1}) $. Now observe that we can write 
\begin{align*}
    \sq F_{0k} (x_1, y_1 + y_1 \delta_{y_1}) = &  \sq F_{0k} (x_1, y_1) + H (x_1, y_1) \cdot (y_1 \delta_{y_1}) \\ & +  \textup{higher order terms in } (y_1 \delta_{y_1}). 
\end{align*}
Therefore, the $\delta_{y_1} (\delta_{y_2})^k$-term in $(\star)$ is precisely 
\begin{align} \nonumber 
& (y_1)^k  \cdot 1 \cdot H(x_1, y_1)  (y_1 \delta_{y_1}) \cdot  (\delta_{y_2})^k  +
(y_1)^k \cdot k \delta_{y_1} \cdot \sq F_{0k} (x_1, y_1) \cdot   (\delta_{y_2})^k  \\ 
= \quad & (y_1)^{k+1} H(x_1, y_1) \cdot  \delta_{y_1}  (\delta_{y_2})^k  + k \cdot F_{0k} (x_1, y_1) \cdot  \delta_{y_1}  (\delta_{y_2})^k. \label{eq:the_term_of_interest_in_star}
\end{align}
Substituting (\ref{eq:the_term_of_interest_in_star}) back into (\ref{eq:expression_involving_star_log_ncr_case}), we see that $(k+1) \cdot F_{0, k+1} (x_1, y_1)$ is equal to the coefficient of the $\delta_{y_1} (\delta_{y_2})^k$-term in 
\[
 (y_1)^{k+1} H(x_1, y_1) \cdot  \delta_{y_1}  (\delta_{y_2})^k  +  F_{0k} (x_1, y_1 + y_1 \delta_{y_1}) \cdot F_{01} (x_1, y_1) \cdot \delta_{y_1} (\delta_{y_2})^k,
\]
which is indeed divisible by $(y_1)^{k+1}$ by the induction hypotheses.

It remains to check (under the same induction hypotheses) that $(x_1)^i (y_1)^j$ divides $F_{ij}(x_1, y_1)$ with $i+j = k+1$ and $i, j \ge 1$. This is in fact more straightforward. For example, we may look at the coefficient of the term $(\delta_{x_1})^i (\delta_{y_2})^j$ in the cocycle condition. The only contribution from the right hand side (\ref{eq:log_scheme_cocycle_part_2}) comes from $F_{ij} (x_1, y_1) \cdot (\delta_{x_1})^i (\delta_{y_2})^j$. From the left hand side (\ref{eq:log_scheme_cocycle_part_1}), the contribution comes from 
\[ F_{0j} (x_1 + x_1 \delta_{x_1}, y_1) \cdot (\delta_{y_2})^j\] in the first factor, and from 
\[ F_{i'0}(x_1, y_1) \cdot (\delta_{x_1})^{i'}\] in the second factor for all $i' \le i$. For each such $i'$, 
the first factor contributes $(x_1)^{i - i'}(y_1)^j$, while the second factor is divisible by $(x_1)^{i'}$, thus $F_{ij} (x_1, y_1)$ is indeed divisible by $(x_1)^i (y_1)^j$. This proves Step 2, and thus finishes the proof of the implication $(1) \so (6)$ in the case of log crystals.

Finally, let us remark that, for log isocrystals, we apply a similar argument as in the proof of Theorem \ref{theorem:log_crystal_1}. The assertions for log $F$-crystals and log $F$-isocrystals follow from the non-Frobenius versions, since the Frobenius structures are again irrelevant in the argument, as in the proof of Theorem \ref{theorem:log_crystal_1}. 
\end{proof}

\subsection{Log (iso)crystals over semistable log schemes: Part I\!I (log points revisited)}  

Next we prove the equivalence between $(3)$ and $(4)$, which is a generalization of Lemma \ref{lemma:descent_over_log_point} and is interesting in its own right. It suffices to prove the following lemma. 

\begin{lemma}\label{lemma:equivalence_3_4_in_theorem_log_st} Consider a log point $s_n^{\log}=\spec k$ equipped with the log structure associated with the pre-log structure $\N^{n+1} \ra k$ sending every nonzero element to $0$.\footnote{There is an isomorphism $z_*^{\log}\cong s_n^{\log}$ where $z_*^{\log}$ is as in Notation \ref{notation:log_scheme_section}.} 
Let $\mE$ be a log crystal on $s_n^{\log}$. Suppose that for every map $\iota: s^{\log} \ra s_n^{\log}$ from the standard log point $s^{\log} = (\spec k, 0^\N)$, the pullback log crystal $\iota^* \mE$ has trivial monodromy, then $\mE$ descends to a crystal over $s=\spec k$. The same assertion holds for log isocrystals and log $F$-(iso)crystals. 
\end{lemma}

\begin{proof} Again, as in the proof of Theorem \ref{theorem:log_crystal_2}, we only treat the case when $n = 1$ for notational simplicity. Consider the element $A^{(0), \log} = (W, 0^{\N \oplus \N})$ in the $p$-completed affine log-crystalline site of $s_1^{\log}$. Its self-coproduct is given by 
\[ 
A^{(1), \log} = W \gr{u_1, u_2, v_1/u_1 -1, v_2/u_2 - 1}^{\mathrm{PD}}/(u_1, u_2)^{\mathrm{PD}} = W \gr{t_1, t_2}^{\mathrm{PD}}.
\]
Here, our choice of notation is similar to the notation in the proof of Lemma \ref{lemma:descent_over_log_point}, where $t_1, t_2$ stands for $v_1/u_1 - 1$ and $v_2 /u_2 -1$, respectively. As in the previous arguments, 
upon choosing a basis, the descent isomorphism can be written as a matrix of the form 
\[ T_{\delta} = \sum_{i, j \ge 0 } N_{ij} (t_1)^i (t_2)^j
\]  for some matrices $N_{ij}$ with entries in $\frac{1}{i! j!} W$. We will show that, under the hypothesis of the lemma, all the $N_{ij} = 0$ except when $i = j = 0$.  To this end,  let $\iota_{m, n}$ denote the map $s^{\log} \ra s_1^{\log}$ induced from the map $\N^{\oplus 2} \ra \N$ on monoids sending $(a, b) \mapsto ma + nb$, where $(m, n) \in (\Z_{\ge 0})^{\oplus 2}$ is a pair of non-negative integers. 

Now, consider $B^{(0), \log} = (W, 0^{\N})$ in (the opposite category of) the $p$-completed affine log-crystalline site of $s^{\log}$, then its self-coproduct is $B^{(1), \log} = W \gr{t}^{\mathrm{PD}}$ as we have seen from the proof of Lemma \ref{lemma:descent_over_log_point}. Tracing through the constructions, we observe that the map $\iota_{m, n}$ induces a map\footnote{which we again denote by $\iota_{m, n}$ by a slight abuse of notation. } 
\[
\iota_{m, n}:  W \gr{t_1, t_2}^{\mathrm{PD}} \ra W \gr{t}^{\mathrm{PD}} \]
between $A^{(1), \log}$ and $B^{(1), \log}$, sending 
\[
t_1 \mapsto (t+1)^m - 1, \quad t_2 \mapsto (t+1)^n - 1.
\]
By assumption, we know that 
\begin{equation} \label{eq:expression_of_iota_mn}
\iota_{m, n} (T_{\delta}) = \sum_{ij} N_{ij} (t^m + m t^{m-1} + \cdots + mt)^i  (t^n + n t^{n-1} + \cdots + nt)^j = 1
\end{equation} for all pairs of 
$(m, n)$. The claim that all $N_{ij} = 0$ for $i+ j \ge 0$ follows from this using an induction argument on $i + j$. For the induction step, let us first note that, considering the linear term in (\ref{eq:expression_of_iota_mn}) implies that 
\[
m N_{10} + n N_{01} = 0
\] 
for all $m, n$, thus we have $N_{ij} = 0$ when $i + j = 1$. Now suppose that the claim holds up to $i + j = k-1$ for some $k \ge 2$, then we know from analyzing the coefficient of $t^{k}$ in (\ref{eq:expression_of_iota_mn}) that 
\[
m^{k} N_{k0} + m^{k-1} n N_{k-1, 1} + \cdots + n^k N_{0k} = 0
\]
for all pairs $(m,n)$ of positive integers (using the fact that we already know $N_{ij} = 0$ for $i + j\le k -1$), which then implies that 
\[ N_{k0} = N_{k-1, 1} = \cdots = N_{0k } = 0. \]  This finishes the proof of the lemma. 
\end{proof}

\subsection{Log (iso)crystals over semistable log schemes: Part I\!I\!I} \label{ss:log_F_isocrystal_semistable_log_schemes_part_2}

Let us consider a toy example before we handle the final implication $(4)\so (1)$ in the proof of Theorem \ref{theorem:log_crystal_2}. 

\begin{lemma} \label{lemma:descent_for_log_A1}
Let $\A^{1, \log}_k$ be the log scheme associated with the pre-log ring $(k[x], x^\N)$. It comes with a natural map $\A^{1, \log}_k \ra \A^{1}_k$. Let $z_0 \in \A^1_k$ denote the origin of $\A^1_k$ defined by $x = 0$ and let $z_0^{\log}$ be the pullback of $z_0$ along $\A^{1, \log}_k \ra \A^{1}_k$. In particular, $z_0^{\log}$ is isomorphic to the standard log point $(\spec k, 0^{\N})$. Then a log (iso)crystal $\mE$ on $\A^{1, \log}_k$ descends to an (iso)crystal over $\A^1_{k}$ if and only if it does so when restricted to $z_0^{\log}$, that is, if and only if it has trivial monodromy at $z_0^{\log}$. 
\end{lemma}

The argument is very similar to the proof in \S \ref{ss:log_F_isocrystal_semistable_log_schemes}, so we only provide a sketch here (in the case of log crystals) and invite the reader to fill in the rest of the details. 
\begin{proof} 
Consider the weakly initial object $A^{(0)} = W \gr{x}$ in (the opposite category of) the $p$-completed affine crystalline site of $k[x]$, which also gives rise to a weakly initial object $A^{(0), \log} = (A^{(0)}, x^{\N})$ in (the opposite category of) the $p$-completed affine log crystalline site of $(k[x], x^{\N})$. As in the previous arguments, we form the self-coproducts $A^{(1)}$, $A^{(1), \log}$ and the triple self-coproducts $A^{(2)}$, $A^{(2), \log}$. We compute that 
\[ 
A^{(1)} = W \gr{x_1} \gr{d_x}^{\mathrm{PD}}   \quad \text{ and} \quad   
A^{(1), \log} = W \gr{x_1} \gr{\delta_x}^{\mathrm{PD}},
\]
where $\delta_x = x_2/x_1 -1 $ and $d_x = x_2 - x_1$. The natural map $A^{(1)} \ra A^{(1), \log} $ is given by $d_x \mapsto x_1 \delta_x$. As before, the data of a log crystal $\mE$ corresponds to a finite projective $W\gr{x}$-module equipped with a descent isomorphism  $\delta$ over $A^{(1), \log}$ that satisfies the cocycle condition.\footnote{Similar as before, in the case of log isocrystals, the data of a log isocrystal $\mE$ gives rise to a finite projective module over $K_0\gr{x} = W \gr{x} [1/p]$.
} 
Our goal is to show that the descent isomorphism actually lives over $A^{(1)}$ if $\mE$ has trivial monodromy over $z_0^{\log}$. As in the proof of Theorem \ref{theorem:log_crystal_1}, we reduce to the case when $M$ is free and $\delta$ is represented by a matrix (upon choosing a basis of $M$) of the form 
\[ T_{\delta} = N_0 (x_1) + N_1 (x_1) \cdot \delta_x + N_2 (x_1) \cdot \delta_x^2 + \cdots 
\] 
where each $N_i$ is a matrix with entries in $\frac{1}{i!}W \gr{x_1}$. We want to show that each $N_i (x_1)$ is divisible by $(x_1)^i$. By the assumption that $\mE$ has trivial monodromy at $z_0^{\log}$ and (the proof of) Lemma \ref{lemma:descent_over_log_point}, we know that $N_i(x_1)$ is divisible by $x_1$ for all $i \ge 1.$ To prove the stronger claim that $N_i (x_1)$ is divisible by $(x_1)^{i}$, we use the cocycle condition on $T_{\delta}$ and apply an induction on $i$. The argument is similar to the final step in the proof of  $(1) \so (6)$ in Theorem \ref{theorem:log_crystal_2} given above. 
\end{proof}

Now we are ready to prove $(4)\so (1)$ in Theorem \ref{theorem:log_crystal_2} (and hence complete the proof of Theorem \ref{theorem:log_crystal_2}), which uses a slight variant of the lemma above. 

\begin{proof}[Proof of $(4) \Rightarrow (1)$ in Theorem \ref{theorem:log_crystal_2}] \indent 

\noindent We will treat the case of log crystals. The cases of log isocrystals and their Frobenius variants will follow from a similar argument as in the proof of Theorem \ref{theorem:log_crystal_1}.  Without loss of generality, let $z$ be a closed point on $U_1$ (see Notation \ref{notation:log_scheme_section}). We want to show that $\mE$ has trivial monodromy at $z$. Note that $U_1= \spec k[x^{\pm 1}]\cong \G_m$ is the subscheme of $D$ given by $x \ne 0, y = 0$ and $D_1 = \spec k[x]$ is a copy of $\A^1_{k}$. Let $D_1^{\log}$ be $D_1$ equipped with the pullback log structure from $D^{\log}$, namely, it is the log scheme associated to the log $k$-algebra $(k[x], x^{\N} \oplus 0^{\N})$. Let $D_1^{\partial\mathrm{-}\!\log}$ denote the log scheme associated to the log algebra $(k[x], 0^\N)$. Note that the natural map $D_1^{\log} \ra D_1^{\partial\mathrm{-}\!\log}$ agrees with the base change of the map $\A^{1, \log}_k \ra \A^1_k$ in Lemma \ref{lemma:descent_for_log_A1} along the natural projection $D_1^{\partial\mathrm{-}\!\log}  \ra \A^1_k$.

Let $\mE_1$ denote the restriction of $\mE$ over $D_1^{\log}$. We will first show that, under Condition (4), the log crystal $\mE_1$ descends to a log crystal $\mE_1^{\partial}$ on $ D_1^{\partial\mathrm{-}\!\log}$. This follows from a similar argument as the proof of Lemma \ref{lemma:descent_for_log_A1} above. The difference is that, now we consider the weakly initial objects $A^{(0)} = W \gr{x} \gr{u}^{\mathrm{PD}}$ and $A^{(0), \log} = (A^{(0)}, x^{\N} \oplus u^{\N})$ in $(k[x], 0^{\N})^{\mathrm{aff}, \wedge, \mathrm{op}}_{\mathrm{crys}}$ and $(k[x], x^{\N}\oplus 0^{\N})^{\mathrm{aff}, \wedge, \mathrm{op}}_{\mathrm{crys}}$, respectively. In turn, we compute the self-coproducts 
\begin{align*}
A^{(1)} & = W \gr{x_1} \gr{u_1, d_x, t}^{\mathrm{PD}} \qquad  \\  
A^{(1), \log} &= W \gr{x_1} \gr{u_1, \delta_x, t}^{\mathrm{PD}} 
\end{align*}
where $d_x = x_2 - x_1$, $\delta_x = x_2/x_1 - 1$, and $t = u_2/u_1 - 1$. Again we reduce to the case when the evaluation of $\mE_1$ on $A^{(0)}$ is free and thus the descent isomorphism $\delta$ is represented by a matrix of the form 
\[ 
T_{\delta} = N_0 + N_1 \cdot \delta_x + N_2   \cdot \delta_x^2 + \cdots 
\] 
where each $N_i$ is now a matrix with entries in $\frac{1}{i!}W \gr{x_1} \gr{u_1, t}^{\mathrm{PD}}$. 

As in the proof of Lemma \ref{lemma:descent_for_log_A1}, we want to show that each $N_i$ is divisible by $(x_1)^i$. To this end, consider the object $B^{(0), \log} = (W\gr{u}^{\mathrm{PD}}, 0^{\N} \oplus u^{\N})$ in (the opposite category of) the $p$-completed affine log crystalline site of $z_*^{\log}\cong s_1^{\log}$. Its self-coproduct is \[B^{(1), \log} = W \gr{u_1, \delta_x, t}^{\mathrm{PD}}.\] 
Let $\mE_*$ denote the restriction of $\mE$ to $z_*^{\log}$. We obtain a finite projective $B^{(0), \log}$-module $M_*$ together with a descent isomorphism by evaluating $\mE_*$ on $B^{(0), \log}$. Notice that we obtain the same $B^{(0), \log}$-module (with the same descent isomorphism) by first evaluating $\mE_1$ on $A^{(1), \log}$ and then base change along the maps
\[\lambda^{(0)}: W \gr{x} \gr{u}^{\mathrm{PD}}\ra W\gr{u}^{\mathrm{PD}}\] sending $x\mapsto 0$, and
\[ 
\lambda^{(1)}: W \gr{x_1} \gr{u_1, \delta_x, t}^{\mathrm{PD}} \ra W \gr{u_1, \delta_x, t}^{\mathrm{PD}} 
\]
sending $x_1 \mapsto 0$. Now, by Condition (4), we know that $\lambda^{(1)} (T_{\delta}) = 1$, so in particular, $\lambda^{(1)} (N_i) = 0$ for all $i \ge 1$. In other words, each $N_i$ is divisible by $x_1$ when $i \ge 1$.  Finally, as in the proof of Lemma \ref{lemma:descent_for_log_A1} and the proof of $(1) \so (6)$ in Theorem \ref{theorem:log_crystal_2}, we use the cocycle condition on $T_{\delta}$ and apply an induction on $i$ to prove that $N_i$ is divisible by $(x_1)^i$ for all $i \ge 1.$ 

So far, we have successively descended $\mE_1$ to some log crystal $\mE_1^{\partial}$ on $D_1^{\partial\mathrm{-}\!\log}$. But by Theorem \ref{theorem:log_crystal_1}, $\mE_1^{\partial}$ further descends to a crystal on $D_1$. In particular, $\mE$ has trivial monodromy at $z$, as desired. This completes the proof of Theorem \ref{theorem:log_crystal_2}. 
\end{proof}

\subsection{Log (iso)crystals over semistable log schemes: Part I\!V}  

Theorem \ref{theorem:log_crystal_2_global} then follows as a corollary of Theorem \ref{theorem:log_crystal_2} and its proof.  

\begin{proof}[Proof of Theorem \ref{theorem:log_crystal_2_global}] 
First let us observe that, in the statement of Theorem \ref{theorem:log_crystal_2}, we may replace $D^{\log} = D^{n,\log}$ by the base change 
\[
D^{n,\log}\times_{\spec k} \A^m_k
\]
for any positive integer $m$. Indeed, the same proof goes through \textit{verbatim}, except that now we have to carry over the coefficients coming from $\A^m_k$. Since, \'etale locally, $Z^{\log}$ admits strictly \'etale maps to $D^{n,\log}\times_{\spec k} \A^m_k$, Theorem \ref{theorem:log_crystal_2_global} then follows immediately using \'etale descent. 
\end{proof}

\subsection{Rigidity of crystallinity for semistable local systems} \label{ss:global_crystaline_rigidity_of_semistable_local_system}
 
As an application, the rigidity properties for log (iso)crystals imply a rigidity result (Theorem \ref{cor:rigidity_crystalline_over_semistable}) on crystallinity of $p$-adic local systems, which contains Theorem \ref{thm:main_intro_for_log_schemes} as a special case.  For the setup, let us first make the following (slightly nonstandard) definition of \textit{smooth boundaries} of a  rigid analytic space  with semistable reduction (compare with the notion of \textit{Shilov boundaries} of Berkovich \cite{Berkovich_spectral} and \textit{Shilov points} from \cite{DLMS2}).

\begin{definition}
Let $\fY$ be a geometrically connected $p$-adic formal scheme over $\mO_K$ with semistable reduction. Let $Y$ be the adic generic fiber over $K$ and $\fY_s$ be the special fiber over $k$. Let $\fY_s^{\mathrm{sm}}$ denote the smooth locus inside $\fY_s$. For an irreducible component $Z$ of $\fY_s$, let $Z^{\mathrm{sm}}:=Z\cap \fY_s^{\mathrm{sm}}$. Then we define a \emph{smooth boundary} of $Y$ to be an open subset $\mathrm{sp}^{-1}(Z^{\mathrm{sm}})\subset Y$ for some irreducible component $Z$ of $\fY_s$, where $\mathrm{sp}: Y\ra \fY_s$ is the specialization map. We shall denote the set of smooth boundaries in $Y$ by $\mathcal S_{\fY}$ (note that it depends on the choice of the semistable model $\fY$). 
\end{definition}

\begin{remark} \label{remark:Shilov_points}
The set $\mathcal S_{\fY}$ of smooth boundaries of $Y$ is in bijection with the set of irreducible components of $\fY_s$, which in turn corresponds to the set of \textit{Shilov boundaries} and \textit{Shilov points} on $Y$ in the sense of \cite{Berkovich_spectral,DLMS2}. 
\end{remark}

\begin{example}
The smooth boundaries of the standard thick annulus $A_1 = \{1/p \le  |z| \le 1 \}$ described in the introduction are precisely $\mathbb B_0$ and $\mathbb B_1$  (see Figure \ref{fig:model_for_A1}). 
\end{example}

The following is a generalization of Theorem \ref{thm:main_intro_for_log_schemes} to all smooth rigid analytic varieties with semistable reduction. 

\begin{theorem} \label{cor:rigidity_crystalline_over_semistable}
Let $\fY$ be a geometrically connected $p$-adic formal scheme over $\mO_K$ with semistable reduction and let $Y$ be the adic generic fiber over $K$. Let $\{y_i\}_{i \in \mathcal S_{\fY}}$ be a collection of classical $K$-points such that $y_i$ is contained in the $i^{\mathrm{th}}$-smooth boundary for each $i$. Let $\L$ be a semistable \'etale $\Z_p$-local system on $Y$, then the following are equivalent 
\begin{enumerate}
\item $\L|_{y_i}$ is crystalline for each $i\in \mathcal S_{\fY}$. 
\item $\L$ is crystalline at all classical points on $Y$. 
\item $\L$ is crystalline (in the sense of Definition \ref{def:st_local_sys}). 
\end{enumerate}
\end{theorem}

\begin{proof} 
This is immediate from Theorem \ref{theorem:log_crystal_2_global}. 
\end{proof}

\begin{remark}  
For local systems arising from abelian varieties, \cite{OSZ} proves the following slightly stronger result. They consider an arbitrary rigid analytic map $f: A_1 \ra \mathcal A_{g, K}^{\mathrm{an}}$ from the annulus $A_1$ to (the rigid analytification of) the Shimura variety parametrizing $g$-dimensional polarized abelian varieties with a torsion-free level structure $K \subset \mathrm{GSp}_g (\A_{f})$ and consider the $p$-adic local system $\L = f^* T_p$ coming from the pullback of the universal $p$-adic Tate module. They show that if $\L$ is crystalline at one classical point on $A_1$, it is in fact crystalline at all classical points on $A_1$. Let us remark that the proof of this result in \cite{OSZ} uses $l$-adic techniques and crucially uses the geometry of $\mathcal A_{g, K}$. We refer the curious reader to \S 3 (in particular \S 3.1.2) of \textit{loc.cit.} for more details. 
\end{remark}

\newpage 

\section{\large Maps between semistable models and the non-contracting locus} \label{sec:dodging}
 
In this section, we study maps between semistable integral models of smooth curves and analyze their local reduction behavior near a given classical point on the generic fiber. More precisely, we prove existence of certain semistable models so that the given classical point ``dodges the contracting locus'' in the sense of Definition \ref{definition:contracting_locus}. The main results we prove in this direction are Theorem \ref{theorem:modified_spreadout_for_dodging} and Theorem \ref{thm:dodging_contracting_locus_general}, where the former is a more precise version of Theorem \ref{mainthm:modified_spreadout_for_dodging_intro} from the introduction. As mentioned there, these results are intimately related to crystallinity of $p$-adic local systems. In particular, as an application of Theorem \ref{thm:dodging_contracting_locus_general}, we prove a refined $p$-adic monodromy theorem (Theorem \ref{thm:dodging}) which plays an essential role in the proof of Theorem \ref{thm:conjecture_for_projective_varieties_intro} (cf. \S \ref{sec:conjecture}).

\subsection{Points on rigid analytic curves} \label{ss:points_on_curves}

Before stating the main results of this section, let us first make precise some basic notions that are used in the article, including the notion of \emph{types} of points on smooth rigid analytic curves (cf. \S \ref{sss:classification_of_points_on_P1} and \S \ref{subsection: points on smooth curves}) and the notion of \emph{$K$-rational points} (cf. \S \ref{sss:rational_points}). Smooth rigid analytic curves are always viewed as adic spaces, unless specified.

\subsubsection{Classification of points on $\P^1$}\label{sss:classification_of_points_on_P1} 
Let us first recall the classification of points on $\P^1$ over an algebraically closed non-archimedean field. We refer the reader to \cite{Huber1,Huber_etale} (also see \cite{Berkovich_spectral, Berkovich_etale} and \cite[Example 2.20]{perfectoid}) for a detailed exposition on this subject. 
Let us fix a completed algebraic closure $\C_p$ of $K$, and fix a coordinate $t$ of $\P^1_{\C_p}$. For any element $\alpha \in \C_p$ and $r \in \R_{\ge 0}$, let
\[\D_{\C_p} (\alpha, r) = \left\{|t - \alpha| \le r\right\}  \quad (\textup{resp. } \D_{\C_p} (\alpha, <r) = \left\{|t - \alpha| < r\right\}) \] denote the closed (resp. open) disc centered at $\alpha$ of radius $r$. We classify points on the closed unit disc 
\[ \D_{\C_p} := \D_{\C_p}(0, 1) = \spa (\C_p \gr{t}, \mO_{\C_p} \gr{t})\]  
viewed as an affinoid open subset of $\P^{1, \mathrm{ad}}_{\C_p}$. Given this, one easily deduce a classification of points on $\P^{1, \mathrm{ad}}_{\C_p}$ (see Remark \ref{rmk: classificaiton of pts on P1}).

The points on $\D_{\C_p}$ are divided into the following 5 types. \\ 

\noindent \textit{Type} I. \: These points correspond to the elements of $\mO_{\C_p}$ as follows: each $\alpha \in \mO_{\C_p}$ gives rise to a continuous valuation 
\[ x_{\alpha}: \C_p \gr{t} \ra \R_{\ge 0}\] sending $f \mapsto |f(\alpha)|.$  

\vspace*{0.2cm}
\noindent \textit{Type} I\!I and I\!I\!I. \: These points correspond to the discs $\D_{\C_p} (\alpha, r) = \left\{|t - \alpha| \le r\right\} \subset \D_{\C_p}$, where $\alpha \in \mO_{\C_p}$ and $0 < r \le 1$. For each such disc, we get a continuous valuation 
\[ 
x_{\alpha, r}: \C_p \gr{t} \ra \R_{\ge 0}, \quad f \mapsto \sup_{\substack{\beta \in \mO_{\C_p} \\ |\beta-\alpha|\le r}} |f(\beta)|. 
\] 
Such a point $x_{\alpha, r}$ is a \emph{type I\!I} point if $r \in |\C_p^\times| = p^{\Q}$, otherwise $x_{\alpha, r }$ is a \emph{type I\!I\!I} point. In the case of type I\!I, the point $x_{\alpha, r}$ is called the \emph{Gaussian point} of the disc $\D_{\C_p} (\alpha, r)$.

\vspace*{0.2cm}
\noindent \textit{Type} I\!V.  These points correspond to a sequence $\eta$ of closed discs $\D_1 \supset \D_2 \supset ...$ in $\D_{\C_p}$ such that their intersection $\cap_{i \in \Z_{>0}} \D_i$ contains no type I point,\footnote{Note that this could only happen when $\C_p$ is not spherically complete, for example, when $K/\Q_p$ is a finite extension.} which gives rise to a continuous valuation 
\[x_{\eta}: \C_p \gr{t} \ra \R_{\ge 0}, \quad f \mapsto \inf_{i \in \Z_{>0}} \sup_{\substack{\beta\in \mO_{\C_p} \\ \beta \in \D_i}} |f(\beta)|. \]
In this case, we often write $x_{\eta} = \cap_{i \in \Z_{>0}} \D_i$ and view the intersection of these nested discs as the type I\!V point itself.  

\vspace*{0.2cm}
\noindent \textit{Type} V.  These are the points of rank $2$ valuations (and thus do not appear in the theory of Berkovich spaces). Let us start with the type I\!I point $x_{\alpha, r}$ corresponding to the closed disc $\D_{\C_p}(\alpha, r)$. For a sign $? \in \{>, <\}$, consider the totally ordered abelian group 
\[
\Gamma_{? r}= \begin{cases}
    \Gamma_{<r} = \R_{> 0} \times \gamma^{\Z} \quad \textrm{where } \gamma = r^{-}, \textrm{ i.e., with } r'< \gamma < r \textrm{ for all } r' < r  \\ 
    \Gamma_{>r} = \R_{> 0} \times \gamma^{\Z} \quad \textrm{where } \gamma = r^{+},  \textrm{ i.e., with } r < \gamma < r' \textrm{ for all } r' > r
\end{cases}.
\]
Let $\beta \in \mO_{\C_p}$ such that $|\beta-\alpha|\le r$. Then $\beta$ corresponds to a type I point $x_{\beta}$ in $\D_{\C_p}(\alpha, r)$. Notice that every element $f\in \C_p \gr{t} $ can be uniquely written as $f = \sum_{i=0}^{\infty} a_i (t - \beta)^i$ with $a_i\in \C_p$. There is a continuous rank 2 valuation $\C_p \gr{t} \ra \Gamma_{? r}\cup \{0\}$ sending 
\[ 
f = \sum_{i=0}^{\infty} a_i (t - \beta)^i  \:  \longmapsto \: \sup_{i} |a_i| \gamma^i. 
\]
These are points of \emph{type V} (except when $r=1$ and $?$ is ``$>$'', this continuous valuation lives in $\P^1_{\C_p}$ but not in $\D_{\C_p}$). When $?$ is ``$>$'', the valuation does not depend on the choice of $\beta$; it only depends on the disc $\D_{\C_p} (\alpha, r)$. We denote this type V point by $x_{\alpha, > r}$. On the other hand, when $?$ is ``$<$'', we denote the corresponding point by $x_{\beta, < r}$. Such a point $x_{\beta, < r}$ only depends on the open disc $\D_{\C_p} (\beta, < r) \subset \D_{\C_p} (\alpha, r)$; it does not depend on the specific choice of $\beta$ in the open disc.

\begin{remark}\label{rmk: classificaiton of pts on P1}
The classification of points on $\D_{\C_p}$ generalizes to points on $\P^{1, \mathrm{ad}}_{\C_p}$. They are still divided in five types:
\begin{itemize}
\item Type I points $x_{\alpha}$ for $\alpha\in \C_p$, together with a point ``$\infty$'';
\item Type I\!I and I\!I\!I points $x_{\alpha, r}$ corresponding to discs $\D_{\C_p}(\alpha,r)$ with $\alpha\in \C_p$ and $r>0$, which are type I\!I if $r\in p^{\Q}$ and type I\!I\!I otherwise;
\item Type I\!V points $x_{\eta}$ corresponding to nested discs with empty intersection;
\item Type V points $x_{\alpha, > r}$ and $x_{\beta, < r}$ with $\alpha, \beta\in \C_p$ and $r\in p^{\Q}$. Again, the point $x_{\alpha, > r}$ only depends on the disc $\D_{\C_p} (\alpha, r)$, while the point $x_{\beta, < r}$ only depends on the open disc $\D_{\C_p} (\beta, < r)$.
\end{itemize}
Let us note that, all the points above except for the type I\!I points are closed. The closure of a type I\!I point $x_{\alpha, r}$ precisely consists of $x_{\alpha, r}$ and the type V points $x_{\beta, < r}$ and $x_{\alpha, >r}$ described as above, and is homeomorphic to $\P^1_{\cl \F_p}$, with $x_{\alpha, r}$ being the generic point.  
\end{remark}

Now we describe $\P^{1, \mathrm{ad}}_K$. In fact, the absolute Galois group $G_K=\Gal(\overline{K}/K)$ naturally acts on $\P^{1, \mathrm{ad}}_{\C_p}$ and induces a canonical projection $\pi: \P^{1, \mathrm{ad}}_{\C_p} \ra \P^{1, \mathrm{ad}}_{K}$ as a quotient, which sends a point on $\P^{1, \mathrm{ad}}_{\C_p}$ to its $G_K$-orbit. For $\star \in \{\textup{I, I\!I,  I\!I\!I,  I\!V, V}\}$, we say that a point on $\P^{1, \mathrm{ad}}_K$ has \emph{type} $\star$ if it is the image of a point on $\P^{1, \mathrm{ad}}_{\C_p}$ of type $\star$ under $\pi$. The explicit description of points on $\P^{1, \mathrm{ad}}_{\C_p}$ immediately yields a description of points on $\P^{1, \mathrm{ad}}_K$.

For $\alpha\in \C_p$ and $r>0$, let $\D_K(\alpha,r)$ (resp. $\D_K(\alpha,<r)$) denote the image of $\D_{\C_p}(\alpha,r)$ (resp. $\D_{\C_p}(\alpha,<r)$) under the projection $\pi$. We still refer to them as the \emph{closed} (resp. \emph{open}) \emph{discs}. We caution the reader that $\D_K(\alpha,1)$ is not necessarily isomorphic to the standard closed unit disc
\[\D_K:=\D_K(0,1)=\spa (K \gr{t}, \mO_K \gr{t})\]
over $K$; in fact, $\D_K(\alpha,1)\cong \D_K$ if and only if $\D_K(\alpha,1)$ contains some $x_{\beta}$ with $\beta\in K$.

\subsubsection{Points on smooth curves}\label{subsection: points on smooth curves}

More generally, let $Y$ be a smooth rigid analytic curve over $K$, viewed as an adic space. For $\star \in \{\textup{I, I\!I,  I\!I\!I,  I\!V, V}\}$, we say that a point $x \in Y_{\C_p}$ has \emph{type} $\star$ if for an/any \'etale map $f: U \ra \P^{1, \mathrm{ad}}_{\C_p}$ from an open neighhorhood $U$ of $x$ in $Y_{\C_p}$, the image $f(x) \in \P^{1, \mathrm{ad}}_{\C_p}$ has type $\star$. We say that a point on $Y$ has \emph{type} $\star$ if it is the image of a point on $Y_{\C_p}$ of type $\star$ under the quotient $\pi:Y_{\C_p}\rightarrow Y$.

\subsubsection{$K$-rational points} \label{sss:rational_points} 
Let $Y$ be a smooth rigid analytic curve over $K$, viewed as an adic space. Recall that a point on $Y$ corresponds to a $G_K$-orbit in $Y_{\C_p}$. We say that a point in $Y_{\C_p}$ is \emph{$K$-rational} if its $G_K$-orbit consists of a single element, and we say that a point on $Y$ is \emph{$K$-rational} if it is the image of a $K$-rational point in $Y_{\C_p}$. More generally, we say that a point on $Y$ is \emph{$K'$-rational} for some finite extension $K'/K$ if it is the image of a $K'$-rational point on $Y_{\C_p}$.  

Notice that a type I point on $Y$ is \emph{classical} if and only if it is $K'$-rational for some finite extension $K'/K$. Notice that the points $x_{\alpha}$ in \S \ref{sss:classification_of_points_on_P1} (viewed as a point in $\P^1_K$) with $\alpha\notin \overline{K}$ are of type I but not classical.

\begin{remark}\label{remark:K_rational_point_under_field_extension}
If $x$ is a $K$-rational point on $Y$, and $K'/K$ is a finite extension, then the pre-image of $x$ under the projection $Y_{K'} \ra Y$ consists of a singleton, thus it can also be viewed as a $K'$-rational point on $Y_{K'}$. 
\end{remark} 
\begin{example} 
Let us consider the closed unit disc $\D_{K}=\D_K(0,1)$ over $K$. A type I point on $\D_K$ is a $G_K$-orbit of a type I point $x_{\alpha}$ on $\D_{\C_p}$, with $\alpha \in \mO_{\C_p}$ (see \S \ref{sss:classification_of_points_on_P1}). This point  is $K$-rational precisely when $\alpha \in \mO_K$ (in which case the $G_K$-orbit consists of just $x_{\alpha}$ itself), and is $K'$-rational if it is the $G_K$-orbit of some $x_{\alpha}$ where $\alpha \in \mO_{K'}$. A point of type I\!I or I\!I\!I, given by the $G_K$-orbit of $x_{\alpha, r} \in \D_{\C_p}$, is $K'$-rational if one can choose $\alpha \in \mO_{K'}$. A similar remark holds for type V points given by the $G_K$-orbit of $x_{\alpha, >}$ or $x_{\beta, <}$. In particular, every point of type I\!I, I\!I\!I, or V is $K'$-rational for some finite extension $K'/K$.
\end{example}

\begin{example}
More generally, for a smooth rigid analytic curve $Y$ over $K$, every point $x \in Y$ of type I\!I, I\!I\!I, or V is $K'$-rational for some finite extension $K'/K$. In this case, all points in the pre-image of this $x$ under the projection $Y_{K'} \ra Y$ are $K'$-rational. 
\end{example}

\begin{example}
Let $t$ be a coordinate on $\P^1_K$; namely, $t$ is a rational function on $\P^1_K$, and the field of rational functions on $\P^1_K$ is precisely $K(t)$. Then it is valid to evaluate $t$ at a $K$-rational point $x\in\P^{1, \mathrm{ad}}_K$ to obtain $t(x)\in K$; this is because a $K$-rational point on $\P^{1, \mathrm{ad}}_K$ is the same as a closed $K$-point on $\P^1_K$. For a general type I point $x\in\P^{1, \mathrm{ad}}_K$, it might not make sense to write $t(x)$; but $|t(x)|=|t|_x\in \R_{\ge 0}$ is still well-defined, namely, it is the evaluation of $t$ under the continuous valuation $|\cdot|_x$ corresponding to $x$.
\end{example}

\subsubsection{Stalks and residue fields} \label{sss:stalks}

For any point $x$ on a smooth rigid analytic curve $Y$ over $K$ (viewed as an adic space), we write $\mO_{Y, x}$ for the stalk of the structure sheaf at $x$, and denote the residue field of $\mO_{Y, x}$ by $\mathbf{k}(x)$. The continuous valuation corresponding to $x$ extends to a valuation $|\cdot|_x$ on $\mathbf{k}(x)$. The completion of $\mathbf{k}(x)$ with respect to $|\cdot|_x$ is denoted by $\sH (x)$. The field $\mathbf{k}(x)$ (resp. $\sH (x)$) is referred to as the \emph{residue field} (resp. \emph{completed residue field}) at $x$. 

\begin{example}[Stalks of non-type I points]\label{example:stalk_on_curve}
For a non-type I point $\xi$ on a smooth rigid analytic curve $Y$ over $K$, its stalk $\mO_{Y, \xi}$ is a field and is equal to $\mathbf k(\xi)$. Hence the completed stalk agrees with $\sH(\xi)$. Moreover, suppose $\xi$ is a type V point and $\xi_{\mathrm{I\!I}}$ is the unique type I\!I point specializing to $\xi$, then we have an injection $\iota:\mathbf k(\xi)\hookrightarrow \mathbf k(\xi_{\mathrm{I\!I}})$ of fields by \cite[Lemma 1.1.10]{Huber_etale}. In fact, $\iota$ is a homeomorphism onto its image and $\iota(\mathbf k(\xi))$ is dense in $\mathbf k(\xi_{\mathrm{I\!I}})$. This further induces an injection $\sH (\xi) \hookrightarrow \sH (\xi_{\mathrm{I\!I}})$ with dense image.
\end{example}

\subsubsection{Comparison with Berkovich curves} \label{sss:comparison_between_stalks} Let $Y$ be a smooth affinoid rigid analytic curve over $K$ associated to a Tate $K$-algebra $\mathscr A$. Let $Y^{\mathrm{Ber}}$ (resp. $Y^{\mathrm{ad}}$) be its associated Berkovich space (resp. adic space). Note that $Y^{\mathrm{Ber}}$ identifies as the maximal Hausdorff quotient of $Y^{\mathrm{ad}}$ as a topological space. Let $\mO_{Y^{\mathrm{Ber}}}$ denote the structure sheaf of the Berkovich space. On an open subset\footnote{This is an open subset in the Hausdorff topology of $Y^{\mathrm{Ber}}$, also referred to as the \textit{canonical} topology, in comparison with the theory of rigid analytic spaces of Tate.
}  
$U \subset Y^{\mathrm{Ber}}$, the global sections are given by 
\[ 
\Gamma(U, \mO_{Y^\mathrm{Ber}}) = \lim_{V \subset U} \mathscr A_V,
\]
where the limit is taken over all $V \subset U$ that are finite unions of affinoid subdomains (namely, \textit{special subsets}) in $Y$, and $\mathscr A_V$ is the ring of analytic functions on $V$ (see \cite[2.2.6]{Berkovich_spectral}). Let $x$ be a point on $Y^{\mathrm{Ber}}$, which we also view as a point on $Y^{\mathrm{ad}}$ (it has type I, I\!I, I\!I\!I, or I\!V). By the Gerritzen--Grauert Theorem (see \cite[\S 3.3 Theorem 20]{Bosch} and \cite[Proposition 2.2.3]{Berkovich_spectral}), the Berkovich stalk $\mO_{Y^{\mathrm{Ber}}, x}$ can be computed as 
\[ 
\mO_{Y^{\mathrm{Ber}}, x} = \underset{x \in U}{\mathrm{colim}} \lim_{V \subset U} \mathscr A_V,
\] 
where the colimit is taken over all affinoid open neighborhoods $U$ of $x$, and for each such $U$ the limit is taken over all \textit{rational subsets} $V$ contained in $U$. On the other hand, the stalk of the structure sheaf $\mO_{Y^{\mathrm{ad}}}$ of the adic space $Y^{\mathrm{ad}}$ is given by 
\[ 
\mO_{Y^{\mathrm{ad}}, x} =  \underset{x \in V}{\mathrm{colim} } \: \mathscr A_V
\] 
where the colimit is taken over all rational subsets $V \subset Y$ that contains the point $x$. Thus, we have a natural inclusion $ \mO_{Y^{\mathrm{Ber}}, x}  \hookrightarrow \mO_{Y^{\mathrm{ad}}, x}$ of local rings. This inclusion is an isomorphism when $x$ is a type I point but \textit{not} necessarily an isomorphism in general. When $x$ has type I\!I, I\!I\!I, or I\!V, the Berkovich stalk $\mO_{Y^{\mathrm{Ber}}, x}$ and adic stalk $\mO_{Y^{\mathrm{ad}}, x}$ are both fields (also see Example \ref{example:stalk_on_curve}) and are equal to the respective residue fields of the stalks. Moreover, the inclusion $ \mO_{Y^{\mathrm{Ber}}, x} \hookrightarrow
\mO_{Y^{\mathrm{ad}}, x}$ respects the topology defined by the canonical valuations on both sides. We claim that, in this case,  the inclusion $\mO_{Y^{\mathrm{Ber}}, x}  \hookrightarrow\mO_{Y^{\mathrm{ad}}, x}$ has dense image and thus becomes an isomorphism upon taking completions, which we identify and simply denote by $\sH (x)$ as in Example \ref{example:stalk_on_curve}.  To see the claim, let $h \in  \Gamma (V, \mO_{Y^{\mathrm{ad}}})$ be an analytic function on a rational subset $V$ containing $x$, we want to show that for any $\epsilon \in \R_{> 0}$, there exists an open subset $U \subset Y^{\mathrm{Ber}}$ (in the canonical topology) containing $V$ and an analytic function $h' \in \Gamma (U, \mO_{Y^{\mathrm{Ber}}})$, such that $|h - h'|_{x} < \epsilon$. Without loss of generality, we may write 
\[
V = Y (\frac{f_1, ..., f_n}{g})  =   \left\{ y \in Y \:  \Big\vert \: |f_i|_y \le |g|_y, \: \forall i    \right\}
\]
where $f_1, ..., f_n \in A$ generates the unit ideal, $g \in A$, and $f_n = \varpi^N$ for some sufficiently large integer $N$ (see \cite[Remark 2.8]{perfectoid}). By construction of rational subsets, there exists some $h'$ of the form 
\[ 
h' = \sum_{\substack{I = (i_1, ..., i_n) \\ \mathrm{finite sum} }} a_I \cdot (f_1/g)^{i_1} (f_2/g)^{i_2} \cdots (f_n/g)^{i_n}, \qquad a_I \in A,
\]
such that $|h - h'|_{V} < \epsilon$, where $|\cdot|_{V}$ denotes the spectral norm on $V$. Therefore, in particular, $|h - h'|_{x} < \epsilon$. Note that $h'$ is defined on the Zariski open subset $Y \minus V(g)$, so it is defined on some open subset $U \subset Y^{\mathrm{Ber}}$ in the canonical topology that contains $V$. This justifies our claim.

\subsection{Contracting and non-contracting loci}  \label{ss:reduction_type_near_point}
For the setup, let $f: Y' \ra Y$ be a finite cover between smooth rigid analytic curves over $K$ (viewed as adic spaces). Let $\fY$ (resp. $\fY'$) be a formal integral model of $Y$ (resp. of $Y'$) over $\mO_K$, that is, an admissible $p$-adic formal scheme over $\spf \mO_K$ such that its rigid analytic generic fiber over $K$ is isomorphic to $Y$ (resp. to $Y'$). Further suppose that $\fY'$ has semistable reduction.\footnote{To ensure the correct level of generality, we do not require $\fY$ to be a semistable formal model of $Y$ at this stage.} 

Let $\fY_s$ (resp. $\fY'_s$) denote the special fiber of $\fY$ (resp. $\fY'$) over $\spec k$ and let $\mathrm{sp}: Y \ra \fY_s$ (resp. $\mathrm{sp}: Y' \ra \fY'_s$) denote the specialization map. Let 
\begin{equation}\label{eq:choice_of_f_hat}
\widehat f: \mathfrak Y' \ra \mathfrak Y 
\end{equation} be a morphism of $p$-adic formal schemes that induces $f$ on the (rigid analytic) generic fibers,\footnote{Note that, given the map $f: Y' \ra Y$ and fix the integral model $\fY$ (but not fixing $\fY'$), after replacing $K$ by a finite extension if necessary, there always exists a semistable model $\fY'$ together with an integral model $\widehat f: \fY' \ra \fY$ of $f$ by the works of Raynaud \cite[\S 8]{Bosch} and Deligne--Mumford \cite{DM}.} and let $f_s: \fY_s' \ra \fY_s$ denote the induced map on the special fibers. Write $\fY_s' = \cup_{i \in I} Z_i$ where the $Z_i$'s are the irreducible components of $\fY_s'$. For each $i \in I$, the map $f_s$ restricts to a map $g_i: Z_i \ra \fY_s$.

\begin{definition}\label{definition:contracting_locus}
\begin{enumerate}
    \item We say that an irreducible component $Z_i \subset \fY_s'$ is a \textit{finite component} (with respect to $f_s$) if the map $g_i$ is finite. Otherwise we say that $Z_i$ is a \textit{contracting component}, in which case the map $g_i$ is not finite and the image of $Z_i$ under $g_i$ is a closed point on $\fY_s$. 
    \item Let $I_{\mathrm{c}} \subset I$ be the set of contracting components of $\fY_s'$ (with respect to $f_s$) and let $I_{\mathrm{nc}} = I - I_{\mathrm{c}}$ be the set of finite components. 
    We define  the \textit{contracting locus} of  $Y'$ (with respect to $\widehat f$ ) to be the open subset  
    \[
        Y_{\mathrm{c}}' := \mathrm{sp}^{-1} \left( \bigcup_{i \in I_{\mathrm{c}}} Z_i \right) \subset Y',
        \]  
         and define    the \textit{non-contracting locus} of  $Y'$ (with respect to $\widehat f$ ) to be its complement 
        \[
        Y_{\mathrm{nc}}' := \mathrm{sp}^{-1} \left( \fY_s' \minus \bigcup_{i \in I_{\mathrm{c}}} Z_i \right)=Y'\minus Y'_{\mathrm{c}}.
        \]        
        These notions depend on the map $\widehat f$ in (\ref{eq:choice_of_f_hat}). 
    \end{enumerate}
\end{definition}

The following question is closely related to our proof of the rigidity conjecture, and seems interesting on its own. 

 \begin{question} \label{question:dodging_contracting_locus}
     Let $f: Y' \ra Y$ be a finite map of connected smooth projective curves over $K$. Let $\fY$ be a semistable formal integral model of $Y$ over $\mO_K$. Let $y_0 \in Y(K)$ be a closed point (viewed as a $K$-rational classical point on the associated adic space). Then, up to replacing $K$ by a finite extension and replacing $Y'$ by a further finite cover, does there always exist a semistable formal integral model $\fY'$ of $Y'$ and an integral model $\widehat f: \fY' \ra \fY$ of $f$, such that $y_0$ lies on the image of the non-contracting locus $Y'_{\mathrm{nc}}$ of $Y'$ with respect to $\widehat f$? 
 \end{question} 

In other words, the question above asks for the existence of semistable models with respect to which the given point ``dodges the contracting locus''. 

\begin{remark}
When $Y$ has genus at least 2, up to enlarging $K$, it admits a minimal semistable model. In this case, in formulating the question above, we may choose $\fY$ to be this minimal semistable model, and hence the question purely concerns the generic fibers.
\end{remark}

\begin{remark}[The effect of enlarging $K$]\label{remark:enlarging_K}  In the formulation of the question above, as well as various other places later in this section, we need to  allow ourselves to replace $K$ by some finite extensions. This often leads to the following caveat: after replacing $K$ by a finite extension $K'$, a $p$-adic formal scheme with semistable reduction over $\mO_{K}$ may no longer have semistable reduction over $\mO_{K'}$.\footnote{This is related to why we do not require $\fY$ to have semistable reduction in the definition of contracting and non-contracting loci in Definition \ref{definition:contracting_locus}.} In particular, if $\widehat f: \fY' \ra \fY$ is a map between integral models over $\mO_K$ as in (\ref{eq:choice_of_f_hat}), and $y$ is a closed $K$-point on $Y'$ that specializes to the intersection point $\cl y$ of two non-contracting components on $\fY_s'$, then after replacing $K$ by a finite ramified extension $K'$, we typically need to perform further admissible blowing-ups of $\fY'_{\mO_{K'}}$ at $\cl y$ in order to obtain a semistable formal model. After such a blowing-up, $y$ will no longer live on the non-contracting locus with respect to the new map between integral models. 

Note, however, this issue goes away if $y$ specializes to a point on the smooth locus of $\fY'_s$. More precisely, let $y \in Y'(K)$ be a closed $K$-point on the non-contracting locus $Y'_{\mathrm{nc}}$ (with respect to $\widehat f: \fY'\ra\fY$ as in (\ref{eq:choice_of_f_hat})) such that $y$ specializes to a point $\cl y$ lying on the smooth locus of $\fY'_s$. Consider the base change $f_{K'}: Y'_{K'}\ra Y_{K'}$ for some finite extension $K'/K$. Then, up to replacing $K'$ by a further finite extension, there exists a semistable model $\fY''$ of $Y'_{K'}$ over $\mO_{K'}$ and an integral model $\widehat{f}_{K'}: \fY''\ra \fY_{\mO_{K'}}$ of $f_{K'}$ such that $y$ continues to live on the non-contracting locus (with respect to $\widehat{f}_{K'}$). Here $\fY''$ is obtained from $\fY'_{\mO_{K'}}$ via admissible blowing-ups, and such blowing-ups can be chosen to avoid the point $\cl y$. In particular, $\cl y$ continues to live on the smooth locus of $\fY''_s$. We will frequently make use of this observation in this section.
\end{remark}


Later in the section, we will discuss weaker versions of this question which are enough for our purpose (cf. Theorem \ref{theorem:modified_spreadout_for_dodging} and Theorem \ref{thm:dodging_contracting_locus_general}). We suspect that the answer to Question \ref{question:dodging_contracting_locus} is yes when both $Y'$ and $Y$ have good reductions, but no in general, by inspecting the proof of the following variant of the question. 

\begin{lemma}\label{lemma:dodging_contracting_locus_after_composition}
 Let $X$ be a smooth projective curve over $K$ with good reduction. Let $x_0 \in X(K)$ be a closed point (viewed as a $K$-rational classical point on the associated adic space). Then, up to enlarging $K$ if necessary, there exists a semistable model $\fX'$ of $X$ and a morphism of $p$-adic formal schemes \[ \widehat f': \fX' \ra \widehat \P^1_{\mO_K}\] which induces a finite cover $f': X \ra \P^1_K$ on the generic fibers, such that $x_0$ lies on the non-contracting locus of $X$ with respect to $\widehat f'$.
\end{lemma}

This follows from the following more precise version of the lemma. 

\begin{lemma} \label{lemma:dodging_contracting_locus_after_composition2}
Let $X$ be a smooth projective curve over $K$ with good reduction and let $\fX$ be a smooth integral model of $X$ over $\mO_K$. Let $f: X \ra \P^1_K$ be a finite morphism of curves over $K$. Let $x_0 \in X(K)$ be a closed $K$-point as above. Then, up to enlarging $K$ if necessary, there exists a semistable model $\fX'$ of $X$ over $\mO_K$ (possibly different from $\fX$), an automorphism $\alpha: \P^1_K \ra \P^1_K$, and a morphism of $p$-adic formal schemes $\widehat f': \fX' \ra \widehat \P^1_{\mO_K}$ such that
\begin{enumerate}
\item[(i)] $\widehat f'$ is an integral model of the composition $f' = \alpha \circ f: X \ra \P^1_K$; and
\item[(ii)] $x_0$ lies on the non-contracting locus $X_{\mathrm{nc}}$ of $X$ with respect to $\widehat f'$.
\item[(iii)] the image $\cl x_0$ of $x_0$ under the specialization map lies on the smooth locus of the special fiber $\fX'_s$ of $\fX'$.
\end{enumerate}
\end{lemma}

\begin{proof} The proof builds on a careful inspection of Raynaud's observation that rigid analytic spaces can be viewed as formal schemes up to admissible blowing-ups (see \cite[\S 8]{Bosch}). Let $\fX_s$ denote the special fiber of $\fX$ over $k$ and let $\cl{x}_0$ denote the image of $x_0$ in $\fX_s$ under the specialization map.  
Let us recall Raynaud's method on how to modify the integral model $\fX$ via admissible blowing-ups in order to construct a map to $\widehat{\P}^1_{\mO_K}$. This is done in 3 steps. 
\begin{enumerate}
\item  First we cover $\fX$ by affine open formal subschemes $\{\fX_i\}_{i\in I}$ (resp. cover $\widehat{\P}^1_{\mO_K}$ by affine open formal subschemes $\{\mS_j\}_{j \in J}$) such that for each $i \in I$, the rigid analytic generic fiber $X_i$ of $\fX_i$ maps to some $S_j$, where $S_j$ is the rigid analytic generic fiber of $\mS_j$.\footnote{Strictly speaking, in order to achieve this, we might need to take a further admissible blowing-up of $\fX$ (see \cite[\S 8.4 Lemma 5]{Bosch}). For curves this is not necessary.} 
\item Then one constructs an admissible blowing-up $\pi_i: \fX_i' \ra \fX_i$ as in \cite[\S8.4 Lemma 6]{Bosch} which admits a map $\fX_i' \ra \mS_j$. 
\item Finally, one ``glues'' these $\fX'_i$'s by constructing a further admissible blowing-up $\pi: \fX' \ra \fX$ that dominates the $\fX'_i$'s in the sense of \cite[\S8.2 Proposition 14]{Bosch}. In fact, since we are in the case of curves, after additional admissible blowing-ups on $\fX'$ and enlarging $K$ if necessary, one can ensure that $\fX'$ is semistable. 
\end{enumerate}
We would like to show that, by carefully choosing the coverings of $\fX$ and $\widehat{\P}^1_{\mO_K}$, as well as modifying the map $f$ by composing with an automorphism $\alpha: \P^1_K \ra \P^1_K$, one can construct the desired map $\widehat f'$ (following Steps (1)-(3) above) with respect to which the point $x_0$ avoids the contracting locus. To this end, we first choose $\{\fX_i\}_{i\in I}$ and $\{\mS_j\}_{j \in J}$ as in Step (1). We may assume that $x_0$ is contained in $X_0$ for some index $0 \in I$ and does not belong to any other $X_i$'s. Let $s_0=f(x_0)$. Choose a coordinate $t$ on $\P^1_K$ such that $t(s_0)=0$. By shrinking $\fX_0$ is necessary, we may further assume that the image of $X_0$ under $f$ lies in the unit disc $S_0=\D_K = \{|t| \le 1 \}$. 

Write $\fX_0 = \spf A_0$ and let $\mS_0=\widehat \A^1 = \spf \mO_K \gr{t} \subset \widehat{\P}^1_{\mO_K}$. The map $f|_{X_0}: X_0 \ra \D_K$ corresponds to a ring map $K \gr{t} \ra A_0 [1/\varpi]$ sending $t$ to some $a_0 \in A_0[1/\varpi]$. Let $m_0 \ge 0$ be the smallest non-negative integer such that $a_0^+ := \varpi^{m_0} a_0$ lives in $A_0$. Consider the map $\alpha: \P^1_K \ra \P^1_K$ induced by $t \mapsto \varpi^{m_0} t$ (note that the restriction of $\alpha$ to $\D_K$ has the effect of ``shrinking'' the unit disc). Then the composition \[ \alpha \circ f|_{X_0}: X_0 \ra \D_K\]
is given by $t \mapsto a_0^+ \in A_0$, which admits an integral model $\fX_0 \ra \mS_0$ on the nose (without any blowing-ups on $\fX_0$). Now we consider the composition \[f' = \alpha \circ f: X\rightarrow \P^1_K\] 
and perform Step (2) above with respect to $f'$. In this step, by our construction, there is no need to blow up $\fX_0$ (i.e., $\fX'_0=\fX_0$). Finally, by the proof of  \cite[\S8.2 Proposition 14]{Bosch}, we know that since $\fX_0' = \fX_0$ and $x_0 \notin X_i$ for any $i \ne 0$, the map $\pi: \fX' \ra \fX$ in Step (3) above is obtained from admissible formal blowing-ups at closed points on $\fX_s$ that are disjoint from $\cl x_0$. Moreover, upon enlarging $K$, the further admissible blowing-ups needed in order to make $\fX'$ semistable over $\mO_K$ are along points that are disjoint from $\cl x_0$ (cf. Remark \ref{remark:enlarging_K}). Consequently, we obtain the desired 
\[ 
\widehat f': \fX' \ra \widehat \P^1_{\mO_{K}},
\]
with respect to which $x_0$ lies on the non-contracting locus $X_{\mathrm{nc}}$ on $X$, and $\cl x_0$ lives on the smooth locus of $\fX'_s$. This finishes the proof of the lemma, which in turn implies Lemma \ref{lemma:dodging_contracting_locus_after_composition}. 
\end{proof}


\subsection{Dodging the contracting locus on curves over $\P^1$} Lemma \ref{lemma:dodging_contracting_locus_after_composition2} asserts that, for a map from a smooth projective curve of good reduction to $\P^1_K$, Question \ref{question:dodging_contracting_locus} has a positive answer once we are allowed to modify the target by a further automorphism. For more general curves, we prove a weaker statement that suffices for our applications. Roughly speaking, we prove that for a map from an arbitrary smooth projective curve $Y$ to $\P^1_K$, the effect of ``dodging the contracting locus'' can be achieved after certain ``controlled modifications'' of $Y$. 
For convenience, let us introduce the following terminology, borrowed from \cite{DDMY}. 

\begin{definition}
   Let $X$ be a smooth rigid analytic space over $K$, view as an adic space. Let $U \subset X$ be an open subset and $f_U: V \ra U$ be a finite cover between smooth rigid analytic spaces over $K$. A \textit{spread-out of $f$ across} $X$ is a finite cover $f: Y\ra X$ of smooth rigid analytic spaces defined over $K$ such that $f_U$ agrees with the pullback of $f$ along $U\subset X$. 
\end{definition}

\begin{theorem} \label{theorem:modified_spreadout_for_dodging}
Let $Y$ be a smooth projective curve over $K$, viewed as an adic space, and let $g: Y \ra \P^1_K$ be a finite cover. Let $\xi_{0, \C_p} \in \P^1_{\C_p}$ be a non-type I point and let $s_0\in \P^1_K$ be a $K$-rational classical point. Then there exist
\begin{itemize}
\item a finite extension $L/K$;
\item an open neighborhood $U_0\subset \P^1_L$ of the image $\xi_{0,L}$ of $\xi_{0,\C_p}$ in $\P^1_L$ such that $s_0\notin U_0$ \footnote{Since $s_0$ is a $K$-rational point, it can be viewed as an $L$-rational point in $\P^1_L$ by Remark \ref{remark:K_rational_point_under_field_extension}.} and such that the pullback $g_{U_0}: g^{-1}(U_0)\ra U_0$ of $g$ along $U_0\hookrightarrow \P^1_L\ra \P^1_K$ is finite \'etale;
\item a finite cover $g': Y'\ra \P^1_L$ between smooth projective curves over $L$ such that $f$ is a spread-out of $g_{U_0}$;
\item a morphism of $p$-adic formal schemes
\[\widehat{g}': \fY' \ra \widehat{\P}^1_{\mO_L}\]
over $\mO_L$ which induces $g'$ on the generic fibers, where $\fY'$ is a semistable model of $Y'$ over $\mO_L$,
\end{itemize}
such that there exists a classical point $y_0 \in (g')^{-1} (s_0)$ that lies on the non-contracting locus $Y'_{\mathrm{nc}}$ of $Y'$ with respect to $\widehat{g}'$. Moreover, we arrange it so that $y_0$ specializes to a point on the smooth locus of the special fiber $\fY'_s$ of $\fY'$.
\end{theorem}

Below is a cartoon of this ``dodging the contracting locus'' result.  
\begin{figure}[h]
\includegraphics[scale=0.8]{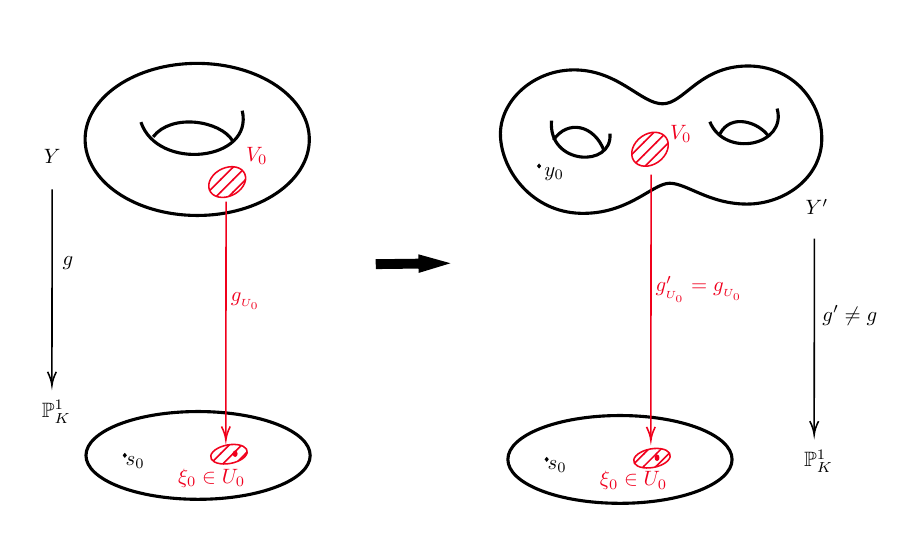}
    \caption{A cartoon illustration for dodging contracting locus over $\P^1$}
    \label{fig:dodging}
\end{figure} 
On the left-hand side, we are given a classical point $s_0$ and a non-type I point $\xi_0$ on $\P^1_K$, as well as a finite cover $g:Y\ra \P^1_K$ which restricts to $g_{U_0}: V_0\ra U_0$ (the red-shaded area) where $U_0$ is an open neighborhood of $\xi_0$. Then Theorem \ref{theorem:modified_spreadout_for_dodging} asserts that, by enlarging $K$ and shrinking $U_0$ if necessary, we can find another finite cover $g':Y'\ra \P^1_K$ while keeping $g_{U_0}$ unchanged, but this time $s_0$ admits a pre-image $y_0$ that dodges the contracting locus. This is depicted on the righ-hand side of the cartoon.

We will prove Theorem \ref{theorem:modified_spreadout_for_dodging} in \S \ref{subsection: dodging argument}. To this end, we need some preparations.

\subsection{Standard open neighborhoods of non-type I points}\label{sss:choice_of_nbhd}  

We first specify certain ``standard'' open neighborhoods of non-type I points on $\P^1_K$. For this subsection, we let
\begin{itemize}
    \item $\xi_0 \in \P^1_{K}$ be an arbitrary point of type I\!I, I\!I\!I, I\!V, or V (when $\xi_0$ is of type I\!I, I\!I\!I, or V, we further assume that $\xi_0$ is $K$-rational); 
    \item $s_1 \in \P^1_K$ be an arbitrary $K$-rational classical point. 
\end{itemize}  
Consider the following open neighborhoods of $\xi_0$ in $\P^1_{K}$, referred to as \emph{standard open neighborhoods}. We will constantly refer to these neighborhoods in the rest of \S \ref{sec:dodging}. 

\begin{enumerate}
    \item Suppose that $\xi_0$ is a type I\!I point which is $K$-rational. Then 
     there exists a coordinate $t$ on $\P^1_K$ such that $t(s_1) = 0$ and $\xi_0$ is the Gaussian point of the closed unit disc $\D_K(0, 1)$. In this case, a basis of open neighborhoods of $\xi_0$ (in $\P^1_K$) is given by $U_0$ of the form 
    \begin{equation}
    \label{eq:U_0_type_2}
    U_0 = \big\{|t| = 1\big\} - \bigcup_{\substack{\D_K (a_i, r_i) \subset \{|t| = 1\}\\ i = 1, ..., m}} \D_K (a_i, < r_i)
     \end{equation}
     with $r_i\in p^{\Q}$; namely, the complement of finitely many open discs in the thin annulus $\{|t| = 1\}$. We denote this basis of open neighborhoods by $\Sigma_{\xi_0}$. 
    \item Suppose that $\xi_0$ is a $K$-rational type I\!I\!I point. Then 
    there exists a coordinate $t$ on $\P^1_K$ such that $t(s_1) = 0$ and $\xi_0$ corresponds to $\D_K (0, r)$ for some $r \notin p^{\Q}$. In this case, we consider the basis of open neighborhoods $\Sigma_{\xi_0}$ of $\xi_0$ (in $\P^1_K$) consisting of closed annuli of the form 
  \begin{equation} 
    \label{eq:U_0_type_3} 
    U_0 = \big\{  r_1 \le |t| \le r_2 \big\} = \D_K(0, r_2) - \D_K (0, < r_1), 
  \end{equation}
    where $0 < r_1 < r < r_2, \mathrm{ and }  r_1, r_2 \in p^{\Q}$. 
    \item Suppose that $\xi_0$ is a type I\!V point. Then 
    there exists a coordinate $t$ on $\P^1_K$ such that $t (s_1) = 0$ and $\xi_0$ is represented by a sequence of nested closed discs 
    \begin{equation}\label{eq:U_0_type_4}
    \D_i =  \D_K (a_i, r_i) 
    \end{equation}  (with $r_i\in p^{\Q}$) whose intersection contains no type I point. In this case, such a choice of closed discs $\D_K (a_i, r_i)$ forms a basis of open neighborhoods of $\xi_0$ (in $\P^1_K$), which we fix and denote again by $\Sigma_{\xi_0}$.  
    \item Finally suppose that $\xi_0$ is a $K$-rational type V point. Let $\xi_{\mathrm{I\!I}} $ be the unique ($K$-rational) type I\!I point that specializes to $\xi_0$. As in Case (1) above, let $t$ be a coordinate on $\P^1_K$ such that $t(s_1) = 0$ and $\xi_{\mathrm{I\!I}}$ is the Gaussian point of $\D_K (0, 1)$. From the classification of points on $\P^1$, we know that $\xi_0$ falls into one of the following two classes: either $\xi_0=x_{a, <1}$ for some $a \in K$ with either $a = 0$ or  $|a| = 1$, or $\xi_0=x_{0, >1}$.
    \begin{itemize}
        \item If $\xi_0=x_{a, <1}$, we consider a basis of neighborhoods of $\xi_0$ (in $\P^1_K$) of the form  
   \begin{equation}
    \label{eq:U_0_type_5_a} 
        U_0 = \big\{r \le |t| \le 1 \big\} - \D_K (a, < r_a)  - \bigcup_{\substack{\D_K(a_i, r_i)\subset \D_K(0,1)\\ \D_K (a_i, r_i) \cap \D_K (a, < 1) = \emptyset\\ i = 1, ..., m}} \D_K (a_i, < r_i)
  \end{equation}
        where $r, r_a\in (0,1)\cap p^{\Q}$ and $r_i\in (0,1]\cap p^{\Q}$. 
        \item If $\xi_0=x_{0, >1}$, we consider a basis of neighborhoods of $\xi_0$ (in $\P^1_K$) of the form 
    \begin{equation}
    \label{eq:U_0_type_5b} 
        U_0 = \D_K (0, r) - \bigcup_{\substack{\D_K (a_i, r_i) \subset  \D_K (0, 1)\\ i = 1, ..., m}} \D_K (a_i, < r_i) 
     \end{equation}
        where $r\in (1,\infty)\cap p^{\Q}$ and $r_i\in (0,1]\cap p^{\Q}$. 
    \end{itemize} 
  We again denote such a basis by $\Sigma_{\xi_0}.$
\end{enumerate}

\subsection{Modification of coordinates}\label{subsection: modification of coordinates}
We resume the setup of \S \ref{sss:choice_of_nbhd}, namely, $\xi_0\in \P^1_K$ is a point of type I\!I, I\!I\!I, I\!V, or V (when $\xi_0$ is of type I\!I, I\!I\!I, or V, we further assume that $\xi_0$ is $K$-rational) and $s_1\in \P^1_K$ is a $K$-rational classical point. The discussion above provides a coordinate $t$ on $\P^1_K$ such that $t(s_1)=0$, together with a standard basis of open neighborhoods $\Sigma_{\xi_0}$ of $\xi_0$ in $\P^1_K$. The purpose of the following lemma is to find another coordinate satisfying some additional technical property, which will become convenient in the proof of Theorem \ref{theorem:modified_spreadout_for_dodging}.

\begin{lemma} \label{lemma:smart_choice_of_coordinates_for_dodging} 
Let $\xi_0, s_1\in \P^1_K$ be as above and let $U_0 \in \Sigma_{\xi_0}$. Let $s_0\in \P^1_K$ be a $K$-rational classical point different from $s_1$. Assume that $s_0, s_1\notin U_0$. Then, up to replacing $U_0$ by a smaller open neighborhood of $\xi_0$ in $\Sigma_{\xi_0}$ if necessary, there exists a coordinate $t'$ on $\P^1_K$ such that 
\begin{enumerate}
    \item $t' (s_1) = 0$; 
    \item $t' (s_0) = a$ for some $a \in K$; and 
    \item one of the following holds: 
    \begin{itemize}
        \item either there exists $0 < \delta < 1$ such that $ |t' (u)| \le  \delta \cdot |a| < |a| $ for all $u \in U_0$,
        \item or there exists $0 < \delta < 1$ such that $ |a| \le   \delta \cdot  |t' (u) | < |t' (u)|$ for all $u \in U_0$.
    \end{itemize}
\end{enumerate} 
\end{lemma}


\begin{proof} 
We shall modify the coordinate $t$ to obtain the new coordinate $t'$. Let us proceed case by case according to the type of the point $\xi_0$. 
\vspace*{0.2cm}

\noindent \textbf{Type I\!I.}  \\ 
Let $t$ be as in \S \ref{sss:choice_of_nbhd} (1) and $U_0$ be as in (\ref{eq:U_0_type_2}). In particular, $\xi_0$ is the Gaussian point of $\D_K(0,1)$ (in coordinate $t$), $t(s_1) = 0 $, and $|t (u)| = 1$ for all $u \in U_0$. If $|t(s_0)| \ne 1$ then we can simply take $t' = t$ and we are done. Now suppose that $t(s_0) = a_0$ for some $a_0 \in K$ such that $|a_0| = 1$. By shrinking $U_0$ if necessary, we may assume that 
\[ 
\D_K (a_0, < 1) \cap U_0  = \emptyset.
\]
Now let us take 
\[ t' = \frac{t}{t- a_0 + p},\] then by construction we have 
\begin{itemize}
    \item $t' (s_1) = 0$; 
    \item $|t' (s_0)| = |a_0/p| = p$; 
    \item $|t'(u)| = 1$ for all $u \in U_0$. 
\end{itemize}
One checks that $\xi_0$ is still the Gaussian point of $\D_{K}(0,1)$ under the new coordinate $t'$. Equivalently, the automorphism $t\mapsto \frac{a_0 t}{t- a_0 + p}$ of $\P^1_K$ fixes $\xi_0$. Moreover, the open neighborhood $U_0$ is still of the form (\ref{eq:U_0_type_2}) with respect to the new coordinate $t'$.

\vspace*{0.3cm}

\noindent \textbf{Type I\!I\!I.}  \\ 
Let $U_0 = \{ r_1 \le |t| \le r_2 \}$ be as in (\ref{eq:U_0_type_3}). Let $a = t(s_0)$. Since $s_0 \notin U_0$, we have either $|a| > r_2$ or $|a| < r_1$. So we simply take $t' =t$.

\vspace*{0.3cm}

\noindent \textbf{Type I\!V.}  \\ 
As in \S \ref{sss:choice_of_nbhd} (3),  $U_0 $ is one of the $\D_K (a_i, r_i) \in \Sigma_{\xi_0}$ under the specified coordinate $t$ therein. By increasing the index $i$ if necessary, we may assume that $|t(u)|$ is constant for all $u \in U_0$. Let $a_0 = t(s_0)$. If $|t(s_0)| = |a_0| \ne |t(u)|$, then we are done. Otherwise, we have $|t(u)| = |t(s_0)| = |a_0|$ for all $u \in U_0$. We proceed in a similar way as in the case of Type I\!I points. This time we pick a sufficiently large $m$ such that $\D_K (a_0, p^{-m}) \cap U_0 = \emptyset$ and consider a new coordinate \[t' = \frac{t}{t - a_0 + p^{m+1}}.\] Then we have 
\begin{itemize}
    \item $t' (s_1) = 0$; 
    \item $|t' (s_0)| = |a_0/p^{m+1}| = p^{m+1} \cdot |a_0|$; 
    \item $|t'(u)| \le p^m \cdot |a_0|$ for all $u \in U_0$, 
\end{itemize}
as desired. One checks that $U_0$ is still of the form (\ref{eq:U_0_type_4}) in the new coordinate $t'$.

\vspace*{0.3cm}

\noindent \textbf{Type V.} \\ 
Let $\xi_{\mathrm{I\!I}}$ be the unique type I\!I point specializing to $\xi_0$. As discussed in \S \ref{sss:choice_of_nbhd} (4), we choose a coordinate $t$ such that $t(s_1)=0$ and $\xi_{\mathrm{I\!I}}$ is the Gaussian point of $\D_K(0,1)$. Then there are two cases, in which the standard open neighborhoods $\Sigma_{\xi_0}$ have the forms (\ref{eq:U_0_type_5_a}) and (\ref{eq:U_0_type_5b}), respectively. In either cases, for a fixed $0 < \epsilon < 1$, upon replacing $U_0$ by a smaller neighborhood in $\Sigma_{\xi_{0}}$ if necessary, we may assume that 
\begin{equation} \label{eq:bound_type_5_by_epsilon_strip}
\epsilon \le |t(u)| \le 1/\epsilon  
\end{equation}
for all $u \in U_0$. Let $a_0 = t(s_0)$. Again, if $|a_0| \ne 1$, we may choose an $\epsilon \in (0, 1)$ such that either $|a_0| < \epsilon$ or $|a_0| > 1/\epsilon$. Then we are done. 

From now on, let us suppose that $|a_0| = 1$. First consider the case where $\xi_0=x_{0,<1}$.  In this case, we may choose $U_0$ of the form (\ref{eq:U_0_type_5_a}) such that 
\[ 
\D_K (a_0, < 1) \cap U_0 = \emptyset 
\] 
and the proof in the case of type I\!I points above applies \textit{verbatim}. In particular, we use the new coordinate $t' = \frac{t}{t - a_0 + p}$ in this case.

Next we treat the case where $\xi_0=x_{a, <1}$ for some $a \in K$ with $|a| = 1$. Now let us pick a choice of $U_0$ in (\ref{eq:U_0_type_5_a}), and a sufficiently large $m$ such that 
\[ 
\D_K (a_0, p^{-m}) \cap U_0 = \emptyset. 
\] 
Consider the new coordinate 
\[ 
t'  = \frac{t}{t - a_0 + p^{m+1}}.
\] 
Then we have 
\begin{itemize}
    \item $t' (s_1) = 0$; 
    \item $|t' (s_0)| = |a_0 /p^{m+1}| =  p^{m+1}$; 
    \item $1 \le |t'(u)| \le p^m$ for all $u \in U_0$. 
\end{itemize}

Finally, suppose that $\xi_0=x_{0,>1}$. Then we may pick $U_0$ to be of the form 
\[
U_0 = \{1 \le |t| \le r\} - \D_K(a_0, < 1) - \bigcup_{\substack{\D_K (a_i, r_i) \subset  \{|t|=1\}\\ i = 1, ..., m}} \D_K (a_i, < r_i) 
\] 
where $r\in (1,\infty)\cap p^{\Q}$ and $r_i\in (0,1]\cap p^{\Q}$. Pick $a_1\in K$ such that $|a_1|\le 1/r$. Now consider the new coordinate 
\[
t' = \frac{t}{ t - a_0 + a_1}.
\]
We compute that 
\begin{itemize}
    \item $t' (s_1) = 0$; 
    \item $|t' (s_0)|  =  1/|a_1|\ge r $; 
    \item $|t'(u)| = 1$ for all $u \in U_0$. 
\end{itemize}
We point out that, in all situations, $\xi_{\mathrm{I\!I}}$ is still the Gaussian point of $\D_K(0,1)$ under the new coordinate $t'$. Equivalently, the automorphism $t\mapsto t'$ of $\P^1_K$ fixes $\xi_{\mathrm{I\!I}}$ (but does not necessarily fix $\xi_0$). This finishes the proof for type V points and completes the proof of the lemma. 
\end{proof}

\subsection{The dodging argument}\label{subsection: dodging argument}

Now we proceed to prove the first main result of this section.  

\begin{proof}[Proof of Theorem \ref{theorem:modified_spreadout_for_dodging}.]
Let $Y, g, \xi_{0,\C_p}$ and $s_0$ be as in the statement of the theorem. Notice that the statement is insensitive to replacing $K$ by a finite extension. By enlarging $K$ and passing to a connected component, we may assume that $Y$ is geometrically connected. By further enlarging $K$ if necessary, we may assume that $\xi_{0, \C_p}$ is $K$-rational when it has type I\!I, I\!I\!I, or V. 

We will prove the following more precise statement: there exist
\begin{itemize}
\item a finite extension $L/K$;
\item an open neighborhood $U_0\subset \P^1_L$ of the image $\xi_{0,L}$ of $\xi_{0, \C_p}$ in $\P^1_L$, such that $s_0\notin U_0$ and such that the pullback $g_{U_0}: g^{-1}(U_0)\ra U_0$ of $g$ along $U_0\hookrightarrow \P^1_L\ra \P^1_K$ is finite \'etale;
\item a finite cover $g': Y'\ra \P^1_L$ between smooth projective curves over $L$ such that $g'$ is a spread-out of $g_{U_0}$;
\item a morphism of $p$-adic formal schemes
\[\widehat{g}': \fY' \ra \mS'\]
over $\mO_L$ which induces $g'$ on the generic fibers, where $\fY'$ (resp. $\mS'$) is a semistable formal model of $Y'$ (resp. $\P^1_L$) over $\mO_L$,
\end{itemize}
such that 
\begin{itemize} 
\item there exists a classical point $y_0 \in (g')^{-1} (s_0)$ that lies on the non-contracting locus of $Y'$ with respect to $\widehat{g}'$, and that $y_0$ specializes to a point on the smooth locus of the special fiber $\fY'_s$ of $\fY'$;
\item $\mS'$ is obtained from $\widehat{\P}^1_{\mO_L}$ via an admissible blowing-up $\pi: \mS' \ra \widehat{\P}^1_{\mO_L}$ along some closed points on the special fiber of $\widehat{\P}^1_{\mO_L}$ that are disjoint from $\cl s_0$, where $\cl s_0$ is the image of $s_0$ under the specialization map.
\end{itemize}
Then we simply take $\widehat g' = \pi \circ \widehat{g}_1'$ to arrive at the conclusion asserted by the theorem.  

First, there exists a finite extension $L/K$ such that the base change $g_L:Y_L\ra \P^1_L$ of $g$ admits an integral model $\widehat g: \fY \ra \widehat{\P}^1_{\mO_L}$ over $\mO_L$ where $\fY$ is a semistable formal model of $Y_L$. Since the non-contracting locus $Y_{L,\mathrm{nc}}$ of $Y_L$ (with respect to $\widehat{g}$) is non-empty, by replacing $L$ by a further finite extension if necessary, we can pick an $L$-rational classical point $y_1 \in Y_{L,\mathrm{nc}}$ such that its image $\cl{y}_1$ under the specialization map lies on the smooth locus of $\fY_s$.\footnote{After enlarging $L$, we might need additional admissible blowing-up on $\fY$ to get a semistable model. But by Remark \ref{remark:enlarging_K}, $y_1$ still lies on the non-contracting locus with respect to the new semistable model because it specializes to a point on the smooth locus of the special fiber. We will use the trick several times throughout this proof.} We assume that $y_1\neq y_0$, otherwise there is nothing to prove. Let $s_1 = g_L(y_1)$ be its image in $\P^1_L$. 

Pick an open neighborhood $U_0$ of $\xi_{0,L}$ in $\P^1_L$ such that $s_0, s_1 \notin U_0$. By Lemma \ref{lemma:smart_choice_of_coordinates_for_dodging}, up to shrinking $U_0$, there exists a coordinate $t$ on $\P^1_L$ such that 
\begin{itemize}
\item $t (s_1) = 0$, $t (s_0) = a$ for some $a \in L$;
\item $U_0$ is a standard open neighborhood of $\xi_{0,L}$ in the sense of \S \ref{sss:choice_of_nbhd};
\item either of the following holds: 
    \begin{enumerate}
        \item there exists $0 < \delta < 1$ such that $|t(u)| \le \delta \cdot |a|$ for all $u \in U_0$; 
        \item there exists $0 < \delta < 1$ such that $|a| \le \delta \cdot |t(u)|$ for all $u \in U_0$.
    \end{enumerate}
\end{itemize}
By further shrinking $U_0$, we can make sure that the base change $g_{U_0}: g^{-1}(U_0)\ra U_0$ of $g$ along $U_0\hookrightarrow \P^1_L\ra \P^1_K$ is finite \'etale (hence $g_L:Y_L\ra \P^1_L$ is a spread-out of $g_{U_0}$) and we may assume that $g_{U_0}$ is given by a monic polynomial $F[Z] \in \mO(U_0) [Z]$, which is separable in $\mO_{{U_0}, \xi_{0,L}}[Z]$. (Notice that $\mO_{U_0, \xi_{0,L}}$ is a field as $\xi_{0,L}$ is non-type I.) We can write 
\[
F(Z) = F(Z, t) = \sum_{i = 0}^d a_i (t) Z^i
\]
with $a_i (t) \in \mO (U_0)$ and $a_d(t) = 1$. By further shrinking $U_0$ if necessary, $F(Z, t)$ can be factored as
\[F(Z, t)= \prod_{i=1}^m F_i(Z,t)\]
where all of $F_i(Z,t)\in \mO(U_0) [Z]$ are monic polynomials, such that each $F_i(Z,t)$ is irreducible in $\mO_{U_0, \xi_{0,L}}[Z]$.

Now we split the proof into two cases, according to the case (1) and (2) above.
\vspace*{0.2cm}

\noindent 
\textbf{Case (1)}. \\ 
For each $n \ge 0$, consider the automorphism $\P^1_L \ra \P^1_L$ given by 
\[ 
h_n: t \mapsto t - \frac{t^{n+1}}{a^n}. 
\] 
Notice that $h_n$ sends $s_0$ to $s_1$. Since $|\frac{t^{n+1}}{a^n}| \ra 0$ uniformly on $U_0$ as $n \ra \infty$, we have $U_0\subset (h_n)^{-1} (U_0)$ for $n$ sufficiently large.

By Lemma \ref{lemma:p_adic_perturbation} below, up to shrinking $U_0$, the finite cover of $U_0$ defined by $F_i(Z,t)=0$ does not change after a small $p$-adic perturbation of the coefficients of $F_i(Z)$. Consequently, up to shrinking $U_0$, the equation
\[ \sq F (Z, t) := F\big(Z, t - \frac{t^{n+1}}{a^n}\big)=\prod_{i=1}^m F_i\big(Z, t - \frac{t^{n+1}}{a^n}\big)=0\] 
defines the same finite cover $g_{U_0}: g^{-1}(U_0)\ra U_0$ as $F(Z, t)=0$.

Consider the normalized base change 
\begin{equation} \label{diagram:base_change_of_g_by_smart_choice_of_coordinate}
\begin{tikzcd}
    Y' \arrow[d, swap, "g'"]  \arrow[r] & Y_L \arrow[d, "g_L"]  \\ 
    \P^1_L \arrow[r, "h_n"]  & \P^1_L 
\end{tikzcd}
\end{equation}
which produces a geometrically connected smooth projective curve $Y'$ over $L$ (after further enlarging $L$ and restricting to a connected component if necessary). By construction, the map 
\[ g': Y' \ra \P^1_L
\] is indeed a spread-out of $g_{U_0}$. 

Now we claim that, the map $h_n$ admits an integral model 
\[ 
\widehat h_n: \mS' \ra \widehat{\P}^1_{\mO_L} 
\]
where, on the source of the map, $\mS'$ is obtained as an admissible blowing-up $\pi$ of (another copy of) $\widehat{\P}^1_{\mO_L}$ along a closed subscheme of the special fiber of $\widehat{\P}^1_{\mO_L}$ that is disjoint from $\cl s_0$. In particular, we have a diagram
\[ 
\begin{tikzcd}
    \mS' \arrow[r, "\widehat h_n"] \arrow[d, "\pi"] & \widehat{\P}^1_{\mO_L} \\ 
    \widehat{\P}^1_{\mO_L}
\end{tikzcd}
\]
where $\pi$ is an admissible blowing-up, and $s_0$ lives on the non-contracting locus of $\P^1_L$ (on the source of $h_n$) with respect to $\widehat h_n$. To prove the claim, we modify the coordinate on $\P^1_L$ in a similar way as in the proof of \ref{lemma:dodging_contracting_locus_after_composition2}. Let $t'' = p^{-N} (t- a)$ be a new coordinate on $\P^1_L$ (on the source of $h_n$) where $N$ is a sufficiently large integer that we fix later. We have $t''(s_0)=0$. Under this new coordinate, the map $h_n$ is given by 
\begin{align*}
    t \: \:  \longmapsto \: \: t - \frac{t^{n+1}}{a^n}  \quad &   = \: \:  p^N t'' + a - \frac{1}{a^n} \Big( p^N t'' + a \Big)^{n+1}  \\ 
    & =  \: \: - \Big(  p^N t'' + a \Big) \cdot \left( \sum_{i = 1}^n \:  {n \choose i} \Big( \frac{p^N t''}{a} \Big)^i   \right).
\end{align*}
For $N$ sufficiently large, we can make sure all the coefficients of $(t'')^i$ in this expansion lie in $\mO_L$. Now, the claim follows from a similar argument as in the proof of \ref{lemma:dodging_contracting_locus_after_composition2}. 

Given the claim, we are ready to finish the proof of the theorem. Consider the fiber product \[\fY'_{\mathrm{naive}}=\fY \times_{\widehat{\P}^1_{\mO_L}} \mS'\] in the category of $p$-adic formal schemes. By enlarging $L$ if necessary, $Y'$ admits a semistable formal model $\fY'$ obtained by successive admissible blowing-ups on $\fY'_{\mathrm{naive}}$ (so that the blowing-ups avoid the smooth locus of the special fiber of $\fY'_{\mathrm{naive}}$). This induces a morphism 
\[\widehat{g}':\fY'\ra \mS'\]
which induces $g: Y'\ra \P^1_L$ on the generic fibers, such that 
there exists a classical point $y_0 \in (g')^{-1} (s_0)$ that lies on the non-contracting locus of $Y'$ with respect to $\widehat{g}'$, as required. In fact, by our choice of the point $y_1$ at the beginning of the proof, we know that the point $y_0$ can be chosen such that its image $\cl y_0$ under the specialization map $Y' \ra \fY'_{s}$ lies on the smooth locus of $\fY'_{s}$.

This finishes the proof of Case (1). \\ 

\noindent 
\textbf{Case (2)}. \\ 
The argument is similar to Case (1). Instead of $h_n$, here we consider the map on $\P^1_L$ given by 
\[ 
h_n': t \mapsto t - \frac{a^n}{t^{n-1}}
\]
for each $n$ sufficiently large. The rest of the argument follows \textit{verbatim}. To find the desired integral model $\widehat h_n': \mS' \ra \widehat{\P}^1_{\mO_L}$, we again consider a new coordinate $t'' = p^{-N} (t - a)$. Under this coordinate, $h_n'$ becomes 
\[ 
t \: \: \longmapsto  \: \: t - \frac{a^n}{t^{n-1}} \:  = \:   \frac{a}{\big(1+ \frac{p^N t''}{a} \big )^{n-1}} \cdot \left( \sum_{i = 1}^n \:  {n \choose i} \Big( \frac{p^N t''}{a} \Big)^i   \right),
\]
which again has integral coefficients when $N$ is chosen large enough. This handles Case (2) and finishes the proof of the theorem.  
\end{proof}

The following result is used in the proof above. 

\begin{lemma}[$p$-adic perturbation of finite covers of affinoids]\label{lemma:p_adic_perturbation}
Let $Y$ be an affinoid smooth rigid analytic curve over $K$, viewed as an adic space. Let $\xi$ be a non-type I point on $Y$ and let $V = \spa (A, A^+)$ be an open neighborhood of $\xi$ that is a rational subset of $Y$. Let $V'$ be a finite cover of $V$ defined by a monic irreducible polynomial 
\[ F(Z) =  Z^d +  a_{d-1} Z^{d-1} + \cdots    +  a_0 \] where $a_i \in A$, which remains irreducible as a polynomial over $\mO_{Y, \xi}$. Then there exists sufficiently small $\epsilon \in \R_{> 0}$ such that the following holds: for any polynomial
\[ 
 G(Z) =  Z^d +  b_{d-1} Z^{d-1} + \cdots + b_0 
\]
with $b_i\in A$, satisfying $|a_i - b_i|_{V} < \epsilon$ for the spectral norm $|\cdot |_V$ on $V$, up to replacing $V$ by a smaller rational subset containing $\xi$ if necessary, $G(Z)$ defines the same finite cover $V' \ra V$.     
\end{lemma}

\begin{proof} It suffices to show that, as long as $\epsilon$ is sufficiently small, $F(Z)$ and $G(Z)$ define the same extension of the stalk $\mO_{Y, \xi}$. 
  First suppose that $\xi$ has type I\!I, I\!I\!I, or I\!V. Then the natural map $\mO_{Y, \xi} \ra \mO_{Y, \xi}^{\wedge} = \sH (\xi)$ from the (adic) stalk to its completion under the seminorm given by $\xi$ induces an equivalence  between the categories of finite separable extensions of $\mO_{Y, \xi}$ and that of $\sH(\xi)$. This follows, for example, from considering the Berkovich stalk $\mO_{Y^{\mathrm{Ber}}, \xi}$ and the fact that it is \textit{quasicomplete} in the sense of \cite[Theorem 2.3.3, Proposition 2.4.1]{Berkovich_etale}, and using the fact that we have an identification  $\mO_{Y^{\mathrm{Ber}}, \xi}^{\wedge} = \mO_{Y, \xi}^{\wedge} = \sH (\xi)$  (see \S \ref{sss:comparison_between_stalks}). 
  Therefore, in this case, it suffices to show that for sufficiently small $\epsilon$, $F(Z)$ and $G(Z)$ define the same extension of $\sH(\xi)$. Note that the condition  $|a_i - b_i|_{V} < \epsilon$ in particular implies that $|a_i - b_i|_{\xi} < \epsilon$, thus the claim follows from Krasner's lemma since $\sH(\xi)$ is a complete Henselian field.
    
 Now suppose that $\xi$ has type V and let $\xi_{\mathrm{I\!I}}$ denote the unique type I\!I point that specializes to $\xi$. Then $\xi_{\mathrm{I\!I}} \in V$. By  \cite[Lemma 1.1.10]{Huber_etale} (also see Example \ref{example:stalk_on_curve}), we have an injective map $\mO_{Y, \xi} \hookrightarrow \mO_{Y, \xi_{\mathrm{I\!I}}}$ with dense image. To check that $F(Z)$ and $G(Z)$ defines the same extension of $\mO_{Y, \xi}$, it suffices to check this over $\mO_{Y, \xi_{\mathrm{I\!I}}}$. Notice that the condition  $|a_i - b_i|_{V} < \epsilon$ implies that $|a_i - b_i|_{\xi_{\mathrm{I\!I}}} < \epsilon$, thus we are reduced to the case of type I\!I points. This finishes the proof of the lemma. 
\end{proof}

\subsection{A spread-out theorem}\label{subsection:spread-out thm}

In \S \ref{ss:dodging_II}, we shall prove a generalization of Theorem \ref{theorem:modified_spreadout_for_dodging}. In the proof, we will need the following spread-out theorem. In fact, we only need the case of curves, but we present the theorem in full generality.

\begin{theorem}\label{thm:spread-out thm}
Let $X$ be a smooth projective variety over $K$ and let $X^{\mathrm{ad}}$ denote the associated adic space. Let $\xi$ be a non-type I point on $X^{\mathrm{ad}}$ and let $U \subset X^{\mathrm{ad}}$ be an affinoid open neighborhood of $\xi$. Let $f: V \ra U$ be a finite \'etale cover. Then, up to shrinking $U$, there exist
\begin{itemize}
\item a finite cover $V'\ra V$;
\item an alteration $\sq f: Y\ra X$ (in the sense of \cite{deJong_alteration}) such that 
\begin{enumerate}
\item the composition $V'\ra V\xrightarrow[]{f}U$ is the pullback of $\sq f$ along $U\subset X^{\mathrm{ad}}$,
\item there exists a finite extension $K'/K$ such that $Y_{K'}$ is smooth with semistable reduction over $\mO_{K'}$.
\end{enumerate}
\end{itemize}
Moreover, in the case of curves, we can take $V'=V$.
\end{theorem} 

To prove Theorem \ref{thm:spread-out thm}, a key ingredient we need is the following approximation result (Proposition \ref{prop:p_adic_Runge}). In the special case of curves, this is essentially a result of Raynaud \cite[Corollaire 3.5.2]{Raynaud},\footnote{The approximation result of Raynaud in \cite{Raynaud} is stated in the context of Berkovich curves. Note that our proof of Proposition \ref{prop:p_adic_Runge} below yields a similar statement for Berkovich spaces -- which is slightly different from Raynaud's result as an affinoid open subset in an adic space might not yield an open subset in the corresponding Berkovich space.} which he calls the \textit{$p$-adic Runge's Theorem}, as it can be viewed as an analogue of Runge's theorem in the theory of complex Riemann surfaces. 

\begin{proposition}\label{prop:p_adic_Runge}
Let $X$ be a smooth projective variety over $K$ and let $K(X)$ denote the field of rational functions on $X$. Let $\xi$ be a point on $X^{\mathrm{ad}}$ and let $U \subset X^{\mathrm{ad}}$ be an affinoid open neighborhood of $\xi$. Then, after shrinking the neighborhood $U$ if necessary, $K(X) \cap \Gamma (U, \mO_{X^{\mathrm{ad}}})$ (intersection taken in the fractional field of $\Gamma (U, \mO_{X^{\mathrm{ad}}})$) is dense in $\Gamma (U, \mO_{X^{\mathrm{ad}}})$ with respect to the topology induced from the spectral norm on $U$. 
\end{proposition}

\begin{proof} The idea is to reduce to the case of the projective spaces $\P^n_K$. Let $n$ be the dimension of $X$ over $K$. Let us pick  a finite surjective cover  $f: X \ra \P^n_{K}$ such that $f$ is smooth (thus finite \'etale) on an affinoid neighborhood containing $\xi$. To see such a cover exists, one can embed $X \subset \P_K^N$ as a closed subscheme where $N > n$, and consider a linear subspace $L \subset \P_K^{N}$ of complementary dimension (so $L$ is isomorphic to a copy of $\P_K^{N-n-1}$) such that $L$ is disjoint from $X$. The projection map from $\P_K^{N}$ that projects ``away from $L$'' determines a finite map $f: X \ra \P_K^n$. We have to arrange $L$ such that the singular locus does not contain $\xi$. To see that this is possible, let $\xi$ be a point of (cohomological) dimension $d_0$ and let $Z$ be a $d_0$-dimensional irreducible closed subscheme of $X$ that contains $\xi$ (i.e., such that $\xi \in Z^{\mathrm{ad}}$). We just need to arrange $f$ so that it is generically smooth on $Z$. Indeed, we pick a closed point $z \in Z$ and arrange the linear subspace $L$ so that it does not intersect the tangent space $T_z X$ of $X$ inside $\P_K^N$ (by comparing dimensions). The resulting finite map $f: X \ra \P_K^n$ is then smooth at the point $z \in Z$, so it is generically smooth on $Z$, as desired.

Now let $U$ be an affinoid open neighborhood of $\xi$ in $X^{\mathrm{ad}}$. By shrinking $U$, we may assume that $f$ is finite \'etale on $U$. Let $W$ be the image of $U$ under $f: X^{\mathrm{ad}}\ra \P_K^{n, \mathrm{ad}}$. In particular, since $f$ is smooth, $W$ is an open subset of $\P_K^{n, \mathrm{ad}}$ (for example, by \cite[Proposition 1.7.8]{Huber_etale}). Note that, in order to prove the proposition, it suffices to replace $U$ by a smaller open affinoid subset, so by possibly shrinking $W$ (and thus $U$) we may assume that $W$ is affinoid. We claim that, by further shrinking $U$ and $W$, we may assume that $U$ is a connected component of $f^{-1}(W) = W \times_{\P^{n,\mathrm{ad}}_K} X^{\mathrm{ad}}$. To see this, we reduce to the case where $U$ is irreducible and let $U_1$ be the irreducible component in $f^{-1} (W)$ that contains $U$ and let $Z = U_1 \minus U$ be the complement, since finite maps are proper (in the sense of adic spaces),  $f(Z) \subset W$ is closed and its complement $W' \subset W$ is a non-empty open subset. Let $U' = U \cap f^{-1} (W')$, then $U'$ is a connected component in $f^{-1}(W')$. By shrinking further if necessary, we may again assume that both $W'$ and $U'$ are affinoid open subsets of $\P^{n,\mathrm{ad}}_{K}$ and $X^{\mathrm{ad}}$, respectively. This finishes the proof of the claim. Given the claim, we further reduce to the case that $U = f^{-1}(W)$ without loss of generality. 

Consider the following commutative diagram
\[ 
\begin{tikzcd}
K(\P_K^n) \cap \Gamma(W, \mO_{\P_K^{n,\mathrm{ad}}}) \arrow[r, "\alpha"] \arrow[d] & \Gamma(W,  \mO_{\P_K^{n, \mathrm{ad}}}) \arrow[d]    \\     K(X) \cap \Gamma(U, \mO_{X^{\mathrm{ad}}}) \arrow[r, "\beta"]  &  \Gamma (U, \mO_{X^{\mathrm{ad}}}). 
\end{tikzcd}
\] 
Suppose that the map $\alpha$ has dense image, we would like to show that $\beta$ also has dense image. 
Let us consider an affine open subscheme $W_0 = \spec A_0 \subset \P^n_K$ such that $W \subset W_0$, and let $U_0 = W_0 \times_{\P_K^n} X$ be the pullback scheme. We may assume that $W_0$ is small enough so that $U_0 = \spec B_0$ for a finite \'etale algebra $B_0$ over $A_0$ (and we identify $A_0$ as a sub-algebra of $B_0$). In particular, there exist elements $x_1, \ldots, x_m \in B_0$ which generate $B_0$ as a module over $A_0$. Let $B = \Gamma(U, \mO_{X^{\mathrm{ad}}})$ and $A = \Gamma(W, \mO_{\P_K^{n, \mathrm{ad}}})$. Then $B$ is generated by $x_1, \ldots, x_m$ as an $A$-module. By scaling, we may assume that $x_1, \ldots, x_m$ all have spectral norm $\le 1$ on $U$. Now, for any $b \in B$, we can write $b = \sum_{i=1}^m a_i x_i$ with $a_i \in A$. By assumption, for any $\epsilon > 0$, there exists $a'_i \in K(\P_K^n) \cap \Gamma(W, \mO_{\P_K^{n, \mathrm{ad}}})$ such that $|a_i - a'_i|_W < \epsilon$. Take \[b' = \sum_{i=1}^m a'_i x_i \in K(X) \cap \Gamma (U, \mO_{X^{\mathrm{ad}}}),\] then $|b - b'|_U < \epsilon$. (Notice that $|a_i-a'_i|_U\le |a_i-a'_i|_W$.)

Consequently, we are reduced to proving the proposition for $\P^n_K$, which is the content of the next lemma. 
\end{proof}

\begin{lemma}
Let $U$ be an affinoid open neighborhood of some point $\xi \in \P^{n, \mathrm{ad}}_{K}$. Then, up to shrinking $U$, the intersection $K(\P_K^n) \cap \Gamma(U, \mO_{\P_K^{n, \mathrm{ad}}})$ is dense in $\Gamma(U, \mO_{\P_K^{n, \mathrm{ad}}})$ with respect to the spectral norm on $U$.
\end{lemma}

\begin{proof}
  We first place $U$ inside a standard $n$-dimensional unit ball $\mathbb{D}_K^n$ such that the coordinates $t_1, ..., t_n \in K(\P_K^n)$ are global rational functions on $\P_K^n$. Without loss of generality, we may assume that $U$ is a Laurent domain of the form \[U = \{z \in \mathbb{D}_K^n \: | \:| f_i(z)| \le a_i, |g_j(z)|\ge b_j\}\] for some $f_i, g_j \in \mO(\mathbb{D}_K^n)$ and $a_i, b_j\in p^{\Q}$. Now we may replace $f_i, g_j$ by $f_i', g_j'$ which are polynomials in $t_1, ..., t_n$ while keeping 
\[U = \{z \in \mathbb{D}_K^n \: | \:| f'_i(z)| \le a_i, |g'_j(z)|\ge b_j\}.\] In particular, these $f'_i$'s and $g'_j$'s are rational functions on $\P_K^n$, in which case, the claim becomes evident.
\end{proof}
 

\begin{proof}[Proof of Theorem \ref{thm:spread-out thm}]
By shrinking $U$ (and hence $V$) and passing to connected components of $V$, we may assume that \[\Gamma(V, \mO_{X^{\mathrm{ad}}}) = \Gamma(U, \mO_{X^{\mathrm{ad}}})[T]/F(T)\] for a monic polynomial $F(T) = \sum_{i=0}^l\beta_i T^i$ with coefficients $\beta_i \in \Gamma(U, \mO_{X^{\mathrm{ad}}})$ (in particular, $\beta_l=1$) such that $F(T)$ is irreducible in $\mO_{X^{\mathrm{ad}}, \xi}[T]$.  By Proposition \ref{prop:p_adic_Runge}, upon further shrinking $U$ if necessary, for any $\epsilon > 0$, there exists $\sq \beta_i\in K(X)\cap \Gamma(U, \mO_{X^{\mathrm{ad}}})$ such that $|\beta_i-\sq \beta_i|_U<\epsilon$ for all $0\le i\le l-1$ (set $\sq \beta_l=1$), where $|\cdot|_U$ stands for the spectral norm on $U$. Applying Lemma \ref{lemma:p_adic_perturbation}, we know that, for $\epsilon$ sufficiently small, up to further shrinking $U$, the polynomials $F(T)$ and $\sq F(T)=\sum_{i=0}^l \sq \beta_i T^i$ define the same finite \'etale cover $V\ra U$. Now, since $\sq F(T) \in K(X)[T]$, it defines a separable extension $L/K(X)$. Let $Y'$ denote the normalization of $X$ in $L$, so in particular, we have $V \cong Y' \times_X U$. Let $Z^{\mathrm{sing}} \subset Y'$ denote the singular locus (which is disjoint from $V$). Then by the work of de Jong (\cite[Theorem 4.1 and Theorem 6.5]{deJong_alteration}, and the proof explained in Section 6.7 -- 6.16 of \textit{loc.cit.}), there exists an alteration $\sq g: Y \ra Y'$ where $Y$ is smooth and of semistable reduction over some finite extension $K'/K$,\footnote{More precisely, we first find an alteration $Y_1 \ra Y'$ with respect to the pair $(Y', Z^{\mathrm{sing}})$ to obtain a smooth projective scheme $Y_1$ over some finite extension $K_1/K$. Let $D_1 \subset Y_1$ be the pre-image of $Z^{\mathrm{sing}}$ in $Y_1$. Then we take an integral model $\fY_1/\mO_{K_1}$ of $Y_1$ and apply \cite[Theorem 6.5]{deJong_alteration} to the pair $(\fY_1, \fY_{1, s} \cup D_1)$ to obtain a projective scheme $\fY/\mO_{K'}$ with semistable reduction, for some finite $K'/K$. We then take $Y$ to be the generic fiber of $\fY$.} such that ${\sq g}^{-1} (Z^{\mathrm{sing}}) \subset Y$ is a union of normal crossing divisors and  $\sq g$ is finite flat over $Y' \minus Z$ for some proper closed subscheme $Z \subset Y'$ ($Z$ is called the \emph{center} of the alteration $\sq g$ in \cite{deJong_alteration}). Finally, we take $\sq f: Y \ra X$ to be the composition of $Y'\ra X$ with the alteration $\sq g$, and take $V'$ to be the pullback of $\sq g$ along $V \ra Y'$. Notice that, the alteration step is unnecessary if $X$ is a curve.
\end{proof}


\subsection{Dodging the contracting locus: Part II} \label{ss:dodging_II}

Theorem \ref{theorem:modified_spreadout_for_dodging} in fact implies the following more general result, which we will apply in the proof of Theorem \ref{thm:dodging}.  

\begin{theorem}\label{thm:dodging_contracting_locus_general}
Let $g: Y \ra X $ be a finite cover between smooth projective curves over $K$ (viewed as adic spaces) where $X$ has good reduction with a smooth integral model $\fX$ over $\mO_K$. Let $\xi_{\C_p}$ be a non-type I point on $X_{\C_p}$ and let $x_0\in X(K)$ be a closed $K$-point (viewed as a $K$-rational classical point on the associated adic space). Then there exist
\begin{itemize}
\item a finite extension $L/K$;
\item an open neighborhood $U\subset X_L$ of the image $\xi_L$ of $\xi_{\C_p}$ in $X_L$, such that $x_0\notin U$;\footnote{Since $x_0$ is a $K$-rational point, it can be viewed as an $L$-rational point in $X_L$ by Remark \ref{remark:K_rational_point_under_field_extension}.} 
\item a nonempty open subset $V$ of $g_L^{-1}(U)$ where $g_L: Y_L\ra X_L$ is the base change of $g$, such that $g_L|_V: V\ra U$ is a finite \'etale cover;
\item a finite \'etale cover $h:V'\ra V$;
\item a finite cover $g': Y'\ra X_L$ between smooth project curves over $L$ such that $g'$ is a spread-out of the composition \[V'\xrightarrow[]{h} V\xrightarrow[]{g_L|_V}U\] across $X_L$; 
\item a morphism of $p$-adic formal schemes 
\[\widehat{g}': \fY'\ra \fX_{\mO_L}\]
that induces $g'$ on the generic fibers, where $\fY'$ is a semistable model of $Y'$ over $\mO_L$,
\end{itemize}
such that there exists a classical point $y_0\in (g')^{-1}(x_0)$ living on the non-contracting locus of $Y'$ with respect to $\widehat{g}'$. Moreover, we can make sure that $y_0$ specializes to a point on the smooth locus of the special fiber $\fY'_s$ of $\fY'$.
\end{theorem}

\begin{proof} 
The statement is insensitive to replacing $K$ by a finite extension. Up to enlarging $K$ and passing to connected components, we may assume that both $X$ and $Y$ are geometrically connected, and that the image of $\xi_{\C_p}$ in $X$ is $K$-rational when it has type I\!I, I\!I\!I, or V. Let us first pick a finite map $f: X \ra \P^1_K$ that is \'etale at $x_0$. By Lemma \ref{lemma:dodging_contracting_locus_after_composition2}, up to enlarging $K$ if necessary, there exists a semistable model $\fX'$ of $X$, an automorphism $\alpha: \P^1_K \ra \P^1_K$, and an morphism of $p$-adic formal schemes  
\[ \widehat{f}': \fX' \ra \widehat \P^1_{\mO_K} \] which induces the composition $f' = \alpha \circ f$ on the generic fibers, such that $x_0$ lies on the non-contracting locus of $X$ with respect to $\widehat{f}'$, and $x_0$ specializes to a point on the smooth locus of the special fiber of $\fX'$.)

Let $\xi_K$ denote the image of $\xi_{\C_p}$ in $X$. Let $\xi_0 = f' (\xi_K)$ (resp. $s_0 = f' (x_0)$) be the image of $\xi_K$ (resp. $x_0$) in $\P^1_K$. Pick an affinoid open neighborhood $U_1$ of $\xi_K$ in $X$ and let $U_0\subset \P^1_K$ be the image of $U_1$ under $f'$. By shrinking $U_1$, we can make sure that $x_0\notin U_1$ and the restriction of $f'$ on $(f')^{-1}(U_0)$ is finite \'etale. 

Pick an arbitrary $\sq \xi\in Y$ in the pre-image of $\xi_K$. Since the points $\sq \xi, \xi_K, \xi_0$ are non-type I, we know that the stalks $\mO_{Y, \sq \xi}, \mO_{X, \xi_K}, \mO_{\P^1_K, \xi_0}$ are fields, and are equal to the residue fields $\mathbf k(\sq \xi), \mathbf k(\xi_K), \mathbf k(\xi_0)$, respectively (see \S \ref{sss:classification_of_points_on_P1} and \S \ref{sss:stalks}, and the references therein). Let us pick a finite extension $\sq K/ \mathbf k(\sq \xi)$ such that the separable extension $\mathbf k(\xi_K)/\mathbf k(\xi_0)$ (say, of degree $d$) splits completely over $\sq K$; in other words, the tensor product 
\[ 
\mO_{X, \xi_K} \otimes_{\mO_{\P^1_K, \xi_0}} \sq K = \mathbf  k (\xi_K) \otimes_{\mathbf k (\xi_0)} \sq K = \prod_{i=1}^d \sq K_i 
\] 
splits into a product of fields $\sq K_i$, with each $\sq K_i$ isomorphic to $\sq K$. By shrinking $U_1$ if necessary, the composition of field extensions $\mathbf k (\xi_0)\ra\mathbf k (\xi_K)\ra \mathbf k (\sq \xi)\ra \sq K$ spreads into a composition of finite \'etale covers 
\[g_{1,W}:  W \ra V_1\xrightarrow[]{g} U_1 \xrightarrow[]{f'} U_0, \]
where $V_1$ is an affinoid open neighborhood of $\sq \xi$ in $Y$, such that  
\begin{enumerate}
\item The fiber product $U_1 \times_{U_0} W$ splits as a disjoint union  
\[ 
U_1 \times_{U_0} W \cong \bigsqcup_{i = 1}^d W_i,
\] 
with each $W_i \cong W$ as adic spaces.
\item The natural morphism $U_1 \times_{U_0} W\ra U_1$ factors as \[U_1 \times_{U_0} W\cong \bigsqcup_{i = 1}^d W_i \ra U_1,\] where the last arrow is given by the composition of $W \ra V_1\ra U_1$ and the isomorphism $W_i \cong W$.
\end{enumerate}
Applying Theorem \ref{thm:spread-out thm}, by further shrinking $U_1$ (and hence shrinking $U_0$, $V_1$, and $W$) if necessary, there exists a finite cover $g_1:Y_1\ra \P^1_K$ of smooth projective curves over $K$ that spreads out $g_{1,W}:W\ra U_0$; namely, $g_{1,W}$ is precisely the base change of $g_1$ along $U_0\subset \P^1_K$.

Now, apply Theorem \ref{theorem:modified_spreadout_for_dodging} to the data $\left(g_1: Y_1\ra \P^1_K,\, \xi_{0,\C_p}, \, s_0\right)$ where $\xi_{0,\C_p}$ is the image of $\xi_{\C_p}$ under the base change $f'_{\C_p}: X_{\C_p}\ra \P^1_{\C_p}$. We know that there exist
\begin{itemize}
\item a finite extension $L/K$;
\item an open neighborhood $U'_0\subset \P^1_L$ of the image $\xi_{0,L}$ of $\xi_{0,\C_p}$ in $\P^1_L$\footnote{Equivalently, $\xi_{0,L}$ is the image of $\xi_{\C_p}$ under the composition $X_{\C_p}\xrightarrow[]{f'_{\C_p}} \P^1_{\C_p}\ra \P^1_L$} such that the image of $U'_0$ in $\P^1_K$ is contained in $U_0$ (in particular, $s_0\notin U'_0$ and the base change \[g_{1, W'}:W'\ra U'_0\] of $g_{1,W}$ along $U'_0\ra U_0$ is a finite \'etale cover);
\item a finite cover $g'_1: Y'_1\ra \P^1_L$ between smooth projective curves over $L$ such that $g'_1$ is a spread-out of $g_{1, W'}$;
\item a morphism of $p$-adic formal schemes 
\[\widehat{g}'_1: \fY'_1\ra \widehat{\P}^1_{\mO_L}\]
which induces $g'_1$ on the generic fibers, where $\fY'_1$ is a semistale model of $Y'_1$ over $\mO_L$,
\end{itemize}
such that there exists a classical point $y'_0$ in $(g'_1)^{-1}(s_0)$ lying on the non-contracting locus of $Y'_1$ with respect to $\widehat{g}'_1$. Moreover, $y'_0$ specializes to a point on the smooth locus $\fY'_{1,s}$ of $\fY'_1$.

Now we consider the normalized base change 
\[ 
\begin{tikzcd}
Y' \arrow[r] \arrow[d, "g'"] & Y'_1 \arrow[d, "g'_1"] \\ 
X_L \arrow[r, "f'_L"] & \P^1_L 
\end{tikzcd}
\]
where $f'_L$ is the base change of $f'$. By the same argument as in the proof of Theorem \ref{theorem:modified_spreadout_for_dodging} (but considering the fiber product $\fX' \times_{\widehat{\P}^1_{\mO_L}} \fY'_1$ instead), we know that, up to enlarging $L$, there exists a semistable integral model $\fY'$ of $Y'$ and an integral model $\widehat g': \fY' \ra \fX_{\mO_L}$ of $g'$ such that there exists a point $y_0$ in $(g')^{-1} (x_0)$ that lives on the non-contracting locus of $Y'$ with respect to $\widehat g'$, and that $y_0$ specializes to a point on the smooth locus $\fY'_s$ of $\fY'$.

Recall that $g'_1: Y_1\ra \P^1_L$ is a spread-out of $g_{1,W'}:W'\ra U'_0$. Hence $g' :Y'\ra X_L$ is a spread-out of 
\[g'|_{V'}: V'\ra U\]
where $V'=U_1\times_{U_0} W'$ and $U= U_1\times_{U_0} U'_0$.
Notice that $g'|_{V'}$ is also the base change of $U_1\times_{U_0}W\ra U_1$ along $U'_0\ra U_0$. By construction, $U_1\times_{U_0}W\ra U_1$ factors as \[U_1\times_{U_0}W\ra V\xrightarrow[]{g|_{V_1}} U_1.\]
Therefore, $g'|_{V'}:V'\ra U$ factors as 
\[V'\xrightarrow[]{h}U\xrightarrow[]{g_L|_V} U\]
where $V\subset Y_L$ is the base change of $V_1\subset Y$ along $U'_0\ra U_0$, as desired.
\end{proof}

\subsection{A refined $p$-adic monodromy theorem}

In practice, we apply Theorem \ref{thm:dodging_contracting_locus_general} to obtain the following refined version of the $p$-adic monodromy theorem for curves. This will be a key ingredient in the proof of Theorem \ref{thm:conjecture_for_projective_varieties_intro}.

\begin{theorem}[Refined $p$-adic monodromy theorem for curves]\label{thm:dodging}
    Let $X$ be a smooth projective curve over $K$ with good reduction. Let $\fX$ be a smooth integral model of $X$ over $\mO_K$. Let $\L$ be a de Rham $\Z_p$-local system on $X$ and $x_0 \in X(K)$ be a closed $K$-point (viewed as a $K$-rational classical point on the associated adic space).  Suppose that there exists a smooth projective curve $Y$ over $K$ with semistable reduction, together with a finite cover $g: Y \ra X$ such that $g^* \L $ is a semistable local system on $Y$. Then, up to enlarging $K$ if necessary, there exists a possibly different finite cover $f: Y' \ra X$ such that the following holds: 
    \begin{enumerate}
        \item $Y'$ is a smooth projective curve over $K$ with a semistable integral model $\fY'$ over $\mO_K$; 
        \item the pullback local system $f^* \L$ is a semistable local system on $Y'$; 
        \item there is a morphism of $p$-adic formal schemes \[ \widehat f: \fY' \ra \fX\] inducing $f$ on the generic fibers, such that there exists a classical point $y_0 \in (f)^{-1} (x_0)$ that lives on the non-contracting locus $Y'_{\mathrm{nc}}$ of $Y'$ with respect to $\widehat f$. 
        \item $y_0$ specializes to a point on the smooth locus of the special fiber $\fY'_s$ of $\fY'$.
    \end{enumerate}
\end{theorem}

To prove Theorem \ref{thm:dodging}, we need the following generalization of \cite[Theorem 9.2]{Shimizu_2}, which also deal with non-classical type I points.

\begin{proposition}[Local $p$-adic monodromy theorem around a type I point]\label{prop:Shimizu_generalization}
Let $X$ be a smooth rigid analytic space over $K$ (viewed as an adic space) and let $\L$ be a de Rham $\Z_p$-local system on $X$. Let $x\in X_{\C_p}$ be a type I point. Then there exist a finite extension $L/K$ and an open neighborhood $U\subset X_{L}$ of the image $x'$ of $x$ in $X_L$ such that $\L|_U$ is semistable at all classical points in $U$. 
\end{proposition}

\begin{proof}
With some modifications, the proof of \cite[Theorem 9.2]{Shimizu_2} remains valid.

First, by the same argument as in \cite[Proposition 9.6]{Shimizu_2}, $x$ admits an open neighborhood $U_x$ that is isomorphic to a closed unit polydisc over $\C_p$. The $p$-adic Riemann-Hilbert functor of Liu-Zhu \cite{Liu_Zhu} yields a vector bundle with connection $(D_{\mathrm{dR}}(\L), \nabla)$ over $X$. By base change, we obtain a vector bundle with connection $(D_{\mathrm{dR}}(\L)_{\C_p}, \nabla)$ over $X_{\C_p}$. By \cite[Theorem 9.7]{Shimizu_2}, up to shrinking the polydisc, we may assume that the connection is trivial on $U_x$.

Notice that $U_x$ contains a $K'$-rational point for some finite extension $K'/K$. In particular, the image $U'_x$ of $U_x$ in $X_{K'}$ is isomorphic to the closed unit polydisc over $K'$. By construction, the vector bundle with connection $(D_{\mathrm{dR}}(\L)_{K'}, \nabla)$ is trivial over $U'_x$, hence $\L|_{U'_x}$ is horizontal de Rham. By shrinking $U'_x$, we may assume that $U'_x$ is isomorphic to 
\[\mathbb{T}^n_{K'}\cong \spa\left(K'\langle T_1^{\pm}, \ldots, T_n^{\pm}\rangle, \mO_{K'}\langle T_1^{\pm}, \ldots, T_n^{\pm}\rangle\right).\]
Finally, by the same argument as in \cite[\S 9.4]{Shimizu_2}, there exists a finite extension $L/K'$ such that $\L|_{U'_{x, L}}$ is horizontal de Rham and semistable at all classical points.
\end{proof}

\begin{proof}[Proof of Theorem \ref{thm:dodging}]
We first claim that, for any given point $\xi \in X_{\C_p}$ (of any type), there exist
\begin{itemize}
\item a finite extension $L(\xi)$ over $K$;
\item an open neighborhood $U_{\xi}$ of the image of $\xi$ in $X_{L(\xi)}$;
\item a finite cover $f_{\xi}: Y_{\xi}\ra X_{L(\xi)}$ of smooth projective curves over $L(\xi)$ where $Y_{\xi}$ admits a semistable integral model $\fY_{\xi}$ over $\mO_{L(\xi)}$, together with a morphism of $p$-adic formal schemes
\[
\widehat f_{\xi}: \fY_{\xi} \ra \fX_{\mO_{L(\xi)}}
\] 
that induces $f_{\xi}$ on the generic fibers,
\end{itemize}
such that the following two conditions hold
\begin{itemize}
    \item Let $f_{U_{\xi}}: V_{\xi} \ra U_{\xi}$ denote the base change of $f_{\xi}$ along $U_{\xi} \subset X_{L(\xi)}$. Then $f_{U_{\xi}}^* (\L|_{U_{\xi}})$ is semistable when restricted to every classical point in $V_{\xi}$. 
    \item There exists a classical point $\sq y_{\xi} \in Y_{\xi}$ lying above $x_0$, that lives on the non-contracting locus with respect to $\widehat f_{\xi}$, and moreover specializes to a point on the smooth locus of the special fiber $\fY_{\xi, s}$ of $\fY_{\xi}$. 
\end{itemize}
If $\xi$ is a type I point, the claim follows from Proposition \ref{prop:Shimizu_generalization}. In this case, $f_{\xi}$ is simply the identity map. If $\xi$ is non-type I, we first restrict $f: Y \ra X$ to some neighborhood $U$ of the image of $\xi$ in $X$ over which $f$ is \'etale and such that $x_0\notin U$, and then apply Theorem \ref{thm:dodging_contracting_locus_general}. 

So far, for each $\xi\in X_{\C_p}$, we have produced a quadruple
\[\left(L(\xi),\,\,\, U_{\xi}\subset X_{L(\xi)},\,\,\, f_{\xi}: Y_{\xi}\ra X_{L(\xi)},\,\,\, \widehat f_{\xi}: \fY_{\xi} \ra \fX_{\mO_{L(\xi)}} \right).\]
By shrinking all the other $U_{\xi}$'s if necessary, we may assume that $x_0 \notin U_{\xi}$ for all points $\xi \ne x_0$. Notice that $\{U_{\xi, \C_p}\}_{\xi\in X_{\C_p}}$ cover $X_{\C_p}$. By compactness, $X_{\C_p}$ is covered by a finite sub-collection $\{U_{\xi, \C_p}\}_{\xi\in \mathscr{C}}$. Pick a finite extension $L/K$ containing $L(\xi)$ for all $\xi\in \mathscr{C}$, and hence $\{U_{\xi, L}\}_{\xi\in \mathscr{C}}$ form a cover of $X_L$.

Now let us consider the curve $Y'$ formed by the normalization of the fiber product of all $\{Y_{\xi, L} \}_{\xi \in \mathscr{C}}$ over $X_L$ and denote the canonical map to $X_L$ by $f: Y' \ra X_L$. By construction, after replacing $L$ by a further finite extension if necessary, there exists a semistable integral model $\fY'$ of $Y'$ over $\mO_L$ and a map 
\[ \widehat f: \fY' \ra \fX_{\mO_L}\] inducing $f$ on the generic fibers, such that $f^* \L$ is semistable, and there exists a classical point $y_0 \in Y'$ above $x_0$ which lives on the non-contracting locus. (Here we have used Remark \ref{remark:enlarging_K} once again.) Moreover, by construction, $y_0$ specializes to a point on the smooth locus of $\fY'_s$. This finishes the proof of the theorem.
\end{proof}

\newpage 

\section{\large The crystalline rigidity conjecture} \label{sec:conjecture} 

In this section, we prove the crystalline rigidity conjecture of Shankar (Theorem \ref{thm:conjecture_for_projective_varieties_intro}).  Let us restate the result for convenience of the reader. 

\begin{theorem}\label{thm:rigidity_conjecture_main_text}
Let $X$ be a geometrically connected smooth projective variety over $K$ with good reduction and let $\L$ be an \'etale $\Z_p$-local system on $X$. If $\L|_{x_0}$ is potentially crystalline at one classical point $x_0$ on $X$, then $\L$ is potentially crystalline at all classical points on $X$. 
\end{theorem}

The idea is to first reduce to the case of curves, and then combine the $p$-adic monodromy theorem for curves (as well as its refined version) with the rigidity results for log $F$-isocrystals that we have developed in \S \ref{sec:log_crystals} and \S \ref{sec:log_crystals_on_log_schemes}.

\subsection{Reduction to curves} \label{ss:reducing_to_curves}

The goal of this subsection is to reduce the proof of Theorem \ref{thm:rigidity_conjecture_main_text} to the case of smooth projective curves. As stated in the introduction, the key is to show that we can connect two arbitrary classical points on $X$ by a smooth projective curve with good reduction. To proceed, let us switch to the following setup (only for \S \ref{ss:reducing_to_curves})
\begin{itemize}
    \item Assume that the residue field $k$ of $K$ is algebraically closed.  
    \item Let $\mathcal X$ be a smooth projective scheme over $\spec \mO_K$ of relative dimension $r \ge 2$. 
    \item Let $X$ be the generic fiber of $\mX$ over $K$ and let $\mX_s$ be its special fiber over $k$. 
    \item Let $\iota: \mX \hookrightarrow \P^n_{\mO_K}$ be a closed embedding of $\mX$ into the projective space of relative dimension $n$ with $n > r$, and write $\iota_s: \mX_s \ra \P^n_{k}$ for its base change to $\spec k$. 
    \item Let $\mathcal T \subset \mX$ be a closed subscheme that is smooth over $\mO_K$ and let $\mathcal T_s$ denote the special fiber of $\mathcal T$ over $k$.  
  \end{itemize}

Let us first prove the following two elementary results in algebraic geometry (Lemma \ref{lemma:no_obstruction_in_lifting} and Lemma \ref{lemma:bertini}). The first result is a simple observation on deformations of large degree hypersurfaces in $\P^n_{k}$ (and makes no reference to $\mX$). 
 
\begin{lemma}[Unobstructed lifting of hypersurfaces of sufficiently large degree]\label{lemma:no_obstruction_in_lifting} 
Notice that $\mathcal T_s$ is a smooth closed subvariety in $\P^n_k$ of codimension at least $1$. Then there exists a sufficiently large integer $d$ such that any hypersurface $H_s \subset \P^n_{k}$ of degree at least $d$ containing $\mathcal T_s$ lifts to a hypersurface $\mathcal H \subset \P^n_{\mO_K}$ that contains $\mathcal T$. 
\end{lemma}
 
\begin{proof}
Let $\gamma: \mathcal T \hookrightarrow \P^n_{\mO_K}$ (resp. $\gamma_s: \mathcal T_s \hookrightarrow \P^n_{k}$) denote the closed embedding of $\mathcal T$ (resp. $\mathcal T_s$) into the projective space, and let  $\mI \subset \mO_{\P^n_{\mO_K}}$ (resp. $\mI_s \subset \mO_{\P^n_k}$) denote the corresponding sheaf of ideals. Then, for each integer $d$, we have an exact sequence 
\[
0 \ra \mI (d) \ra \mO_{\P^n_{\mO_K}} (d) \ra 
\mO_{\mathcal T} (d) \ra 0
\]
of sheaves on $\P^n_{\mO_K}$, and a similar exact sequence on $\P^n_{k}$. Now we define 
\begin{align*}
    W(d) & := \ker \Big(H^0 (\P^n_{\mO_K}, \mO_{\P^n_{\mO_K}} (d)) \lra  H^0 (\mO_{\mathcal T}, \mO_{\mathcal T} (d))  
    \Big) \\ 
       W(d)_s & := \ker \Big(H^0 (\P^n_{k}, \mO_{\P^n_{k}} (d)) \:  \lra \:   H^0 (\mO_{\mathcal T_s}, \mO_{\mathcal T_s} (d)). 
    \Big) 
\end{align*}
Here $W(d)$ (resp. $W(d)_s$) is a finitely generated $\mO_K$-module (resp. $k$-vector space) parameterizing hypersurfaces of degree $d$ that pass through $\mathcal T$ (resp. $\mathcal T_s$). Note that for large enough $d$, we have $H^1 (\P^n_{\mO_K}, \mI (d)) = 0$, so the natural map 
\[ 
W(d) \lra W(d)_s
\]
is surjective. This finishes the proof of the lemma. 
\end{proof}
 
The second result is a Bertini-type lemma, and is where we need the assumption that $k$ is algebraically closed. For the rest of \S \ref{ss:reducing_to_curves}, we specialize to the case when $\mathcal T = \{x_1, ..., x_m\} \hookrightarrow \mX$ is a finite union of $m$ disjoint closed  $\mO_K$-points on $\mX$. In this case, the special fiber $\mathcal T_s$ consists of $m$ distinct closed points on $\mX_s$ and is reduced (in particular, $\mathcal T$ is indeed smooth over $\mO_K$). 

\begin{lemma}[Refined Bertini for hypersurfaces] \label{lemma:bertini} 
Let $\mathcal T = \{x_1, ..., x_m\}$ be a finite union of $m$ disjoint closed $\mO_K$-points on $\mX$ as above. Then there exists a hypersurface $H_s \subset \P^n_k$ (of sufficiently large degree) containing $\mathcal T_s$, such that the intersection $H_s \cap \mX_s$ is a smooth variety over $k$ of dimension $r - 1$, where $r = \dim_k \mX_s$ is the dimension of $\mX_s$. 
\end{lemma}
 
\begin{proof}
      The argument is similar to the classical proof of Bertini's theorem. Let $\mathcal T_s = \{\cl x_1, ..., \cl x_m\}$ denote the special fiber of $\mathcal T$, which consists of $m$ distinct closed points in $\P^n_{k}$. Let $\P^\vee_s = \P(W(d)_s)$ denote the projective space parametrizing degree $d$ hypersurfaces in $\P^n_{k}$ passing through $\mathcal T_s$, where $W(d)_s$ is defined as in the proof of the previous lemma. Let us first show that, for a generic hypersurface $H$ in $\P^\vee_s$, the intersection $\mX_s \cap H$ is smooth away from the points in $\mathcal T_s$. For this, we consider the following incidence variety 
      \[
    B \subset \P_s^\vee \times (\mX_s - \mathcal T_s) 
      \] defined as the union 
     \[
        B = \left\{(H, w) \,| \,H \in \P^\vee_s , w \in H, \,H \cap \mX_s \textrm{ is singular at } w \right\}   \cup \left\{(H, w) \,| \,\mX_s \subset H\right\}. \]
It suffices to show that the projection map $B \ra \P^\vee_s$ is not dominant. By counting dimensions, it suffices to show that, for each closed point $w \in \mX_s$ such that $w \notin \mathcal T_s$, the following closed subvariety
  \[
B_w := \left\{H \,|\,w \in H, H \cap \mX_s \textrm{ is singular at } w\right\} \subset \P^\vee_s
  \] 
has codimension at least $r + 1$ in $\P^\vee_{s}$, where $r = \dim \mX_s$. For this, let us pick a degree $d$ hypersurface $H_0$ given by a homogeneous polynomial $f_0$ of degree $d$ such that $H_0 $ does not contain  the point $w$ or any point in $\mathcal T_s$. Let $\fm_w \subset \mO_{\mX_s, w}$ denote the maximal ideal in the Zariski local ring of $\mX_s$ at $w$. Then we have a map 
\[ 
\theta_{w}: W(d)_s \ra \mO_{\mX_s, w} \ra \mO_{\mX_s, w}/\fm_w^2
\]
sending an element $f \in W(d)_s$ to the restriction of $f/f_0$ to $\mO_{\mX_s, w}$. The condition that the hypersurface $H = H(f)$ given by $f$ lies in $B_w$ is equivalent to $\theta_w (f) = 0$. Now $\dim_k \mO_{\mX_s, w}/\fm_w^2 = r + 1$ since $\mX_s$ is smooth, thus it suffices to show that for $d$ large enough the map $\theta_w$ is surjective. In fact it suffices to show that the kernel $\ker (W(d)_s \ra \mO_{\mX_s, w}/\fm_w)$ surjects onto the Zariski cotangent space $\fm_w/\fm_w^2$ at $w$. The latter can be achieved, for example, as follows. Let $f_1, ..., f_r$ be a set of linear homogeneous polynomials  which represents a basis of $\fm_w/\fm_w^2$. For each $i$, we may consider the union of the hyperplane $H_1$ defined by $f_i$ and a degree $d-1$ hypersurface $H_2$ defined by $g_i$, which contains $\mathcal T_s$ but does not contain $w$. Then we have  $\theta_w(f_i \cdot g_i) = \mathrm{unit} \cdot f_i \in \fm_w/\fm_w^2$. This proves that $\theta_w$ is indeed surjective and thus the desired claim that for a general hypersurface $H$ in $\P^\vee_s$, $H \cap \mX_s$ is smooth at all points away from $\mathcal T_s$. 

It remains to show that, for a generic hypersurface $H \subset \P^\vee_s$, $H \cap \mX_s$ is smooth at each $\cl x_i \in \mathcal T_s$. The proof here is in fact similar to the proof above. For each $\cl x_i \in \mathcal T_s$, let us consider the map 
\[ 
\theta_i: W(d)_s \ra \mO_{\mX_s, \cl x_i} \ra \mO_{\mX_s, \cl x_i}/\fm_{\cl x_i}^2.
\] 
Note that this map already factors through 
$W(d)_s \ra \fm_{\cl x_i}/\fm_{\cl x_i}^2$ (contrary to the previous situation). Since these $\cl x_i$'s are distinct closed points, it suffices to show that each $\theta_i$ is not constantly $0$. But this is clear, as the same argument as above in fact shows that $\theta_i$ surjects onto $\fm_{\cl x_i}/\fm_{\cl x_i}^2$. This proves the claim, and thus completing the proof the lemma. 
\end{proof}

Now we restate a more precise version of Proposition \ref{mainprop:reducing_to_curves}. 

\begin{proposition}  \label{prop_in_text:reducing_to_curves}
Let $\mathcal T = \{x_1, ..., x_m\}$ be a finite union of $m$ disjoint closed $\mO_K$-points on $\mX$ as above. Then there exists a closed subvariety $\mC \subset \mX$ that is smooth over $\mO_K$ of relative dimension $1$ containing $\mathcal T$. 
\end{proposition}

\begin{proof}
This immediately follows from Lemma \ref{lemma:no_obstruction_in_lifting} and Lemma \ref{lemma:bertini} by induction on the relative dimension of $\mX/\mO_K$.     
\end{proof}
 
\begin{remark}[Nori] \label{remark:Nori}
If $\mathcal T$ consists only of two points (which is what we actually need in this article), say $x_1$ and $x_2$, then Lemma \ref{lemma:no_obstruction_in_lifting} and Lemma \ref{lemma:bertini} still hold even if they have the same reduction $\cl x_1 = \cl x_2$ in characteristic $p$. This follows from the argument we give above, since the tangent space $T_{\cl x_i} \mathcal T_s$ is at most $1$-dimensional, so the image of $\theta_i$  in $\fm_{\cl x_i}/\fm_{\cl x_i}^2$ has dimension at least $r - 1 \ge 1$.
\end{remark}

We are ready to reduce Theorem \ref{thm:rigidity_conjecture_main_text} to the case of curves. Here we no longer assume that $k$ is algebraically closed.

\begin{corollary} \label{corollary:reducing_to_curves}
   If Theorem \ref{thm:rigidity_conjecture_main_text} holds for all smooth projective curves $X$ over $K$ with good reduction, then it holds in general. 
\end{corollary}

\begin{proof}
    Let $x_1$ be any closed point on $X$. Without loss of generality, we may assume that both $x_0$ and $x_1$ are closed $K$-points. Suppose that $x_0$ and $x_1$ have different reductions in the special fiber, then  Proposition \ref{prop_in_text:reducing_to_curves} tells us that, up to replacing $K$ by a finite unramified extension, there is a smooth curve $C \subset X$ with good reduction connecting $x_0$ and $x_1$. Restricting the $\Z_p$-local system to $C$ and applying the theorem for curves, we know that $\L|_{x_1}$ is potentially crystalline. If $x_0$ and $x_1$ have the same reduction mod $p$, we may find a third point $x_2$ with different reduction in the special fiber, then apply the previous argument twice to deduce that $\L$ is potentially crystalline at both $x_1$ and $x_2$. (Alternatively, one could apply Proposition \ref{prop_in_text:reducing_to_curves} together with Remark \ref{remark:Nori}.)
\end{proof}

\subsection{The proof of the rigidity conjecture for curves} \label{ss:proof_of_Shankar_conjecture_for_curves}

Now we finish the proof of Theorem \ref{thm:rigidity_conjecture_main_text}. Without loss of generality, we may assume that $k$ is algebraically closed. By Corollary \ref{corollary:reducing_to_curves}, it suffices to treat the case when $X$ is a smooth projective curve over $K$ with good reduction, which we assume from now on. We fix a smooth formal integral model $\fX$ of $X$ over $\mO_K$, and let $\fX_s$ denote its special fiber over $k$. For simplicity of exposition, let us introduce the following (non-standard) terminology.  

\begin{definition}\label{defn: crystalline point on special fiber}
\begin{enumerate}
\item For any closed point $w\in \fX_s$, let $X_w$ denote the pre-image of $w$ in $X$ under the specialization map $\mathrm{sp}: X\ra \fX_s$.
\item We say that a closed point $w \in \fX_s$ is \emph{potentially crystalline} if for every classical point $x\in X_w$, we have $\L|_x$ is potentially crystalline.
\end{enumerate}
\end{definition}

Our task is to show that all closed points on $\fX_s$ are potentially crystalline. By assumption, $\L|_{x_0}$ is potentially crystalline at some classical point $x_0\in X$. By enlarging $K$ if necessary, we may assume that $x_0$ is a $K$-point and $\L|_{x_0}$ is actually crystalline. Since $\L|_{x_0}$ is de Rham, $\L$ must be de Rham by the rigidity theorem of Liu-Zhu \cite[Theorem 1.1]{Liu_Zhu}. By Theorem \ref{thm:p_adic_monodromy}, after enlarging $K$ if necessary, there exists a finite cover $f: Y \ra X$ between smooth projective curves over $K$ such that $Y$ has semistable reduction over $\mO_K$ and that $f^* \L$ is a semistable local system on $Y$. Up to further enlarging $K$, there exists a morphism of $p$-adic formal schemes
\[ 
\widehat f: \fY \ra \fX
\]
inducing $f$ on the generic fibers, where $\fY$ is a semistable formal integral model of $Y$ over $\mO_K$. Let $f_s: \fY_s \ra \fX_s$ be the induced map on the special fibers. We may view $f_s$ as a map \[f_s:\fY_s^{\log} \ra \fX_s^{\log}\] between log schemes, where both $\fY_s^{\log}$ and $\fX_s^{\log}$ are equipped with the mod $\varpi$ divisorial log structure. Equivalently, $\fX_s$ is equipped with the log structure coming from the base log point $s^{\log}= \spec (k, 0^{\N})$, as in the beginning of \S \ref{sec:log_crystals}, while $\fY_s^{\log}$ is semistable in the sense of Definition \ref{def:st_log_scheme}. Let $\mE$ denote the log $F$-isocrystal over $\fY_s^{\log}$ associated with the semistable local system $f^* \L$. 

As a map between schemes, the local behavior of the map $f_s$ around a closed point $v$ on $\fY_s$ falls into one of the following three cases: 
\begin{enumerate}
    \item The point $v$ lies on a finite component of $\fY_s$ with respect to $f_s$ in the sense of Definition \ref{definition:contracting_locus} and $v$ lies on the smooth locus $\fY_s^{\mathrm{sm}}$ of $\fY_s$. In this case, we say that $v$ lies on the \emph{good locus} of $\fY_s$ with respect to $f_s$.
    \item The map $f_s$ is finite at $v$, but $v$ is a singular point of $\fY_s$ (which is \'etale locally a nodal point by assumption). In this case, we say that $v$ lies on the \emph{finite-singular locus} of $\fY_s$ with respect to $f_s$.
    \item The point $v$ lies on a contracting component of $\fY_s$ with respect to $f_s$ in the sense of Definition \ref{definition:contracting_locus} (so $f_s$ sends this entire irreducible component to a point). In this case, we say that $v$ lies on the \emph{degenerate locus} of $\fY_s$ with respect to $f_s$.
\end{enumerate}
This is captured by Figure \ref{fig:f_splitting_into_3_cases} below. We caution the reader that, in Case (3), the point $v$ could also be a singular point of $\fY_s$; namely, it does not have to avoid the singular point as illustrated in Figure \ref{fig:f_splitting_into_3_cases}.

\begin{figure}[h]  
\tikzset{every picture/.style={line width=0.75pt}} 
   
\scalebox{0.7}{ 
\begin{tikzpicture}[x=0.75pt,y=0.75pt,yscale=-1,xscale=1]

\draw  [color={rgb, 255:red, 239; green, 19; blue, 55 }  ,draw opacity=1 ][line width=1.5] [line join = round][line cap = round] (148.8,88.26) .. controls (148.02,93.11) and (145.63,96.05) .. (142.21,100.25) .. controls (138.84,104.4) and (135.95,108.95) .. (131.72,112.45) .. controls (98.3,140.09) and (57.98,100.33) .. (25.72,102.93) .. controls (22.72,103.17) and (19.91,107.11) .. (18.57,108.8) .. controls (10.27,119.28) and (8.09,133.33) .. (7.15,145.79) .. controls (6.97,148.11) and (7.65,150.43) .. (8.35,152.67) .. controls (17.74,182.83) and (47.9,168.86) .. (69.92,159.48) .. controls (74.88,157.37) and (91.92,148.83) .. (101.76,149.21) .. controls (121.81,149.97) and (132.66,165.26) .. (145.15,177.38) .. controls (148.23,180.36) and (158.36,192.31) .. (160.78,190.73) ;
\draw  [color={rgb, 255:red, 239; green, 19; blue, 55 }  ,draw opacity=1 ][line width=1.5] [line join = round][line cap = round] (15.51,280.13) .. controls (22.16,279.25) and (29.84,277.46) .. (36.86,276.27) .. controls (44.9,274.91) and (82.15,269.56) .. (83.99,269.44) .. controls (130.43,266.32) and (144.95,271.32) .. (172.51,276.9) ;
\draw  [color={rgb, 255:red, 239; green, 19; blue, 55 }  ,draw opacity=1 ][line width=1.5] [line join = round][line cap = round] (389.4,65.8) .. controls (360.21,92.2) and (332.09,117.69) .. (296.17,135.75) .. controls (286.43,140.64) and (255.48,151.17) .. (254.86,131.67) .. controls (254.27,113.3) and (269.12,107.9) .. (288.03,113.41) .. controls (314.17,121.03) and (326.6,132.76) .. (344.17,151.25) .. controls (374.32,182.97) and (375.41,185.15) .. (395.27,210.75) ;
\draw  [color={rgb, 255:red, 239; green, 19; blue, 55 }  ,draw opacity=1 ][line width=1.5] [line join = round][line cap = round] (255.93,280.13) .. controls (262.58,279.25) and (270.27,277.46) .. (277.29,276.27) .. controls (285.33,274.9) and (322.57,269.56) .. (324.41,269.44) .. controls (370.86,266.32) and (385.38,271.31) .. (412.94,276.9) ;
\draw  [line width=1.5] [line join = round][line cap = round] (542.59,75.38) .. controls (542.59,109.56) and (542.59,143.74) .. (542.59,177.92) ;
\draw  [color={rgb, 255:red, 239; green, 19; blue, 55 }  ,draw opacity=1 ][line width=1.5] [line join = round][line cap = round] (492,280.11) .. controls (498.65,279.23) and (506.33,277.45) .. (513.35,276.25) .. controls (521.39,274.89) and (558.63,269.55) .. (560.47,269.42) .. controls (606.92,266.3) and (621.44,271.3) .. (649,276.89) ;
\draw  [color={rgb, 255:red, 239; green, 19; blue, 55 }  ,draw opacity=1 ][line width=1.5] [line join = round][line cap = round] (494.36,149.07) .. controls (499.52,147) and (504.9,143.84) .. (510.15,141.4) .. controls (516.16,138.6) and (544.73,126.6) .. (546.25,126.14) .. controls (584.56,114.67) and (601.02,116.97) .. (629.36,117.5) ;
\draw    (102.04,181.16) -- (102.04,243.66) ;
\draw [shift={(102.04,245.66)}, rotate = 270] [color={rgb, 255:red, 0; green, 0; blue, 0 }  ][line width=0.75]    (10.93,-3.29) .. controls (6.95,-1.4) and (3.31,-0.3) .. (0,0) .. controls (3.31,0.3) and (6.95,1.4) .. (10.93,3.29)   ;
\draw    (336.69,177.92) -- (336.69,240.41) ;
\draw [shift={(336.69,242.41)}, rotate = 270] [color={rgb, 255:red, 0; green, 0; blue, 0 }  ][line width=0.75]    (10.93,-3.29) .. controls (6.95,-1.4) and (3.31,-0.3) .. (0,0) .. controls (3.31,0.3) and (6.95,1.4) .. (10.93,3.29)   ;
\draw    (581.81,177.1) -- (581.81,239.6) ;
\draw [shift={(581.81,241.6)}, rotate = 270] [color={rgb, 255:red, 0; green, 0; blue, 0 }  ][line width=0.75]    (10.93,-3.29) .. controls (6.95,-1.4) and (3.31,-0.3) .. (0,0) .. controls (3.31,0.3) and (6.95,1.4) .. (10.93,3.29)   ;
\draw  [fill={rgb, 255:red, 0; green, 0; blue, 0 }  ,fill opacity=1 ] (63.45,112.89) .. controls (63.45,111.58) and (64.52,110.52) .. (65.84,110.52) .. controls (67.16,110.52) and (68.23,111.58) .. (68.23,112.89) .. controls (68.23,114.2) and (67.16,115.26) .. (65.84,115.26) .. controls (64.52,115.26) and (63.45,114.2) .. (63.45,112.89) -- cycle ;
\draw  [fill={rgb, 255:red, 0; green, 0; blue, 0 }  ,fill opacity=1 ] (63.45,272.08) .. controls (63.45,270.77) and (64.52,269.71) .. (65.84,269.71) .. controls (67.16,269.71) and (68.23,270.77) .. (68.23,272.08) .. controls (68.23,273.38) and (67.16,274.45) .. (65.84,274.45) .. controls (64.52,274.45) and (63.45,273.38) .. (63.45,272.08) -- cycle ;
\draw  [fill={rgb, 255:red, 0; green, 0; blue, 0 }  ,fill opacity=1 ] (311.8,125.21) .. controls (311.8,123.9) and (312.87,122.84) .. (314.19,122.84) .. controls (315.52,122.84) and (316.59,123.9) .. (316.59,125.21) .. controls (316.59,126.52) and (315.52,127.58) .. (314.19,127.58) .. controls (312.87,127.58) and (311.8,126.52) .. (311.8,125.21) -- cycle ;
\draw  [fill={rgb, 255:red, 0; green, 0; blue, 0 }  ,fill opacity=1 ] (311.8,271.13) .. controls (311.8,269.82) and (312.87,268.76) .. (314.19,268.76) .. controls (315.52,268.76) and (316.59,269.82) .. (316.59,271.13) .. controls (316.59,272.44) and (315.52,273.5) .. (314.19,273.5) .. controls (312.87,273.5) and (311.8,272.44) .. (311.8,271.13) -- cycle ;
\draw  [fill={rgb, 255:red, 0; green, 0; blue, 0 }  ,fill opacity=1 ] (539.93,272.08) .. controls (539.93,270.77) and (541,269.71) .. (542.33,269.71) .. controls (543.65,269.71) and (544.72,270.77) .. (544.72,272.08) .. controls (544.72,273.38) and (543.65,274.45) .. (542.33,274.45) .. controls (541,274.45) and (539.93,273.38) .. (539.93,272.08) -- cycle ;
\draw  [fill={rgb, 255:red, 0; green, 0; blue, 0 }  ,fill opacity=1 ] (539.93,114.79) .. controls (539.93,113.48) and (541,112.42) .. (542.33,112.42) .. controls (543.65,112.42) and (544.72,113.48) .. (544.72,114.79) .. controls (544.72,116.1) and (543.65,117.16) .. (542.33,117.16) .. controls (541,117.16) and (539.93,116.1) .. (539.93,114.79) -- cycle ;

\draw (104.11,198.96) node [anchor=north west][inner sep=0.75pt]    {$f_{s}$};
\draw (339.07,198.96) node [anchor=north west][inner sep=0.75pt]    {$f_{s}$};
\draw (583.47,198.96) node [anchor=north west][inner sep=0.75pt]    {$f_{s}$};
\draw (64.95,89.1) node [anchor=north west][inner sep=0.75pt]    {$v$};
\draw (304.69,94.79) node [anchor=north west][inner sep=0.75pt]    {$v$};
\draw (543.35,95.73) node [anchor=north west][inner sep=0.75pt]    {$v$};
\draw (63.1,246.39) node [anchor=north west][inner sep=0.75pt]    {$w$};
\draw (305.72,245.44) node [anchor=north west][inner sep=0.75pt]    {$w$};
\draw (542.46,245.44) node [anchor=north west][inner sep=0.75pt]    {$w$};
\draw (31.32,175.27) node [anchor=north west][inner sep=0.75pt]    {$\fY_{s}$};
\draw (264.36,144.95) node [anchor=north west][inner sep=0.75pt]    {$\fY_{s}$};
\draw (502.06,148.74) node [anchor=north west][inner sep=0.75pt]    {$\fY_{s}$};
\draw (35.19,283.29) node [anchor=north west][inner sep=0.75pt]    {$\fX_{s}$};
\draw (277.8,284.24) node [anchor=north west][inner sep=0.75pt]    {$\fX_{s}$};
\draw (511.67,284.24) node [anchor=north west][inner sep=0.75pt]    {$\fX_{s}$};
\draw (73,313) node [anchor=north west][inner sep=0.75pt]   [align=left] {Case (1)};
\draw (313,314) node [anchor=north west][inner sep=0.75pt]   [align=left] {Case (2)};
\draw (554,313) node [anchor=north west][inner sep=0.75pt]   [align=left] {Case (3)};
\end{tikzpicture}
} 
    \caption{Local behavior of $f_s$ at a closed point $v \in \fY_s$}
    \label{fig:f_splitting_into_3_cases}
\end{figure}
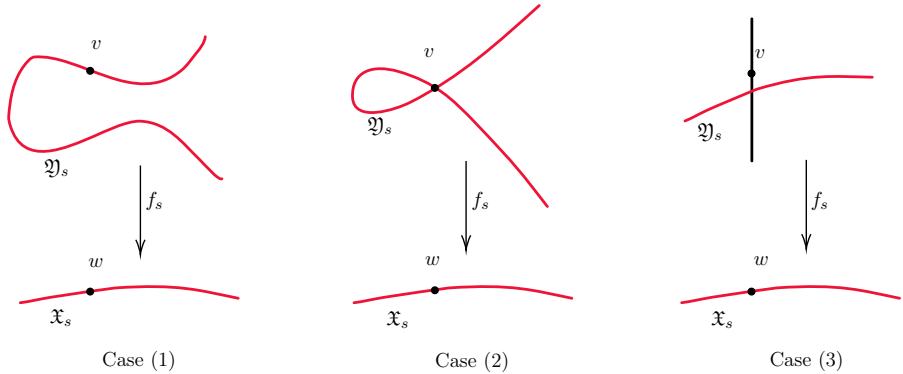

Back to the proof of Theorem \ref{thm:rigidity_conjecture_main_text}. Recall that $\L|_{x_0}$ is crystalline at some $K$-rational classical point $x_0\in X$. Applying Theorem \ref{thm:dodging}, we may assume that there exists a classical point $y_0\in f^{-1}(x_0)$ such that the image $v_0$ of $y_0$ in $\fY_s$ (under the specialization map $\mathrm{sp}: Y\ra \fY_s$) lies in the good locus of $\fY_s$ with respect to $f_s$. In particular, $\mE$ has trivial monodromy at $v_0$. Let $Z$ denote the irreducible component of $\fY_s$ containing $v_0$ and let $Z^{\mathrm{sm}}:=Z\cap \fY_s^{\mathrm{sm}}$. By Theorem \ref{theorem:log_crystal_1} and Theorem \ref{theorem:main_intro}, we conclude that $\mE$ has trivial monodromy at all closed points in $Z^{\mathrm{sm}}$ and $f^*\L$ is crystalline at all classical points in $\mathrm{sp}^{-1}(Z^{\mathrm{sm}})$.

We claim that $\mE$ actually has trivial monodromy at all closed points on the good locus of $\fY_s$. To prove this, we need some preparations.

\begin{definition}\label{defn: linked}
\begin{enumerate}
\item For a closed point $v\in \fY_s$, let $Y_v$ denote the pre-image of $v$ in $Y$ under the specialization map $\mathrm{sp}: Y\ra \fY_s$.
\item Let $v_1, v_2$ be two closed points lying on the good locus of $\fY_s$. We say that $v_1$ and $v_2$ are \emph{linked} if $f_s(v_1)=f_s(v_2)$ and $f(Y_{v_1})\cap f(Y_{v_2})\neq \emptyset$ in $X$.
\item Let $Z_1$ and $Z_2$ be two irreducible components of $\fY_s$. Assume that both $Z_1$ and $Z_2$ are finite components with respect to $f_s$ in the sense of Definition \ref{definition:contracting_locus}. We say that $Z_1$ and $Z_2$ are \emph{linked} if there exist closed points $v_1\in Z_1^{\mathrm{sm}}$ and $v_2\in Z_2^{\mathrm{sm}}$ such that $v_1$ and $v_2$ are linked.
\end{enumerate}
\end{definition}

\begin{lemma}\label{lemma: linked}
Let $Z_1$ and $Z_2$ be two finite components of $\fY_s$ that are linked with each other. Suppose $\mE$ has trivial monodromy at every closed point on $Z_1^{\mathrm{sm}}$, then $\mE$ also has trivial monodromy at every closed point on $Z_2^{\mathrm{sm}}$.
\end{lemma}

\begin{proof}
By definition, there exist $v_1\in Z_1^{\mathrm{sm}}$ and $v_2\in Z_2^{\mathrm{sm}}$ such that $v_1$ and $v_2$ are linked. By \cite[Theorem 0.1]{Mann}, the finite map $f$ is open, thus $f(Y_{v_1})\cap f(Y_{v_2})$ is a non-empty open subset of $X$. Pick a classical point $w\in f(Y_{v_1})\cap f(Y_{v_2})$. Since $\mE$ has trivial monodromy everywhere on $Z_1^{\mathrm{sm}}$, we know that $f^*\L$ is crystalline on $\mathrm{sp}^{-1}(Z_1^{\mathrm{sm}})$. Therefore, $\L$ is potentially crystalline at all classical points in $f(Y_{v_1})$. In particular, $\L$ is potentially crystalline at $w$. 

Let $y_2\in Y_{v_2}$ be a classical point such that $f(y_2)=w$. Then $f^*\L|_{y_2}$ is potentially crystalline. Since $f^*\L|_{y_2}$ is also semistable, it must be crystalline. Applying Theorem \ref{theorem:log_crystal_1} and Theorem \ref{theorem:main_intro} once again, we conclude that $\mE$ has trivial monodromy at every closed point on $Z_2^{\mathrm{sm}}$ and $f^*\L$ is crystalline on $\mathrm{sp}^{-1}(Z_2^{\mathrm{sm}})$.
\end{proof}

Now, we are ready to prove the claim. Thanks to Lemma \ref{lemma: linked}, it suffices to prove the following: for any two finite components $Z$ and $Z'$ of $\fY_s$, there exists a sequence of finite components $Z_0, Z_1, \ldots, Z_m$ such that $Z=Z_0$, $Z'=Z_m$ and such that $Z_i$ and $Z_{i+1}$ are linked, for all $i=0,1,\ldots, m-1$. We say that a closed point $w\in \fX_s$ is \emph{good} if all points in the pre-image $f_s^{-1}(w)$ live on the good locus of $\fY_s$ with respect to $f_s$; otherwise, we say $w$ is \emph{bad}. Notice that all but finitely many closed points on $\fX_s$ are good. Also notice that $f_s(Z^{\mathrm{sm}})$ is the entire $\fX_s$ minus finitely many points (similarly for $f_s({Z'}^{\mathrm{sm}})$). In particular, there exists a good point $w$ lying in the intersection $f_s(Z^{\mathrm{sm}})\cap f_s({Z'}^{\mathrm{sm}})$; i.e., there exist $v\in Z^{\mathrm{sm}}$ and $v'\in {Z'}^{\mathrm{sm}}$ such that $f_s(v)=f_s(v')=w$. We have
\[X_w=\bigcup_{v\in f_s^{-1}(w)} f(Y_v).\]
Using the connectedness of $X_w$ and the openness of $f$, there exists a sequence of closed points $v_0, v_1, \ldots, v_m$ in $f_s^{-1}(w)$ such that $v_0=v$, $v_m=v'$ and such that $v_i$ and $v_{i+1}$ are linked, for all $i=0,1,\ldots, m-1$. Then we take $Z_i$ to be the irreducible component containing $v_i$. This finishes the proof of the claim.

\begin{remark}
Let $\fY_s^{\mathrm{nc}}$ denote the union of the good locus and the finite-singular locus of $\fY_s$ with respect to $f_s$. The claim together with Theorem \ref{theorem:log_crystal_2_global} imply that the restriction of $\mE$ on $\fY_s^{\mathrm{nc}}$ descends to a (non-log) $F$-isocrystal on $\fY_s^{\mathrm{nc}}$, and hence $f^*\L$ is crystalline on $\mathrm{sp}^{-1}(\fY_s^{\mathrm{nc}})$.\footnote{Notice that $\mathrm{sp}^{-1}(\fY_s^{\mathrm{nc}})$ is precisely the non-contracting locus $Y_{\mathrm{nc}}$ of $Y$ with respect to $\widehat{f}$ in the sense of Definition \ref{definition:contracting_locus}.} But we do not need this fact in the argument below.
\end{remark}

The claim above actually implies that every good point in $\fX_s$ is potentially crystalline in the sense of Definition \ref{defn: crystalline point on special fiber}. There are only finitely many bad points left to be dealt with. Let $w''$ be one of these bad points. We have to show that for every classical point $x''\in X_{w''}$, $\L|_{x''}$ is potentially crystalline. By enlarging $K$, we may assume that $x''$ is $K$-rational. By Theorem \ref{thm:dodging}, up to further enlarging $K$ if necessary, we can find a new finite cover $f':Y'\ra X$ between smooth projective curves over $K$ such that
\begin{itemize}
\item $Y'$ admits a semistable formal integral model $\fY'$ over $\mO_K$;
\item there is a morphism of $p$-adic formal schemes
\[\widehat{f}': \fY'\ra \fX\]
inducing $f'$ on the generic fibers, such that ${f'}^*\L$ is a semistable local system;
\item there exists $y''\in (f')^{-1}(x'')$ such that the image $v''$ of $y''$ in $\fY'_s$ (under the specialization map $\mathrm{sp}: Y'\ra \fY'_s$) lies on the good locus of $\fY'_s$ with respect to $f'_s$, where $f'_s: \fY'_s\ra \fX_s$ is the mod $\varpi$ reduction of $\widehat f'$.
\end{itemize}
Let $\mE'$ be the log $F$-isocrystal on ${\fY'}_s^{\log}$ associated with the semistable local system ${f'}^*\L$. Let $Z''$ denote the irreducible component of $\fY'_s$ containing $v''$. Since $Z''$ is a finite component with respect to $f'_s$, we know that $f'_s({Z''}^{\mathrm{sm}})$ is $\fX_s$ minus finitely many points. Hence, there exist a closed point $w''\in \fX_s$ that we already know is potentially crystalline such that ${Z''}^{\mathrm{sm}}$ contains one of the pre-image $v^{\flat}\in (f'_s)^{-1}(w'')$. In particular, $\mE'$ has trivial monodromy at $v^{\flat}$. By applying Theorem \ref{theorem:log_crystal_1} one more time, we conclude that ${f'}^*\L$ is crystalline on $\mathrm{sp}^{-1}({Z''}^{\mathrm{sm}})$; in particular, $\L$ is potentially crystalline at $x''$. This finishes the proof of Theorem \ref{thm:rigidity_conjecture_main_text}.

\begin{remark} \label{remark:counter_example_to_optimistic_guess}
The curious reader may wonder whether in Theorem \ref{thm:conjecture_for_projective_varieties_intro} the local system $\L$ always becomes crystalline after replacing the base field by a finite extension. In this remark, we produce a counterexample in the spirit of \cite{Lawrence_Li} to show that this is false. Let $E$ be an elliptic curve with good ordinary reduction over $\Q_p$ and consider the isogeny $[p]: E \ra E$ given by multiplication by $p$. Let $\L = [p]_*  \ul{\Z_p}$ denote the pushforward of the constant local system. Then $\L$ is potentially crystalline at every closed point on $E$. We claim that for each totally ramified finite extension $L/\Q_p$ we can find a closed point $x \in E (L)$ such that the restriction of $\L$ to $x$ is not crystalline. As in \cite{Lawrence_Li}, given a closed point $x$ on $E(L)$, $\L|_x$ gives rise to a de Rham presentation $\rho_x$ of $G_L$ with finite image, which is crystalline precisely when it is unramified. Now consider the $p$-divisible group $E[p^\infty]$, which is an extension of $\Q_p/\Z_p$ by $\mu_{p^\infty}$. Near the origin, the multiplication by $p$ map is given by the $p$-power map on $\mu_{p^\infty}$. In particular, for $L = \Q_p (\varpi_L)$ where $\varpi_L$ is a uniformizer, there exists a point $x$ in $E$ over which $[p]^{-1} (x)$ becomes unramified over $L(\varpi_L^{1/p}).$
\end{remark}

\newpage 

\section{\large Extension of local systems across normal crossing divisors}\label{section:punctured disc}

In the remainder of the article, we prove Theorem \ref{mainthm:extend_across_ncd} which asserts that a pointwise-crystalline $p$-adic local systems on the complement of a normal crossing divisor on a smooth adic space always extends to the entire space. As mentioned in the introduction, we shall first prove the theorem in the case of a punctured disc, and then deduce the general case from it.

\begin{theorem}[Theorem \ref{theorem:main_intro_punctured_disc}]\label{thm:extendability}
Let $\D=\D_K$ be the closed unit disc over $K$ and let $\D^\times=\D-\{0\}$ be the punctured closed unit disc. Let $\L$ be an \'etale $\Z_p$-local system on $\mathbb D^\times$. Assume that $\mathbb{L}|_x$ is crystalline at all classical points $x$ in $\mathbb D^\times$. Then $\mathbb{L}$ extends uniquely to an \'etale $\Z_p$-local system on $\mathbb{D}$ and is necessarily a crystalline local system.
\end{theorem}

To prove Theorem \ref{thm:extendability}, we consider the natural inclusion
$j: \mathbb D^\times \hookrightarrow\mathbb{D}$ where $\mathbb{D}$ is equipped with the log structure associated with the puncture. Let $j_*\mathbb{L}$ be the Kummer \'etale $\Z_p$-local system on $\mathbb{D}$ extending the \'etale local system $\mathbb{L}$ (cf. \cite[Corollary 6.3.4]{DLLZ1}). Our goal is to show that $j_*\mathbb{L}$ actually comes from an \'etale local system on $\mathbb{D}$. To this end, we will apply the logarithmic $p$-adic Riemann-Hilbert functor $\D_{\mathrm{dR}, \log} $ (constructed in \cite{DLLZ2},  which we recall in \S \ref{subsection:log RH}) to $\L$. In particular, we obtain a vector bundle with integrable log connection $(\D_{\mathrm{dR}, \log} (\mathbb{L}), \nabla)$ on $\mathbb{D}$. We will show that, if $\mathbb{L}$ satisfies the condition of the theorem, the restriction of the connection on $\mathbb{D}^{\times}$ must admit a full set of solutions. This step uses the full strength of the rigidity result for log $F$-isocrystals that we have established in \S \ref{sec:log_crystals} and \S \ref{sec:log_crystals_on_log_schemes}. This in turn implies that $(\D_{\mathrm{dR}, \log} (\mathbb{L}), \nabla)$ has trivial residue around the puncture. Finally, we will show that this is enough to ensure that $j_* \L$ must have trivial geometric monodromy around the puncture and thus extends to $\D$. It is then not difficult to deduce that the extended local system is indeed crystalline. 

\subsection*{Notation} In this section, let $\D^{n, \log }$ denote the $n$-dimensional ``log polydisc'',  given by the $n$-dimensional closed unit polydisc \[\mathbb{D}^n=\spa(K\langle T_1, \ldots, T_n\rangle, \mathcal{O}_K\langle T_1, \cdots, T_n\rangle)\] equipped with the log structure associated with the normal crossing divisor $\{T_1\cdots T_n=0\}$. 


\subsection{The logarithmic $p$-adic Riemann-Hilbert functor}\label{subsection:log RH}
In this section, we recall some results from \cite{DLLZ2} on the logarithmic $p$-adic Riemann-Hilbert functor.  Throughout \S \ref{subsection:log RH}, we work with the following setup: 
\begin{itemize}
    \item 
Let $K$ be a $p$-adic field as before and let $Y$ be a smooth rigid analytic space over $K$, viewed as an adic space. 
   \item Let $D\subset Y$ be a (reduced) normal crossing divisor (cf. \cite[Example 2.3.17]{DLLZ1}) and write $U:=Y-D$. In other words,  analytic locally on $Y$, up to replacing $K$ by a finite extension, $Y$ and $D$ are of the form $S\times \mathbb{D}^r$ and $S\times \{T_1\cdots T_r=0\}$ where $S$ is a connected smooth rigid analytic space over $K$ and the closed immersion $D\hookrightarrow Y$ is given by the pullback of the natural closed immersion \[ \{T_1\cdots T_r=0\}\hookrightarrow \mathbb{D}^r. \]
   \item Equip $Y$ with the log structure given by the natural inclusion $\alpha: \mathcal{M}_Y\rightarrow \mathcal{O}_{Y_{\ett}}$ where 
\[\mathcal{M}_Y=\big\{ f\in \mathcal{O}_{Y_{\ett}}\,\big|\, f \textrm{ is invertible on } U\big\}.\]
Then $(Y, \mathcal{M}_Y, \alpha)$ is a noetherian fs log adic space which is log smooth over $K$.  
\end{itemize}

Note that, \'etale locally, $Y$ admits a \emph{toric chart} $Y\rightarrow \mathbb{D}^{n, \log}$ (in other words, $Y$ is affinoid and $Y \rightarrow \mathbb{D}^{n, \log}$ is strictly \'etale), such that $D\hookrightarrow Y$ is the pullback of the inclusion 
\[\{T_1\cdots T_r=0\}\hookrightarrow \mathbb{D}^{n, \log} \]  for some $1\leq r\leq n$.  Let $\mathcal{V}$ be a vector bundle on the analytic site $Y_{\mathrm{an}}$ equipped with an integrable logarithmic connection \[\nabla:\mathcal{V}\rightarrow \mathcal{V}\otimes_{\mO_Y}\Omega_Y^{1, \log}.\] Let $Z\subset D$ be an irreducible component. There is a well-defined \emph{residue} map along the component $Z$, which is an  $\mathcal{O}_Z$-linear endomorphism \[\mathrm{Res}_Z(\nabla): \mathcal{V}|_Z\rightarrow \mathcal{V}|_Z.\]
\'Etale locally, suppose that $Y$ is equipped with a toric chart $Y\rightarrow \mathbb{D}^{n, \log}$ as above and that $Z$ is given by $\{T_i=0\}$ for some $1\leq i\leq r$, then $\mathrm{Res}_{Z}(\nabla)$ is given by $\nabla(T_i \frac{\partial}{\partial T_i})$. For more details, see for example \cite[\S 11.1]{ABC}.

Now we recall the following definition from {\cite[Definition 6.3.7]{DLLZ1}}. Let $Y_{\ket}$ denote the Kummer \'etale site on the fs log adic space $Y$ (see \cite[\S 4.1]{DLLZ1}). If $\mathbb{L}$ is a $\mathbb{Z}_p$-local system on $Y_{\ket}$ as in  Definition 6.3.1 in \emph{loc.cit.}, we shall use $\widehat{\mathbb{L}}$ to denote the associated $\widehat{\mathbb{Z}}_p$-local system on the pro-Kummer \'etale site $Y_{\proket}$ (see Definition 6.3.2 of \emph{loc.cit.}). 

\begin{definition}\label{defn:geometric monodromy}
We say that a $\mathbb{Z}_p$-local system $\mathbb{L}$ on $Y_{\ket}$ has \emph{trivial geometric monodromy} (resp. \emph{unipotent geometric monodromy}) along $D$ if the Kummer \'etale fundamental group $\pi_1^{\ket}(Y(s), \tilde{s})$ acts trivially (resp. unipotently) on the stalk $\mathbb{L}_{\tilde{s}}$ for every log geometric point $\tilde{s}$ of $Y$ lying above every geometric point $s$ of $D$, where $Y(s)$ denote the strict localization of $Y$ at $s$ equipped with the pullback log structure from $Y$.
\end{definition}

\begin{remark}
Let us note that $\mathbb{L}$ has trivial geometric monodromy if and only if the Kummer \'etale local system $\mathbb{L}$ comes from an \'etale local system. In other words, if $\mathbb{L}\simeq \varepsilon_{\ett}^*\mathbb{L}_0$ for some $\mathbb{Z}_p$-local system $\mathbb{L}_0$ on $Y_{\ett}$ where $\varepsilon_{\ett}:Y_{\ket}\rightarrow Y_{\ett}$ is the natural projection of sites.
\end{remark}

\begin{remark}\label{remark:smooth locus}
By\cite[Lemma 6.3.11]{DLLZ1} and its proof, to check whether $\mathbb{L}$ has trivial geometric monodromy (resp. unipotent geometric monodromy) along $D$, it suffices to verify the condition for geometric points $s$ lying above the smooth locus of $D$. 
\end{remark}

Now we recall the logarithmic $p$-adic Riemann-Hilbert functor of \cite{DLLZ2}. Consider the natural projection $\mu: Y_{\proket}\rightarrow Y_{\mathrm{an}}$ from the pro-Kummer \'etale site onto the analytic site, and the 
\emph{geometric log de Rham period sheaves} $\mathcal{O}\mathbb{B}^+_{\mathrm{dR}, \log, Y}$ and $\mathcal{O}\mathbb{B}_{\mathrm{dR}, \log, Y}$ on $Y_{\proket}$ (see \S 2.2 of \emph{loc. cit.}).\footnote{When the context is clear, we often omit the subscript `$Y$' and denote the period sheaves by $\mathcal{O}\mathbb{B}^+_{\mathrm{dR}, \log}$ and $\mathcal{O}\mathbb{B}_{\mathrm{dR}, \log}$.} For any $\mathbb{Z}_p$-local system $\mathbb{L}$ on $Y_{\ket}$, define
\[\D_{\mathrm{dR}, \log}(\mathbb{L}):=\mu_*(\widehat{\mathbb{L}}\otimes_{\widehat{\mathbb{Z}}_p} \mathcal{O}\mathbb{B}_{\mathrm{dR}, \log}).\]
There is an integrable log connection $\nabla_{\mathbb{L}}$ as well as a decreasing filtration $\mathrm{Fil}^{\bullet}\D_{\mathrm{dR}, \log}(\mathbb{L})$ on $\D_{\mathrm{dR}, \log}(\mathbb{L})$ inherited from the ones on $\mathcal{O}\mathbb{B}_{\mathrm{dR}, \log}$. The assignment \[\mathbb{L}\mapsto (\D_{\mathrm{dR}, \log}(\mathbb{L}), \nabla_{\mathbb{L}}, \mathrm{Fil}^{\bullet}\D_{\mathrm{dR}, \log}(\mathbb{L}))\] defines a functor $\D_{\mathrm{dR}, \log}$ from the category of $\mathbb{Z}_p$-local systems on $Y_{\ket}$ to the category of vector bundles on $Y_{\mathrm{an}}$ equipped with an integrable log connection and a decreasing filtration satisfying Griffiths transversality. The main results of \cite{DLLZ2} (Theorem 3.2.7, Theorem 3.2.12 of \textit{loc.cit.}) assert that, if the restriction of $\mathbb{L}$ to $U_{\ett}$ is de Rham, then $\D_{\mathrm{dR}, \log} (\L)$ is a vector bundle of rank $\mathrm{rk}_{\mathbb{Z}_p}\mathbb{L}$. Moreover, the functor $\D_{\mathrm{dR}, \log} $ restricts to a tensor functor from the category of $\mathbb{Z}_p$-local systems on $Y_{\ket}$ whose restriction to $U_{\ett}$ is de Rham and have unipotent geometric monodromy along $D$ to the category of filtered vector bundles on $Y_{\mathrm{an}}$ equipped with an integrable log connection with nilpotent residues along $D$. Our first task is to prove a slightly refined version of the last statement, in other words, a strengthening of \cite[Theorem 3.2.12(ii)]{DLLZ2}.

\begin{theorem}[Oswal--Shankar--Zhu]\label{prop:residue vs monodromy}
Suppose the restriction of $\mathbb{L}$ on $U_{\ett}$ is de Rham. Then 
\begin{enumerate}
    \item  $\mathbb{L}$ has trivial geometric monodromy  along $D$ (in other words, $\L$ extends to an \'etale local system on $Y$)  if and only if $\D_{\mathrm{dR}, \log} (\mathbb{L})$ has zero residue  along $D$, if and only if $(\D_{\mathrm{dR}, \log}(\mathbb{L}), \nabla_{\mathbb{L}})$ extends to a vector bundle with connection on $Y$. 
    \item Similarly,  $\mathbb{L}$ has unipotent geometric monodromy  along $D$ if and only if $\D_{\mathrm{dR}, \log} (\mathbb{L})$  nilpotent residue along $D$.  
\end{enumerate}
\end{theorem}

\begin{proof}
    The first claim is essentially \cite[Theorem 5.7]{OSZ}. The second claim follows from a similar argument as the first. 
\end{proof}

\subsection{From log $F$-isocrystals to vector bundles with connections}

Given a $p$-adic formal scheme with semistable reduction, one can attach a vector bundle with connection on the generic fiber from a log $F$-isocrystal on the special fiber. The purpose of this subsection is to recall this construction in the case of the standard ``thick annulus'' with semistable reduction. This will be used in our proof of Theorem \ref{thm:extendability}.

\begin{construction} \label{construction:vector_bundle_with_connections}
As in the statement of Theorem \ref{thm:main_intro_for_log_schemes}, let  \[
A_1 = \{ 1/p \le |z| \le 1 \}\]
be the standard ``thick annulus'' with semistable reduction over $K$; in other words, $A_1$ is the affinoid adic space obtained as the generic fiber of the $p$-adic formal scheme \[\mathcal D=\spf \mO_K \gr{x, y}/(xy - \varpi).\] Let $\mathcal D_s$ denote the special fiber 
\[\mathcal D_s=\spec k[x, y]/xy\] 
and let $\mathcal D_s^{\log}$ denote the log scheme associated with the pre-log algebra $(k[x, y]/xy, x^{\N} \oplus y^{\N})$ as in \S \ref{sec:log_crystals_on_log_schemes}. Let $\mE$ be a log $F$-isocrystal on $\mathcal D_s^{\log}$. We attach a vector bundle with connection $(\mathbb E, \nabla_{\mathbb{E}})$ on $A_1$ as follows. For simplicity, let us assume that $K$ is unramified for the moment, in other words, $K = K_0 = W(k) [1/p]$. The general case will follow from a similar construction together with ``Dwork's trick'', see Remark \ref{remark:general_K_vector_bundle}. 

Consider $B^{(0)} = W \gr{x, y}/(xy - p)$, 
equipped with a (pre-)log structure $x^{\N} \oplus y^{\N}$. Let $B^{(1), \log}$ denote the usual $p$-completed self-coproduct of $B^{(0)}$ in (the opposite category of) the $p$-completed affine log crystalline site of $(k[x, y]/xy, x^{\N} \oplus y^{\N})$. Then we have 
\begin{align*}
B^{(1), \log} & = W \gr{x_1, y_1} \gr{\delta_x, \delta_y}^{\mathrm{PD}}/ \big(x_1 y_1 - p, (\delta_x + 1)(\delta_y +1) -1 \big) \\
& =  W \gr{x_1, y_1} \gr{\delta_x}^{\mathrm{PD}}/ (x_1 y_1 - p)
\end{align*}
where $\delta_x = x_2/x_1 - 1$ and $\delta_y = y_2/y_1 -1$. The data of the log $F$-isocrystal $\mE$ gives rise to a $B^{(0)}$-module $M$ such that $M[1/p]$ is a finite free $B^{(0)}[1/p]$-module, 
together with a descent isomorphism 
\begin{equation} \label{eq:descent_iso_for_xy-p}
\delta_B: M  \otimes_{B^{(0)}, \iota_1} B^{(1), \log} \isom B^{(1), \log} \otimes_{B^{(0)}, \iota_2} M 
\end{equation}
satisfying the cocycle condition. Next, we consider the self-coproduct $B^{(0)}\otimes_{(W, p^{\N})} B^{(0)}$ of pre-log rings $(W, p^{\N}) \ra (B^{(0)}, x^\N \oplus y^\N)$ where the map on monoids is given by the diagonal map $a \mapsto (a, a)$, and denote by $\sq B^{(1), \log}$ the $p$-completed exactification of this self-coproduct with respect to the projection map $B^{(0)}\otimes_{(W, p^{\N})} B^{(0)} \ra B^{(0)}$. Note that  
\[ 
\sq B^{(1), \log} = W \gr{x_1, y_1, \delta_x, \delta_y} / \big(x_1 y_1 - p, (\delta_x + 1)(\delta_y +1) -1 \big), 
\] 
and that the exactified pre-log structure is given by $\N \oplus \N \oplus \Z \ra \sq B^{(1), \log}$ sending $(a, b, c) \mapsto x_1^a y_1^b (\delta_x + 1)^c$.  Let $\sq I$ denote the kernel of the projection $\sq B^{(1), \log} \ra B^{(0)}$ induced by the product map. Then $\sq I$ is generated by $\delta_x, \delta_y$ and we have 
\[
\sq B^{(1), \log}/(\sq I)^2 = W\gr{x_1, y_1, \delta_x}/(x_1 y_1 - p, \delta_x^2) = B^{(0)}
[\delta_x]/(\delta_x)^2\]
since the condition $(\delta_x + 1)(\delta_y + 1) = 1$ becomes $\delta_x = - \delta_y$ in this quotient. Now, we observe that $\sq B^{(1), \log}/(\sq I)^2$, equipped with the pre-log structure induced from $\N \oplus \N \oplus \Z$, naturally lives in the $p$-completed affine log-crystalline site of $(k[x, y]/xy, x^\N \oplus y^\N)$. In fact,  we have a natural map \[ 
\pi: B^{(1), \log} \ra \sq B^{(1), \log}/(\sq I)^2
\]
sending $\delta_x \mapsto \delta_x$ and the PD powers $\delta_x^{[n]} \mapsto 0$ for all $n \ge 2$. Let $\cl \delta_B$ denote the base change of the isomorphism $\delta_B$ along $\pi$, and define the map 
\[
\nabla:  M \ra M \otimes_{B^{(0)}, \iota_1} \sq B^{(1), \log}/(\sq I)^2
\]
by 
\[
\nabla (m) = (\cl \delta_B)^{-1} (\iota_2^* m) - \iota_1^* m  = (\cl \delta_B)^{-1} (1 \otimes m) - m \otimes 1. 
\] 
Note that, the image of $\nabla$ lives in 
\[M \otimes_{B^{(0)}, \iota_1} \sq I \sq B^{(1), \log}/(\sq I)^2 \cong M \otimes_{B^{(0)}} \Omega^{1, \log}_{(B^{(0)}, \N^2)/(W, p^\N)}
\]
using the identification $\sq I \sq B^{(1), \log}/(\sq I)^2 \cong \Omega^{1, \log}_{(B^{(0)}, \N^2)/(W, p^\N)}$ sending $\delta_x$ to $\mathrm{d}\!\log x$. Inverting $p$, we obtain the desired vector bundle $\mathbb E$ (corresponding to the finite free module $M[1/p]$ over $B^{(0)}[1/p]$) with a logarithmic connection $\nabla_{\mathbb E}$ over the annulus $A_1$. 
\end{construction}

\begin{remark}[Dwork's trick]\label{remark:general_K_vector_bundle} 
For the general case, it suffices to observe that there is an equivalence between the category of log $F$-isocrystals on $(k[x,y]/xy, x^{\N} \oplus y^{\N})$ and log $F$-isocrystals on  $(\mO_K[x,y]/(xy - \varpi, p), x^{\N} \oplus y^{\N})$ induced by the $e^{\mathrm{th}}$-iterates of the Frobenius map, where $e$ is the ramification index of $K/\Q_p$.  
\end{remark} 




Let us record the following lemma for later use. 

\begin{lemma} \label{lemma:isocrystal_from_P1}
Let us retain notations from Construction  \ref{construction:vector_bundle_with_connections}.  Suppose that $\mE$ descends to an $F$-isocrystal $\mE_0$ on $\mathcal D_s$ (or equivalently, $\mE$ satisfies any of the equivalent conditions in \ref{theorem:log_crystal_2}), and further suppose that the restriction of $\mE$ along the closed immersion 
\[ 
\iota_y: \A^1_{k}  = \spec k[y] \lra \mathcal D_s = \spec k[x, y]/xy
\] 
extends to an $F$-isocrystal on $\P^1_k$,  then the associated vector bundle $(\mathbb E, \nabla_{\mathbb E})$ from the construction above has full solutions when restricted to the ``thin annulus'' 
\[
\mathbb B_1 = \{|y| = 1\} =  \{|x| = 1/p\} 
\]
inside the ``thick annulus'' $A_1$. 
\end{lemma}

\begin{proof} 
The condition that the $F$-isocrystal $\mE_y := (\iota_y)^* \mE$ on $\A^1_k$ extends to an $F$-isocrystal $\sq \mE_y$ on $\P^1_k$ is extremely restrictive: one can associate a vector bundle $(\sq{\mathbb E}, \nabla_{\sq{\mathbb E}})$ on $\P^1_K$ following a similar construction as in Construction \ref{construction:vector_bundle_with_connections}, but such a connection on a vector bundle on $\P^1$ is necessary trivial.\footnote{In fact, one can show that any $F$-isocrystal on $\P^1_k$ is necessarily constant, in the sense that it comes from the base change of an $F$-isocrystal over the point $\spec k$.} This implies the claim by observing that the restriction of $(\mathbb E, \nabla_{\mathbb E})$ to $\mathbb B_1= \{|y| = 1 \}$ agrees with the restriction from $(\sq{\mathbb E}, \nabla_{\sq{\mathbb E}})$. 
\end{proof}

\subsection{Proof of Theorem \ref{thm:extendability}}\label{subsection:proof of extendability}
 
Let $\D_m$ denote the annulus 
\[\D_m = \{1/p^m \le |z| \le 1\} \] 
inside the punctured disc $\mathbb D^\times$. It has a semistable model $\mathcal D_m$ constructed as follows. First, consider the standard semistable model
\[
\mathcal {A}_j = \spf \mO_K \gr{x_{j-1}, y_{j-1}}/(x_{j-1} \cdot y_{j-1}-\varpi)
\] 
for the annulus 
\[
A_j = \{1/p^j \le |z| \le 1/p^{j-1}\}
\]
where each $x_k$ can be viewed as ``$\varpi^{k} x_0$''. Then, for $1 \le j \le m-1$, we glue $\mathcal A_j$ with $\mathcal A_{j+1}$ along the $p$-adic formal torus $\widehat \G_m$ to form the semistable $p$-adic formal scheme $\mathcal D_m$. More precisely, we glue $\spf \mO_K \gr{(y_{j-1})^{\pm}} \subset \mathcal A_j$ with $\spf \mO_K \gr{x_{j}^{\pm}} \subset \mathcal A_{j+1}$ by identifying $y_{j-1}$ with $(x_j)^{-1}$. The special fiber $\mathcal D_{m,s}$ of $\mathcal D_m$ is thus obtained from gluing the $m$ copies of schemes $k[x,y]/xy$. Notice that $\mathcal D_{m,s}$ consists of two copies of $\A^1_k$ and $(m-1)$ copies of $\P^1_k$, where each $\P^1_k$ intersect the other copies of $\A^1_k$ or $\P^1_k$ precisely at two points (``$0$'' and ``$\infty$''). The $\A^1_k$'s lie at the two ends and each intersects with one of the $\P^1_k$'s at one point. Let $\mathcal D_{m, s}^{\log}$ denote the corresponding log scheme, where the log structure comes from the base change of divisorial log structure on $\mathcal D_{m}$ along the divisor $\varpi = 0$ as usual. 

Now let 
\[
\iota_m:  \mathcal D_{m, s} \ra \mathcal D_{m+1, s}
\]
be the natural (open) immersion of the special fibers. We shall view $\mathbb D^\times = \cup_{m \ge 1} \D_m$ as an infinite union of annuli, which then has a ``semistable integral model'' $\mathcal D_\infty$, whose special fiber consists of one $\A^1_k$ and infinitely many $\P^1_k$'s, glued along the maps $\iota_m$'s. 

Now we are ready to prove Theorem \ref{thm:extendability}. 
\begin{proof}[Proof of Theorem \ref{thm:extendability}]

As above, we write $\mathbb D^\times = \cup_{m \ge 1} \mathbb D_m$. On each $\D_m$, applying the pointwise criterion of Guo-Yang (\cite[Theorem 1.1]{GY}), we know that the local system $\L$ is a semistable local system, since it is crystalline (thus in particular semistable) at all classical points. Let $\mE_{m}^{\log}$ denote the associated log $F$-isocrystal on the log scheme $\mathcal D_{m, s}^{\log}$. Note that, by construction, $\mE_m^{\log}$ agrees with the restriction $\iota_m^* (\mE_{m+1}^{\log})$ of $\mE_{m+1}^{\log}$ to $\mathcal D_{m,s}^{\log}$. Let $(\mathbb E_m, \nabla_{m})$ denote the vector bundle with connection on $\mathbb D_{m}$ obtained from $\mE_{m}^{\log}$ as in Construction \ref{construction:vector_bundle_with_connections}.\footnote{More precisely, it is glued from the vector bundles with connection on each $A_j$ as in Construction \ref{construction:vector_bundle_with_connections}.} By construction, we know that $(\mathbb E_m, \nabla_m)$ agrees with the restriction of $(\mathbb E_{m+1}, \nabla_{m+1})$ to $\D_m$ for each $m$. By \cite[Proposition 3.46]{DLMS2}, we know that 
\[
(\mathbb E_m, \nabla_m) = (\D_{\mathrm{dR}}(\mathbb{L}) |_{\D_m}, \nabla_{\L}) = (\D_{\mathrm{dR}, \log}(\mathbb{L}) |_{\D_m}, \nabla_{\L}).   
\]
By the assumption that $\L$ is crystalline at all classical points on $\D^\times$ and Theorem \ref{theorem:log_crystal_2}, we know that each log $F$-isocrystal $\mE_{m}^{\log}$ descends to an $F$-isocrystal $\mE_m$ on $\mathcal D_{m, s}$. Now, applying Lemma \ref{lemma:isocrystal_from_P1} repeatedly, we conclude that $(\D_{\mathrm{dR}, \log}(\mathbb{L}), \nabla_{\L})$ restricts to a trivial bundle with trivial connection over each ``thin annulus'' $\{|z| = 1/p^m\}$, for all $m \ge 1$. This in turn implies that the restriction of  $(\D_{\mathrm{dR}, \log}(\mathbb{L}), \nabla_{\L})$ to the smaller punctured disc $\{0 < |z| \le 1/p\}$ is trivial by the following sublemma. 

\vspace*{0.2cm}

\noindent{\textit{Sublemma}.}  Suppose that a vector bundle with (automatically integrable) connection over a closed annulus $A_{\alpha, \beta} = \{\alpha \le |z| \le \beta\}$ is trivial when restricted to its two boundary ``thin annulus'' $\partial_{\alpha} = \{|z| = \alpha\}$ and $\partial_{\beta} = \{|z| = \beta\}$, then it is trivial over the entire $A_{\alpha, \beta}$. 

\vspace*{0.2cm}

\noindent \textit{Proof of the sublemma}. By assumption, we can glue the vector bundle with connection over $A_{\alpha, \beta}$ to a trivial bundle with trivial connection over a disc with boundary equal to $\partial_\alpha$ (resp. $\partial_\beta$). This way, we get a vector bundle with connection over $\P^1_K$, which has to be trivial. This proves the sublemma.

\vspace*{0.2cm}

Consequently, the connection $(\D_{\mathrm{dR}, \log}(\mathbb{L}), \nabla_{\L})$ has zero residue around the puncture and hence $\L$ extends to the entire disc $\D$ by Theorem \ref{prop:residue vs monodromy}. 

Finally, we show that the extended local system is necessarily crystalline. Indeed, by \cite[Theorem 1.1]{GY}, the restriction of $\L$ on the thin annulus $\{|z|=1\}$ is a crystalline local system. In particular, $\L$ is crystalline at the Gaussian point of $\D$ (which is a type I\!I point). This implies the crystallinity of the extended local system by purity (\cite[Theorem 1.2]{Moon}).
\end{proof}

\subsection{Proof of extendability of pointwise crystalline $p$-adic local systems}
 
Finally, we deduce Theorem \ref{mainthm:extend_across_ncd} from Theorem \ref{thm:extendability}. 
Let us state the result again for reader's convenience. As in \S \ref{subsection:log RH}, we let $Y$ be a connected smooth  rigid analytic space  over $K$. Let $D\subset Y$ be a reduced normal crossing divisor, and let $U = Y -  D$ be the complement of $D$.

\begin{theorem} 
Let $\L$ be an \'etale $\Z_p$-local system on $U$.    Suppose that $\L$ is crystalline at all classical points on $U$. Then $\L$ extends to an \'etale $\Z_p$-local system on $Y$. The extension is necessarily unique. Moreover, if $Y$ is quasi-compact, then after some finite extension $L$ of the base field $K$, $\mathbb{L}$ is crystalline at all classical points on $Y_L$. 
\end{theorem}

\begin{proof} 
Let us first note that, analytic locally on $Y$, up to replacing $K$ by a finite extension, $Y$ and $D$ are of the form $S\times \mathbb{D}^r$ and $S\times \{T_1\cdots T_r=0\}$ where $S$ is a connected smooth rigid analytic space over $K$ and the closed immersion $D\hookrightarrow Y$ is given by the pullback of the natural closed immersion 
\begin{equation} \label{eq:number_of_branches_in_ncd}
\{T_1\cdots T_r=0\}\hookrightarrow \mathbb{D}^r. 
\end{equation}  
 We shall prove the theorem by induction on $r$ in (\ref{eq:number_of_branches_in_ncd}), which is the number of irreducible components of $D$. We start with the case when $r=1$. Since the question is analytic local (after some finite extension of $K$), we may assume that $D\hookrightarrow Y$ is the closed immersion 
 \[ S\times \{T=0\}\hookrightarrow S\times \mathbb{D} 
 \] for some smooth rigid analytic space $S$, and hence $U=S\times \mathbb{D}^{\times}$. For every classical point $s\in S$, taking pullback along $s\hookrightarrow S$, we obtain a local system $\mathbb{L}|_{\{s\}\times \mathbb{D}^{\times}}$ which is crystalline at all classical points. By the proof of Theorem \ref{thm:extendability},
 the log connection $D_{\mathrm{dR}, \log}\left(\mathbb{L}|_{\{s\}\times \mathbb{D}^{\times}}\right)$ has zero residue around the puncture $\{s\}\times\{T=0\}$. Since this is true for all classical points $s\in S$, it follows that the log connection $D_{\mathrm{dR}, \log}(\mathbb{L})$ has zero residue along the boudnary divisor $D=S\times\{T=0\}$. This in turn implies that $\mathbb{L}$ has trivial geometric monodromy around $D$. In particular, $\mathbb{L}$ extends to $Y=S\times \mathbb{D}$. Moreover, by the proof of Theorem \ref{thm:extendability}, we know that $\mathbb{L}$ is crystalline everywhere on $D$ (and thus everywhere on $Y$).

Next, we prove the case for $r$  assuming that the statement holds for $\leq r-1$. Once again, after some finite extension of $K$, we may assume that $D\hookrightarrow Y$ is the closed immersion 
\[ S\times \{T_1\cdots T_r=0\}\hookrightarrow S\times \mathbb{D}^r \]  obtained from the pullback of the natural closed immersion $\{T_1\cdots T_r=0\}\hookrightarrow \mathbb{D}^r$. Consider \[Y^{(r)}:= S\times\left(\mathbb{D}^r\backslash \{T_r=0\}\right)\subset X.\]
Then we have a closed immersion $U\hookrightarrow Y^{(r)}$, given by  
\[\left(S\times \mathbb{D}^{\times}\right)\times \{T_1\cdots T_{r-1}=0\}\hookrightarrow\left(S\times \mathbb{D}^{\times}\right)\times\mathbb{D}^{r-1}\]
obtained from the pullback of the natural closed immersion $\{T_1\cdots T_{r-1}=0\}\hookrightarrow \mathbb{D}^{r-1}$. By the induction hypothesis, $\mathbb{L}$ extends (uniquely) to a local system on $Y^{(r)}$ and it is crystalline everywhere on $Y^{(r)}$.

Now, the inclusion $Y^{(r)}\rightarrow Y$ is simply \[\left(S\times \mathbb{D}^{r-1}\right)\times \mathbb{D}^{\times}\hookrightarrow \left(S\times \mathbb{D}^{r-1}\right)\times \mathbb{D}\]
obtained from the pullback of the inclusion $\mathbb{D}^{\times}\hookrightarrow\mathbb{D}$. Applying the case for $r=1$ one more time and we obtain the desired claim. This finishes the proof of the theorem. 
 \end{proof}
 

\newpage 

\bibliographystyle{amsalpha}
\bibliography{library}

\vspace{1cm}

\end{document}